\documentclass[12pt]{article}
\usepackage{amsthm,amsfonts,amssymb,amsmath}
\usepackage{stmaryrd}
\usepackage{cite,hyperref}
\usepackage{epsfig}
\usepackage{url}
\usepackage{xcolor,tikz,dsfont}
\usetikzlibrary{positioning, arrows.meta,calc}
\usetikzlibrary{decorations.pathreplacing,decorations.markings}
\usetikzlibrary{matrix}
\usepackage{multicol,graphicx}
\newcommand{\cev}[1]{\reflectbox{\ensuremath{\vec{\reflectbox{\ensuremath{#1}}}}}}
\usepackage{fullpage}
\usepackage{bbm}
\usepackage{pdflscape}
\usetikzlibrary{decorations}
\usepackage{svg}
\tikzset{
    vertex/.style={circle, draw, fill=black, inner sep=1.5pt},
    openvertex/.style={circle, draw, fill=white, inner sep=1.5pt},
    lbl/.style={draw=none, fill=none, circle=none, inner sep=1pt},
    arc/.style={-{Stealth[length=2.2mm,width=1.6mm]}, line width=0.8pt},
}

\tikzset{
    proofbox/.style={rounded corners=7pt, draw=black!45, fill=black!3, line width=0.5pt},
    addarc/.style={arc, red!70!black, line width=1.0pt},
    marked/.style={openvertex, line width=1.1pt},
    faint/.style={draw=black!40},
}

\usepackage[shortlabels]{enumitem}

\numberwithin{equation}{section}

\makeatletter
\let\c@figure\c@equation
\let\c@table\c@equation

\makeatother

\newtheorem{thm}[equation]{Theorem}

\newtheorem{lem}[equation]{Lemma}
\newtheorem{prop}[equation]{Proposition}

\newtheorem{constr}[equation]{Construction}

\theoremstyle{definition}
\newtheorem{defn}[equation]{Definition}

\newtheorem{obs}[equation]{Observation}

\newtheorem*{ai}{AI Declaration}

\theoremstyle{remark}

\newtheorem{rem}[equation]{Remark}

\title{On Tournament Anti-Sidorenko Orientations of Trees}
\author{Hao Chen\thanks{School of Mathematical Sciences, University of Science and Technology of China, Hefei, Anhui, 230026, China. E-mail: {\tt mathsch@mail.ustc.edu.cn}. This work was completed while the first author was visiting the University of Victoria. Research supported by National Key Research and Development Program of China 2023YFA1010201, National Natural Science Foundation of China grant 12125106 and China Scholarship Council No. 202406340066.} \and Felix Christian Clemen\thanks{Department of Mathematics and Statistics, University of Victoria, Victoria, B.C., Canada.}$\text{ }^{,}$\thanks{E-mail: {\tt fclemen@uvic.ca}. Research supported by PIMS Postdoctoral Fellowship PIMS-20260513-PDF.} \and Jonathan A. Noel\footnotemark[2]$\text{ }^{,}$\thanks{E-mail: {\tt noelj@uvic.ca}. Research supported by NSERC Discovery Grant RGPIN-2021-02460.}}

\DeclareTextCompositeCommand{\v}{OT1}{l}{l\nobreak\hspace{-.1em}'}
\DeclareTextCompositeCommand{\v}{OT1}{t}{t\nobreak\hspace{-.1em}'\nobreak\hspace{-.15em}}

\DeclareMathOperator{\rng}{rng}
\DeclareMathOperator{\cov}{cov}

\begin{document}

\maketitle

\begin{abstract}
An oriented graph $\vec{H}$ is said to be \emph{tournament anti-Sidorenko} if the homomorphism density of $\vec{H}$ in any tournament $\vec{T}$ is bounded above by the homomorphism density of $\vec{H}$ in a large uniformly random tournament. We prove the following:
\begin{enumerate}[(1)]
    \item Every oriented path with at least three arcs and exactly one non-leaf source or sink vertex is tournament anti-Sidorenko. 
    \item An oriented path is tournament anti-Sidorenko if the distance between any leaf vertex and any source or sink vertex is at least two and the distance between any pair of non-leaf source or sink vertices is a multiple of four.
    \item Every spider with exactly three legs admits a tournament anti-Sidorenko orientation.
\end{enumerate}
The first result proves a conjecture posed by He, Mani, Nie, Tung and Wei~\cite{HeManiNieTungWei25+}. The third resolves a problem from the same paper, in fact establishing a substantially more general statement, and provides evidence in support of a conjecture of Fox, Himwich, Mani and Zhou~\cite{FoxHimwichManiZhou24+}. The second yields the first family of tournament anti-Sidorenko oriented paths which is exponentially large with respect to the number of arcs. 
\end{abstract}

\section{Introduction}
\label{sec:intro}

\setlist{leftmargin=3\parindent,labelindent=3\parindent}
\setlist[enumerate]{%
  leftmargin=3\parindent,%
  align=left,%
  labelwidth=3\parindent,%
  labelsep=0pt%
}
\setlist[enumerate,1]{%
  label={\normalfont (\thesection.\arabic{equation})}, ref={\normalfont \thesection.\arabic{equation}},
  resume%
}

A central theme in extremal combinatorics is to identify small ``patterns'' whose frequency in a large combinatorial object is extremized by a random object of the same type. An archetypal example is Sidorenko's Conjecture~\cite{Sidorenko93,ConlonFoxSudakov10,ConlonLee21,Hatami10} that, for any bipartite graph $H$, the homomorphism density of $H$ in a graph $G$ with edge density $p=2|E(G)|/|V(G)|^2$ is at least the expected density in a large binomial random graph of edge density $p$. 

In this paper, we focus on instances of this phenomenon in the setting of tournaments. An \emph{oriented graph} is a pair $\vec{H}=(V(\vec{H}),A(\vec{H}))$ where $V(\vec{H})$ is a set of \emph{vertices} and $A(\vec{H})\subseteq V(\vec{H})\times V(\vec{H})$ is a set of \emph{arcs} containing no self-loops and at most one arc between any pair of distinct vertices. A \emph{tournament} is an oriented graph in which there is exactly one arc between every pair of distinct vertices. A \emph{homomorphism} from an oriented graph $\vec{H}$ to an oriented graph $\vec{G}$ is a function $f:V(\vec{H})\to V(\vec{G})$ with the property that $(f(u),f(v))\in A(\vec{G})$ for all $(u,v)\in A(\vec{H})$. We write $f:\vec{H}\to \vec{G}$ to indicate that $f$ is a homomorphism from $\vec{H}$ to $\vec{G}$. The number of homomorphisms from $\vec{H}$ to $\vec{G}$ is denoted $\hom(\vec{H},\vec{G})$ and the \emph{homomorphism density} of $\vec{H}$ in $\vec{G}$ is
\[t(\vec{H},\vec{G}):=\frac{\hom(\vec{H},\vec{G})}{v(\vec{G})^{v(\vec{H})}}\]
where $v(\vec{F}):=|V(\vec{F})|$ for any oriented graph $\vec{F}$. Inspired by Sidorenko's Conjecture, there are several recent papers on classifying oriented graphs $\vec{H}$ such that, among all large tournaments $\vec{T}$, the quantity $t(\vec{H},\vec{T})$ is asymptotically maximized or minimized by taking $\vec{T}$ to be a uniformly random tournament~\cite{BucicLongShapiraSudakov21,Hancock+23,CoreglianoParenteSato19,CoreglianoRazborov17,SahSawhneyZhao23,GrzesikKralLovaszVolec23,HeManiNieTungWei25+,NoelRanganathanSimbaqueba26,FoxHimwichManiZhou24+,KalyanasundaramShapira13,ZhaoZhou20,ChenClemenNoelSharfenberg26+,Kral+26+}; see also~\cite{Griffiths13,BitontiHoganNoelTsarev26+,FoxHimwichManiZhou25}. The following definitions make this more precise. 

\begin{defn}
An oriented graph $\vec{H}$ is \emph{tournament anti-Sidorenko} if $t(\vec{H},\vec{T})\leq (1/2)^{a(\vec{H})}$ for every tournament $\vec{T}$, where $a(\vec{H}):=|A(\vec{H})|$.
\end{defn}

\begin{defn}
An oriented graph $\vec{H}$ is \emph{tournament Sidorenko} if $t(\vec{H},\vec{T})\geq (1-o(1))(1/2)^{a(\vec{H})}$ for every tournament $\vec{T}$, where the $o(1)$ term tends to zero as $v(\vec{T})$ tends to infinity.
\end{defn}

In this paper, we solve several open problems on tournament anti-Sidorenko orientations of paths and spiders. Given an oriented path $\vec{P}$, a \emph{block} of $\vec{P}$ is a maximal subpath of $\vec{P}$ in which all arcs are oriented in the same direction; see Figure~\ref{fig:blocks}. A block of $\vec{P}$ is \emph{internal} if it does not contain a leaf of $\vec{P}$. Sah, Sawhney and Zhao~\cite[Theorem~1]{SahSawhneyZhao23} proved that every oriented path with one block (i.e. every \emph{directed path}) is tournament anti-Sidorenko. He, Mani, Nie, Tung and Wei~\cite[Conjecture~2.3]{HeManiNieTungWei25+} conjectured that every oriented path with exactly two blocks and at least three arcs is tournament anti-Sidorenko. We prove this. 

\begin{figure}[htbp]
\centering
\begin{tikzpicture}[scale=1]

\node[openvertex] (v1) at (0,0) {};
\node[vertex]     (v2) at (1,0) {};
\node[openvertex] (v3) at (2,0) {};
\node[vertex]     (v4) at (3,0) {};
\node[vertex]     (v5) at (4,0) {};
\node[vertex]     (v6) at (5,0) {};
\node[openvertex] (v7) at (6,0) {};
\node[vertex]     (v8) at (7,0) {};
\node[vertex]     (v9) at (8,0) {};
\node[openvertex] (v10) at (9,0) {};

\node[lbl, above=0pt] at (v1) {$v_1$};
\node[lbl, above=0pt] at (v2) {$v_2$};
\node[lbl, above=0pt] at (v3) {$v_3$};
\node[lbl, above=0pt] at (v4) {$v_4$};
\node[lbl, above=0pt] at (v5) {$v_5$};
\node[lbl, above=0pt] at (v6) {$v_6$};
\node[lbl, above=0pt] at (v7) {$v_7$};
\node[lbl, above=0pt] at (v8) {$v_8$};
\node[lbl, above=0pt] at (v9) {$v_9$};
\node[lbl, above=0pt] at (v10) {$v_{10}$};

\draw[arc] (v1) -- (v2);
\draw[arc] (v2) -- (v3);

\draw[arc] (v4) -- (v3);
\draw[arc] (v5) -- (v4);
\draw[arc] (v6) -- (v5);
\draw[arc] (v7) -- (v6);

\draw[arc] (v7) -- (v8);
\draw[arc] (v8) -- (v9);
\draw[arc] (v9) -- (v10);

\end{tikzpicture}
\caption{
An oriented path with three blocks of lengths $2$, $4$, and $3$. White nodes represent the source and sink vertices, which are the ends of the blocks.
}
\label{fig:blocks}
\end{figure}
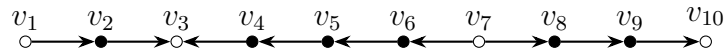

\begin{thm}
\label{th:2blocks}
Every oriented path with exactly two blocks and at least three arcs is tournament anti-Sidorenko.
\end{thm}

We also obtain a general result for paths with an arbitrary number of blocks. 

\begin{thm}
\label{th:0mod4}
If $\vec{P}$ is an oriented path in which every block of $\vec{P}$ has length at least two and every internal block has length divisible by four, then $\vec{P}$ is tournament anti-Sidorenko. 
\end{thm}

To our knowledge, the above theorem provides the first family of tournament anti-Sidorenko oriented paths of exponential size in the number of arcs. He, Mani, Nie, Tung and Wei~\cite[Conjecture~1.9]{HeManiNieTungWei25+} conjectured that a $1/2-o(1)$ proportion of all oriented paths with $k$ arcs are tournament anti-Sidorenko. While we disprove this in a forthcoming paper with Sharfenberg~\cite{ChenClemenNoelSharfenberg26+}, Theorem~\ref{th:0mod4} can be seen as evidence supporting the weaker statement that a positive proportion of all oriented paths with $k$ arcs are tournament anti-Sidorenko. 

An intriguing conjecture of Fox, Himwich, Mani and Zhou~\cite[Conjecture~1.12]{FoxHimwichManiZhou24+} is that every undirected tree admits a tournament anti-Sidorenko orientation. It is known to hold for a handful of trees, including paths~\cite[Theorem~1]{SahSawhneyZhao23}, trees with at most one vertex of even degree~\cite[Proposition~1.13]{FoxHimwichManiZhou24+}, and caterpillars~\cite[Theorem~1.4]{HeManiNieTungWei25+}. For $a_1,\dots,a_k\geq1$, define the \emph{$(a_1,\dots,a_k)$-spider}, or simply a \emph{$k$-spider}, to be the undirected graph obtained by taking disjoint paths of lengths $a_1,\dots,a_k$ and identifying all paths on one of their endpoints; see Figure~\ref{fig:spiders}. He, Mani, Nie, Tung and Wei~\cite[Problem~1.6]{HeManiNieTungWei25+} identified the $(2,3,4)$-spider as being an interesting small case for which~\cite[Conjecture~1.12]{FoxHimwichManiZhou24+} is open. We not only settle this case, but we obtain tournament anti-Sidorenko orientations of $3$-spiders in full generality.\footnote{Problem~1.6 from~\cite{HeManiNieTungWei25+} was also solved independently by another group; see the remarks in Section~\ref{sec:concl}.}

\begin{figure}[htbp]
\centering

\begin{minipage}{0.42\textwidth}
\centering
\begin{tikzpicture}[x=1.05cm, y=0.7cm, 
    vertex/.style={circle, draw, fill=black, inner sep=1.5pt}
]

\node[vertex] (c) at (0,0) {};

\node[vertex] (a1) at (0,-1.0) {};
\node[vertex] (a2) at (0,-2.1) {};

\node[vertex] (b1) at (-0.9,-0.75) {};
\node[vertex] (b2) at (-1.5,-1.75) {};
\node[vertex] (b3) at (-1.9,-2.9) {};

\node[vertex] (d1) at (0.9,-0.75) {};
\node[vertex] (d2) at (1.5,-1.75) {};
\node[vertex] (d3) at (1.9,-2.9) {};
\node[vertex] (d4) at (2.15,-4.05) {};

\draw (c) -- (a1) -- (a2);
\draw (c) -- (b1) -- (b2) -- (b3);
\draw (c) -- (d1) -- (d2) -- (d3) -- (d4);

\end{tikzpicture}

\end{minipage}
\hfill
\begin{minipage}{0.50\textwidth}
\centering
\begin{tikzpicture}[
    scale=1,
    line cap=round,
    line join=round,
    v/.style={circle, draw, fill=black, inner sep=1.3pt}
]

\definecolor{spiderbrown}{RGB}{155,85,35}
\definecolor{spiderlight}{RGB}{185,115,60}
\definecolor{eyeblue}{RGB}{40,140,220}

\coordinate (c) at (0,-0.03);


\node[v] (L1) at (-1.1,0.05) {};
\draw[spiderbrown, line width=1.2pt]
  (c) -- (-0.55,0.10) -- (L1);

\node[v] (L3a) at (-0.9,-0.55) {};
\node[v] (L3b) at (-1.6,-1.0) {};
\draw[spiderbrown, line width=1.2pt]
  (c) -- (-0.45,-0.10) -- (L3a) -- (L3b);

\node[v] (L5a) at (-1.15,0.25) {};
\node[v] (L5b) at (-1.9,0.10) {};
\node[v] (L5c) at (-2.5,-0.45) {};
\draw[spiderbrown, line width=1.2pt]
  (c) -- (-0.5,0.15) -- (L5a) -- (L5b) -- (L5c);

\node[v] (L7a) at (-1.15,-0.25) {};
\node[v] (L7b) at (-1.95,-0.55) {};
\node[v] (L7c) at (-2.6,-1.15) {};
\node[v] (L7d) at (-3.0,-1.9) {};
\draw[spiderbrown, line width=1.2pt]
  (c) -- (-0.5,-0.05) -- (L7a) -- (L7b) -- (L7c) -- (L7d);


\node[v] (R1) at (1.1,0.05) {};
\draw[spiderbrown, line width=1.2pt]
  (c) -- (0.55,0.10) -- (R1);

\node[v] (R3a) at (0.9,-0.55) {};
\node[v] (R3b) at (1.6,-1.0) {};
\draw[spiderbrown, line width=1.2pt]
  (c) -- (0.45,-0.10) -- (R3a) -- (R3b);

\node[v] (R5a) at (1.15,0.25) {};
\node[v] (R5b) at (1.9,0.10) {};
\node[v] (R5c) at (2.5,-0.45) {};
\draw[spiderbrown, line width=1.2pt]
  (c) -- (0.5,0.15) -- (R5a) -- (R5b) -- (R5c);

\node[v] (R7a) at (1.15,-0.25) {};
\node[v] (R7b) at (1.95,-0.55) {};
\node[v] (R7c) at (2.6,-1.15) {};
\node[v] (R7d) at (3.0,-1.9) {};
\draw[spiderbrown, line width=1.2pt]
  (c) -- (0.5,-0.05) -- (R7a) -- (R7b) -- (R7c) -- (R7d);


\filldraw[fill=spiderbrown, draw=black] (0,0.33) ellipse (0.22 and 0.28);
\filldraw[fill=spiderlight, draw=black] (0,-0.03) ellipse (0.17 and 0.16);

\draw[line width=0.45pt]
  (-0.10,0.060) .. controls (-0.075,0.075) and (-0.045,0.075) .. (-0.02,0.060);
\draw[line width=0.45pt]
  (0.02,0.060) .. controls (0.045,0.075) and (0.075,0.075) .. (0.10,0.060);

\filldraw[fill=white, draw=black] (-0.060,0.010) circle (0.036);
\filldraw[fill=white, draw=black] ( 0.060,0.010) circle (0.036);
\filldraw[fill=eyeblue, draw=black] (-0.060,0.000) circle (0.020);
\filldraw[fill=eyeblue, draw=black] ( 0.060,0.000) circle (0.020);

\fill[white] (-0.072,0.022) circle (0.007);
\fill[white] ( 0.048,0.022) circle (0.007);

\draw[line width=0.55pt]
  (-0.060,-0.085) .. controls (0,-0.125) .. (0.060,-0.085);
\draw[line width=0.35pt]
  (-0.022,-0.102) -- (0.022,-0.102);

\end{tikzpicture}
\end{minipage}

\caption{The $(2,3,4)$-spider and the $(1,1,2,2,3,3,4,4)$-spider.}
\label{fig:spiders}
\end{figure}
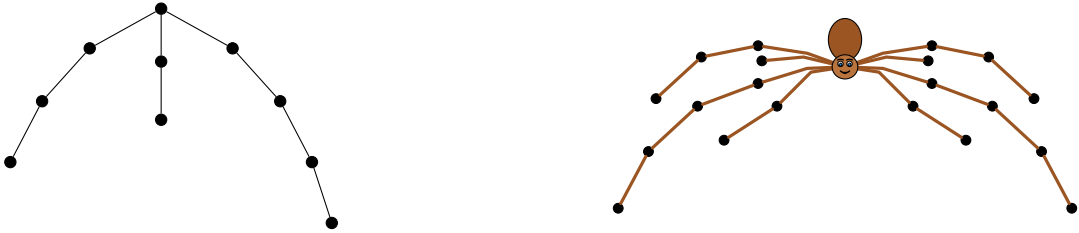

\begin{thm}
\label{th:3spider}
Every $3$-spider admits a tournament anti-Sidorenko orientation.  
\end{thm}

\subsection{Outline of the Paper}

The proofs of our main results are built up from many auxiliary statements with complex relationships between one another. The purpose of this subsection is to provide a road map for the rest of the paper and to elucidate the main logical dependencies among the results.  We use the following convention. The three main results stated in the introduction are referred to as \emph{theorems}. Results establishing that a particular oriented graph, or each member of a small explicit family of oriented graphs, is tournament Sidorenko or tournament anti-Sidorenko are stated as \emph{propositions}. General tools and reduction principles are stated as \emph{lemmas}, and results establishing the existence of specific structures are stated as \emph{constructions}. The reader can compare, for example, Theorem~\ref{th:2blocks}, Proposition~\ref{prop:P12}, Lemma~\ref{lem:reverseAllArcs} and Construction~\ref{constr:D4*}. Table~\ref{tab:dependencies} summarizes the propositions and theorems proved in the paper, the sections in which these results are proved, and the results that are used in their proofs.  

\begin{table}[htbp]
\centering
\small
\begin{tabular}{l|l|l|l|l|l}
\textbf{Thm/Prop} & \textbf{Oriented graph(s)} & \textbf{Section} & \textbf{TS} & \textbf{TAS} & \textbf{Main inputs}\\
\hline
\ref{prop:P12} & $\vec P_{1,2}$ &\ref{sec:tricks}& $\checkmark$ & $\checkmark$ & \ref{lem:edgeFlip},~\ref{lem:union}\\[1mm]
\ref{prop:P1211} & $\vec P_{1,2,1,1}$ &\ref{sec:tricks}& $\checkmark$ &  & \ref{lem:edgeFlip},~\ref{lem:union}\\[1mm]
\ref{prop:P14} & $\vec P_{1,4}$ &\ref{sec:tricks}&  & $\checkmark$ & \ref{lem:asympGoodEnough},~\ref{lem:union},~\ref{prop:P12},~\ref{prop:P1211}\\[1mm]
\ref{prop:P13} & $\vec P_{1,3}$ &\ref{sec:tricks}&  & $\checkmark$ & \ref{lem:inOut}\\[1mm]
\ref{prop:forest} & $\vec P_2\sqcup\vec P_{1,1}$ &\ref{sec:tricks}&  & $\checkmark$ & \ref{lem:edgeFlip},~\ref{lem:union}\\[1mm]
\ref{prop:P1331} & $\vec P_{1,3,3,1}$ &\ref{sec:2blocks}&  & $\checkmark$ & \ref{lem:union},~\ref{obs:1/2diag}--\ref{lem:pathU}\\[1mm]
\ref{prop:P22} & $\vec P_{2,2}$&\ref{sec:2blocks} &  & $\checkmark$ & \ref{lem:union},~\ref{lem:polynomial},~\ref{lem:edgeFlipU},~\ref{lem:pathU}\\[1mm]
\ref{th:0mod4} & See Theorem~\ref{th:0mod4}& \ref{sec:0mod4} &  & $\checkmark$ & \ref{lem:entropyMain},~\ref{constr:D4*},~\ref{constr:Dk*}\\[1mm]
\ref{th:2blocks} & See Theorem~\ref{th:2blocks}& \ref{sec:0mod4}&  & $\checkmark$ & \ref{th:0mod4},~\ref{prop:P12},~\ref{prop:P14},~\ref{prop:P13},~\ref{lem:reduction}\\[1mm]
\ref{th:3spider} &See Theorem~\ref{th:3spider}&\ref{sec:spiders} &  & $\checkmark$ & \ref{lem:entropyMain},~\ref{lem:cover1/2}
\end{tabular}
\caption{The main examples of Sidorenko and anti-Sidorenko oriented graphs exhibited in the paper, the section containing the proof, and the results that the proof depends on. Here TS means tournament Sidorenko and TAS means tournament anti-Sidorenko. }
\label{tab:dependencies}
\end{table}

We now describe the sections in more detail.  Section~\ref{sec:tricks} establishes several elementary properties of homomorphism densities in tournaments.  These tools are used to establish the tournament Sidorenko and tournament anti-Sidorenko properties for several small examples which are used later.  Section~\ref{sec:2blocks} uses an algebraic expansion technique, together with estimates for paths in skew-symmetric matrices from~\cite{Grzesik+23}, to obtain two useful examples of tournament anti-Sidorenko oriented paths and to reduce the proof of Theorem~\ref{th:2blocks} to the case that both blocks have equal length.

After Section~\ref{sec:2blocks}, the paper shifts to focus on developing and applying an entropy-based approach to the problem of classifying tournament anti-Sidorenko oriented trees. Specifically, in Section~\ref{sec:sandwich}, we prove Lemma~\ref{lem:entropyMain} which shows that the existence of a certain type of certificate---which we call an ``$\vec{H}$-sandwich''---is sufficient for an oriented forest $\vec{H}$ to be tournament anti-Sidorenko. To illustrate the approach, we provide a sandwich certificate to re-prove the theorem of Sah, Sawhney and Zhao~\cite{SahSawhneyZhao23} that every directed path is tournament anti-Sidorenko. In Section~\ref{sec:0mod4}, we construct the sandwich certificates for oriented paths needed to prove Theorem~\ref{th:0mod4}. In the same section, we derive Theorem~\ref{th:2blocks} from Theorem~\ref{th:0mod4} and the reduction proved in Section~\ref{sec:2blocks}.  Finally, Section~\ref{sec:spiders} establishes sandwich certificates for an orientation of every $3$-spider, except for those which contain a leg of length 1, which is covered by a result of~\cite{HeManiNieTungWei25+}, thereby proving Theorem~\ref{th:3spider}. Many of the sandwich certificates in the paper are most easily verified by analyzing a diagram. To make the paper easier to navigate, we collect all of the diagrams for the results in Sections~\ref{sec:0mod4} and~\ref{sec:spiders} in Appendix~\ref{app:gadget-atlas}. 

We conclude the paper in Section~\ref{sec:concl} by mentioning some related independent work.

\section{Basic Tricks}
\label{sec:tricks}

Throughout the paper, we assume that the vertices of any oriented path with $k$ arcs are labeled $v_0,\dots,v_k$, which we view as being arranged from left to right. Given $\ell_1,\ell_2,\dots,\ell_m\geq1$, let $\vec{P}_{\ell_1,\dots,\ell_m}$ be the oriented path with $k=\sum_{i=1}^m\ell_i$ arcs in which, for each $1\leq i\leq m$, the $i$th block has length $\ell_i$ and the first block is directed from left to right. In particular, $\vec{P}_k$ is a path with one block of length $k$. For convenience, we let $\vec{P}_0$ be the path with only one vertex and $\vec{P}_{\ell_1,\dots,\ell_m,0,\dots,0}:=\vec{P}_{\ell_1,\dots,\ell_m}$. 

The goals of this section are to build up some basic results on homomorphism densities in tournaments which will be used throughout the paper, and to use them to establish Theorem~\ref{th:2blocks} for several small cases---namely $\vec{P}_{1,2},\vec{P}_{1,3}$ and $\vec{P}_{1,4}$. At the end of the section, we will also show that $\vec{P}_{2}\sqcup\vec{P}_{1,1}$ is tournament anti-Sidorenko, a fact which will be useful to us later in the paper.

Given an oriented graph $\vec{H}$, let $\cev{H}$ be the oriented graph obtained from $\vec{H}$ by reversing the direction of all arcs. The next lemma is quite obvious, but can be useful as it implies that $\vec{H}$ is tournament anti-Sidorenko if and only if $\cev{H}$ is, and so we can freely exchange $\vec{H}$ and $\cev{H}$ whenever it is convenient to do so.

\begin{lem}
\label{lem:reverseAllArcs}
For any oriented graphs $\vec{H}$ and $\vec{G}$, it holds that $\hom(\vec{H},\vec{G})=\hom(\cev{H},\cev{G})$.
\end{lem}

\begin{proof}
For any function $\varphi:V(\vec{H})\to V(\vec{G})$ and vertices $u,v\in V(\vec{H})$, we have that $(u,v)$ is an arc of $\vec{H}$ if and only if $(v,u)$ is an arc of $\cev{H}$ and that $(\varphi(u),\varphi(v))$ is an arc of $\vec{G}$ if and only if $(\varphi(v),\varphi(u))$ is an arc of $\cev{G}$. Therefore, $\varphi$ is a homomorphism from $\vec{H}$ to $\vec{G}$ if and only if it is a homomorphism from $\cev{H}$ to $\cev{G}$.
\end{proof}

The asymmetry between the definition of tournament Sidorenko, which includes a $o(1)$ term, and tournament anti-Sidorenko, which does not, may seem somewhat awkward. However, as we show next, both notions can be written with a $o(1)$ term. 

\begin{lem}
\label{lem:asympGoodEnough}
If $\vec{H}$ is an oriented graph such that $t(\vec{H},\vec{T})\leq (1+o(1))(1/2)^{a(\vec{H})}$ for every tournament $\vec{T}$, then $\vec{H}$ is tournament anti-Sidorenko.
\end{lem}

\begin{proof}
We prove the contrapositive. Suppose that there is a tournament $\vec{T}$ and $\varepsilon>0$ such that $t(\vec{H},\vec{T})=(1+\varepsilon)(1/2)^{a(\vec{H})}$. For each $n\geq1$, let $\vec{T}_n$ be a tournament obtained from $\vec{T}$ by replacing each vertex $u$ with a set $S_u$ of $n$ vertices and adding all arcs from set $S_u$ to set $S_v$ if $(u,v)\in A(\vec{T})$ and adding arcs inside the sets $S_u$ arbitrarily. Then $t(\vec{H},\vec{T}_n)\geq t(\vec{H},\vec{T}) = (1+\varepsilon)(1/2)^{a(\vec{H})}$ by construction and $v(\vec{T}_n)=n\cdot v(\vec{T})$. Therefore, it is not the case that $t(\vec{H},\vec{T}_n)\leq (1+o(1))(1/2)^{a(\vec{H})}$, and so the proof is complete. 
\end{proof}

Next, we obtain a useful property of homomorphisms into tournaments.

\begin{lem}
\label{lem:edgeFlip}
Let $\vec{F}$ be an oriented graph and let $u$ and $v$ be distinct vertices of $\vec{F}$ such that neither of the arcs $(u,v)$ nor $(v,u)$ are present in $\vec{F}$. Let $\vec{F}_{u,v}$ and $\vec{F}_{v,u}$ be the oriented graphs obtained from $\vec{F}$ by adding the arcs $(u,v)$ and $(v,u)$, respectively. Then, for every tournament $\vec{T}$,
\[\hom(\vec{F},\vec{T}) = \hom(\vec{F}_{u,v},\vec{T})+\hom(\vec{F}_{v,u},\vec{T}) + O\left(v(\vec{T})^{v(\vec{F})-1}\right).\]
\end{lem}

\begin{proof}
For each homomorphism $\varphi:\vec{F}\to\vec{T}$ such that $\varphi(u)\neq \varphi(v)$, exactly one of $(\varphi(u),\varphi(v))$ or $(\varphi(v),\varphi(u))$ must be an arc of $\vec{T}$, since $\vec{T}$ is a tournament. Thus, every such homomorphism corresponds to a homomorphism from $\vec{F}_{u,v}$ or $\vec{F}_{v,u}$ to $\vec{T}$, but not both. The number of homomorphisms from $\vec{F}$ to $\vec{T}$ which map $u$ and $v$ to the same vertex is $O\left(v(\vec{T})^{v(\vec{F})-1}\right)$, and so the proof is complete. 
\end{proof}

The next lemma relates the homomorphism density of a disconnected oriented graph to that of its components. Given two oriented graphs $\vec{H}$ and $\vec{F}$, we let $\vec{H}\sqcup \vec{F}$ denote their vertex-disjoint union.

\begin{lem}
\label{lem:union}
For any oriented graphs $\vec{H},\vec{F}$ and $\vec{G}$, we have $t(\vec{H}\sqcup \vec{F},\vec{G})=t(\vec{H},\vec{G})t(\vec{F},\vec{G})$.
\end{lem}

\begin{proof}
A function $\varphi:V(\vec{H}\sqcup \vec{F})\to V(\vec{G})$ is a homomorphism from $\vec{H}\sqcup \vec{F}$ to $\vec{G}$ if and only if its restrictions to $V(\vec{H})$ and $V(\vec{F})$ are homomorphisms from $\vec{H}$ to $\vec{G}$ and $\vec{F}$ to $\vec{G}$, respectively. Therefore, $\hom(\vec{H}\sqcup \vec{F},\vec{G})=\hom(\vec{H},\vec{G})\hom(\vec{F},\vec{G})$, which implies the result.
\end{proof}

The above lemmas are key to building the family of \emph{impartial} oriented graphs---i.e. oriented graphs that are simultaneously tournament Sidorenko and tournament anti-Sidorenko---characterized by Zhao and Zhou~\cite{ZhaoZhou20}. The smallest non-trivial example is $\vec{P}_{1,2}$.

\begin{prop}[See Zhao and Zhou~\cite{ZhaoZhou20}]
\label{prop:P12}
The oriented path $\vec{P}_{1,2}$ is tournament anti-Sidorenko and tournament Sidorenko.
\end{prop}

\begin{proof}
Let $\vec{F}$ be the oriented graph with vertices $0,1,2$ and $3$ and arcs $(0,1)$ and $(3,2)$. Let $\vec{F}_{1,2}$ and $\vec{F}_{2,1}$ be obtained from $\vec{F}$ by adding the arcs $(1,2)$ and $(2,1)$, respectively. Observe that both $\vec{F}_{1,2}$ and $\vec{F}_{2,1}$ are isomorphic to $\vec{P}_{1,2}$. Thus, by Lemma~\ref{lem:edgeFlip}, every tournament $\vec{T}$ satisfies
\[\hom(\vec{F},\vec{T})=2\hom(\vec{P}_{1,2},\vec{T}) + O\left(v(\vec{T})^{3}\right)\]
or, in other words,
\[\hom(\vec{P}_{1,2},\vec{T}) = \frac{1}{2}\left(\hom(\vec{F},\vec{T})  -O\left(v(\vec{T})^{3}\right)\right).\]
Note that $\vec{F}$ is isomorphic to $\vec{P}_1\sqcup\vec{P}_1$. It is easily observed that every tournament $\vec{T}$ satisfies $\hom(\vec{P}_1,\vec{T})=\binom{v(\vec{T})}{2}$ and so, by Lemma~\ref{lem:union}, we have $\hom(\vec{F},\vec{T})=\binom{v(\vec{T})}{2}^2$. Putting all of this together, we get $(1-o(1))(1/2)^3\leq t(\vec{P}_{1,2},\vec{T})\leq (1/2)^3$ for every tournament $\vec{T}$ and so the proof is complete. 
\end{proof}

Next, we show that $\vec{P}_{1,2,1,1}$ is tournament Sidorenko and use that to get that $\vec{P}_{1,4}$ is tournament anti-Sidorenko, thereby establishing another case of Theorem~\ref{th:2blocks}. 

\begin{prop}
\label{prop:P1211}
The oriented path $\vec{P}_{1,2,1,1}$ is tournament Sidorenko. 
\end{prop}

\begin{proof}
Let $\vec{F}$ be the oriented graph with vertices $0,1,2,3,4,5$ and arcs $(0,1),(2,1),(3,4),(5,4)$. Let $\vec{F}_{2,3}$ and $\vec{F}_{3,2}$ be obtained from $\vec{F}$ by adding the arcs $(2,3)$ and $(3,2)$, respectively. Then, clearly, $\vec{F}$ is isomorphic to $\vec{P}_{1,1}\sqcup \vec{P}_{1,1}$ and each of $\vec{F}_{2,3}$ and $\vec{F}_{3,2}$ is isomorphic to $\vec{P}_{1,2,1,1}$. So, by Lemmas~\ref{lem:union} and~\ref{lem:edgeFlip}, we have, for any tournament $\vec{T}$,
\[\hom(\vec{P}_{1,1},\vec{T})^2 = \hom(\vec{F},\vec{T}) = 2\hom(\vec{P}_{1,2,1,1},\vec{T})+O(v(\vec{T})^{5})\]
or, in other words,
\[t(\vec{P}_{1,2,1,1},\vec{T})=\frac{1}{2}t(\vec{P}_{1,1},\vec{T})^2-o(1).\]
It is well-known, and easy to prove, that $\vec{P}_{1,1}$ is tournament Sidorenko; see, e.g.,~\cite[Example~8.1]{ZhaoZhou20}. Therefore,
\[
t(\vec{P}_{1,2,1,1},\vec{T})
\geq
(1-o(1))(1/2)^5,
\]
and so $\vec{P}_{1,2,1,1}$ is tournament Sidorenko.
\end{proof}

\begin{prop}
\label{prop:P14}
The oriented path $\vec{P}_{1,4}$ is tournament anti-Sidorenko. 
\end{prop}

\begin{proof}
Let $\vec{F}$ be the oriented forest with vertices $0,1,2,3,4,5$ with arcs $(0,1),(2,1),(3,2),(5,4)$. Let $\vec{F}_{3,4}$ and $\vec{F}_{4,3}$ be obtained from $\vec{F}$ by adding the arcs $(3,4)$ and $(4,3)$, respectively. Clearly, $\vec{F}_{3,4}$ is isomorphic to $\vec{P}_{1,2,1,1}$ and $\vec{F}_{4,3}$ is isomorphic to $\vec{P}_{1,4}$. So, by Lemma~\ref{lem:edgeFlip}, any tournament $\vec{T}$ satisfies
\[\hom(\vec{P}_{1,4},\vec{T})=\hom(\vec{F},\vec{T})-\hom(\vec{P}_{1,2,1,1},\vec{T})-O(v(\vec{T})^5).\]
Clearly, $\vec{F}$ is isomorphic to $\vec{P}_{1,2}\sqcup \vec{P}_1$. The result now follows from Lemma~\ref{lem:asympGoodEnough} and the facts that $\vec{F}$ is tournament anti-Sidorenko (by Lemma~\ref{lem:union} and Proposition~\ref{prop:P12}) and $\vec{P}_{1,2,1,1}$ is tournament Sidorenko (by Proposition~\ref{prop:P1211}). 
\end{proof}

The next lemma provides a simple construction for generating new examples of tournament anti-Sidorenko graphs from old ones; it is essentially the same as~{\cite[Theorem~1.3]{HeManiNieTungWei25+}}. Let $\vec{H}$ and $\vec{F}$ be oriented graphs and let $u\in V(\vec{H})$ and $v\in V(\vec{F})$. Let $\vec{F}_1$ and $\vec{F}_2$ be two copies of the graph $\vec{F}$ where, for each $w\in V(\vec{F})$ and $i\in\{1,2\}$, we let $w_i$ be the vertex of $\vec{F}_i$ corresponding to $w$. Let $\vec{H}\pm_{u,v}\vec{F}$ be the graph obtained from $\vec{H}\sqcup\vec{F}_1\sqcup\vec{F}_2$ by adding the arcs $(u,v_1)$ and $(v_2,u)$. See Figure~\ref{fig:pm-construction} for a diagram illustrating this construction.

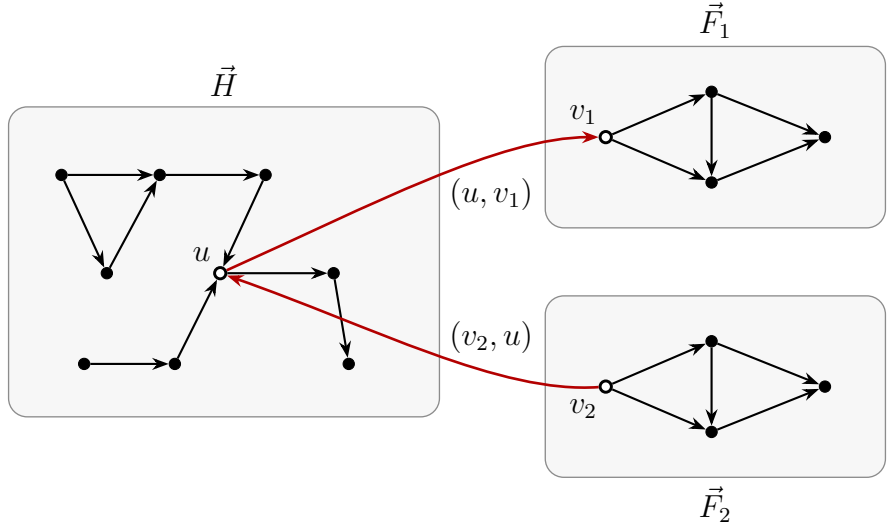
\begin{figure}[htbp]
\centering
\begin{tikzpicture}[x=1.0cm,y=1.0cm]

\draw[proofbox] (-0.8,-2.0) rectangle (4.9,2.1);
\node[lbl] at (2.05,2.45) {$\vec{H}$};

\node[vertex] (h1) at (-0.1,1.2) {};
\node[vertex] (h2) at (1.2,1.2) {};
\node[vertex] (h3) at (2.6,1.2) {};
\node[vertex] (h4) at (0.5,-0.1) {};
\node[marked] (u)  at (2.0,-0.1) {};
\node[vertex] (h6) at (3.5,-0.1) {};
\node[vertex] (h7) at (0.2,-1.3) {};
\node[vertex] (h8) at (1.4,-1.3) {};
\node[vertex] (h9) at (3.7,-1.3) {};

\node[lbl, above left=1pt and 1pt of u] {$u$};

\draw[arc] (h1) -- (h2);
\draw[arc] (h2) -- (h3);
\draw[arc] (h1) -- (h4);
\draw[arc] (h4) -- (h2);
\draw[arc] (h3) -- (u);
\draw[arc] (h7) -- (h8);
\draw[arc] (h8) -- (u);
\draw[arc] (u) -- (h6);
\draw[arc] (h6) -- (h9);

\draw[proofbox] (6.3,0.5) rectangle (10.8,2.9);
\node[lbl] at (8.55,3.25) {$\vec{F}_1$};

\node[marked] (v1) at (7.1,1.7) {};
\node[lbl, above left=1pt and 1pt of v1] {$v_1$};

\node[vertex] (f1a) at (8.5,2.3) {};
\node[vertex] (f1b) at (8.5,1.1) {};
\node[vertex] (f1c) at (10.0,1.7) {};

\draw[arc] (v1) -- (f1a);
\draw[arc] (v1) -- (f1b);
\draw[arc] (f1a) -- (f1c);
\draw[arc] (f1b) -- (f1c);
\draw[arc] (f1a) -- (f1b);

\draw[proofbox] (6.3,-2.8) rectangle (10.8,-0.4);
\node[lbl] at (8.55,-3.15) {$\vec{F}_2$};

\node[marked] (v2) at (7.1,-1.6) {};
\node[lbl, below left=1pt and 1pt of v2] {$v_2$};

\node[vertex] (f2a) at (8.5,-1.0) {};
\node[vertex] (f2b) at (8.5,-2.2) {};
\node[vertex] (f2c) at (10.0,-1.6) {};

\draw[arc] (v2) -- (f2a);
\draw[arc] (v2) -- (f2b);
\draw[arc] (f2a) -- (f2c);
\draw[arc] (f2b) -- (f2c);
\draw[arc] (f2a) -- (f2b);

\draw[addarc] (u) .. controls (4.0,0.8) and (5.7,1.7) .. (v1);
\draw[addarc] (v2) .. controls (5.7,-1.7) and (4.0,-0.8) .. (u);

\node[lbl] at (5.55,0.95) {$\,(u,v_1)$};
\node[lbl] at (5.55,-0.95) {$\,(v_2,u)$};

\end{tikzpicture}
\caption{
A schematic illustration of the construction $\vec{H}\pm_{u,v}\vec{F}$. 
It starts with $\vec{H}\sqcup \vec{F}_1\sqcup \vec{F}_2$, where $\vec{F}_1$ and $\vec{F}_2$ are two copies of $\vec{F}$, and then adds the arcs $(u,v_1)$ and $(v_2,u)$.
}
\label{fig:pm-construction}
\end{figure}

\begin{lem}[He et al.~{\cite[Theorem~1.3]{HeManiNieTungWei25+}}]
\label{lem:inOut}
For any oriented graphs $\vec{H}$ and $\vec{F}$, vertices $u\in V(\vec{H})$ and $v\in V(\vec{F})$ and tournament $\vec{T}$, it holds that
\[\hom(\vec{H}\pm_{u,v}\vec{F},\vec{T})\leq \frac{1}{4}\hom(\vec{H},\vec{T})\hom(\vec{F},\vec{T})^2.\]
\end{lem}

\begin{proof}
Given a homomorphism $\varphi:\vec{H}\to\vec{T}$, let $\hom_{\varphi,u,v}^+(\vec{F},\vec{T})$ be the number of homomorphisms $\psi:\vec{F}\to \vec{T}$ such that $(\varphi(u),\psi(v))\in A(\vec{T})$ and let $\hom_{\varphi,u,v}^-(\vec{F},\vec{T})$ be the number such that $(\psi(v),\varphi(u))\in A(\vec{T})$. Then, clearly, 
\[\hom_{\varphi,u,v}^+(\vec{F},\vec{T}) + \hom_{\varphi,u,v}^-(\vec{F},\vec{T})\leq \hom(\vec{F},\vec{T})\]
and so, by AM-GM,
\[\hom_{\varphi,u,v}^+(\vec{F},\vec{T})\hom_{\varphi,u,v}^-(\vec{F},\vec{T})\leq \left(\frac{\hom_{\varphi,u,v}^+(\vec{F},\vec{T}) + \hom_{\varphi,u,v}^-(\vec{F},\vec{T})}{2}\right)^2 \leq \frac{1}{4}\hom(\vec{F},\vec{T})^2.\]
Now, by construction, we have
\[\hom(\vec{H}\pm_{u,v}\vec{F},\vec{T}) = \sum_{\varphi:\vec{H}\to\vec{T}}\hom_{\varphi,u,v}^+(\vec{F},\vec{T})\hom_{\varphi,u,v}^-(\vec{F},\vec{T})\leq \frac{1}{4}\hom(\vec{H},\vec{T})\hom(\vec{F},\vec{T})^2\]
as desired. 
\end{proof}

As a simple application of the above lemma, we establish Theorem~\ref{th:2blocks} for $\vec{P}_{1,3}$.

\begin{prop}
\label{prop:P13}
The oriented path $\vec{P}_{1,3}$ is tournament anti-Sidorenko. 
\end{prop}

\begin{proof}
Let $\vec{H}:=\vec{P}_0$ and $\vec{F}:=\cev{P}_1$. Then it is easily observed that $\vec{H}\pm_{0,1}\vec{F}$ is isomorphic to $\vec{P}_{1,3}$. So, the result follows from Lemma~\ref{lem:inOut} and the fact that $\vec{P}_0$ and $\vec{P}_1$ are clearly tournament anti-Sidorenko. 
\end{proof}

To close this section, let us apply Lemmas~\ref{lem:edgeFlip} and~\ref{lem:union} one more time to obtain an example of a tournament anti-Sidorenko forest, namely $\vec{P}_{2}\sqcup \vec{P}_{1,1}$, which will be useful in Section~\ref{sec:0mod4}.

\begin{prop}
\label{prop:forest}
The oriented forest $\vec{P}_{2}\sqcup \vec{P}_{1,1}$ is tournament anti-Sidorenko. 
\end{prop}

\begin{proof}
Let $\vec{T}$ be any tournament. Then, by Lemma~\ref{lem:union} and the AM-GM Inequality,
\begin{equation}\label{eq:unionForest}\hom(\vec{P}_{2}\sqcup \vec{P}_{1,1},\vec{T})=\hom(\vec{P}_{2},\vec{T})\hom(\vec{P}_{1,1},\vec{T})\leq \left(\frac{\hom(\vec{P}_{2},\vec{T}) + \hom(\vec{P}_{1,1},\vec{T})}{2}\right)^2.\end{equation}
Let $\vec{F}$ be the oriented forest consisting of three vertices $0,1,2$ and one arc $(0,1)$. Let $\vec{F}_{1,2}$ and $\vec{F}_{2,1}$ be the oriented graphs obtained from $\vec{F}$ by adding the arcs $(1,2)$ and $(2,1)$, respectively. Clearly, $\vec{F}_{1,2}=\vec{P}_2$ and $\vec{F}_{2,1}=\vec{P}_{1,1}$. By Lemma~\ref{lem:edgeFlip}, we have
\[
\hom(\vec{P}_{2},\vec{T})+\hom(\vec{P}_{1,1},\vec{T})
=
\hom(\vec{F},\vec{T}) - O(v(\vec{T})^2).
\]
Since $\vec{F}$ is isomorphic to $\vec{P}_1\sqcup\vec{P}_0$, we have, by Lemma~\ref{lem:union},
\[
\hom(\vec{F},\vec{T})=\binom{v(\vec{T})}{2}v(\vec{T})\leq \frac{1}{2}v(\vec{T})^3.
\]
By combining the last two inequalities and plugging the result into \eqref{eq:unionForest}, we obtain
\[
\hom(\vec{P}_{2}\sqcup \vec{P}_{1,1},\vec{T})
\leq
\left(\frac{v(\vec{T})^3}{4}\right)^2
=
(1/2)^4v(\vec{T})^6,
\]
as desired. For an illustration of the proof, see Figure~\ref{fig:forest-picture}.
\end{proof}

\begin{figure}[htbp] 
\centering
\begin{tikzpicture}[x=0.58cm,y=0.95cm]

\node[vertex] (a0) at (2,0) {};
\node[vertex] (a1) at (3,0) {};
\node[vertex] (a2) at (4,0) {};
\draw[arc] (a0) -- (a1);
\draw[arc] (a1) -- (a2);

\node[vertex] (b0) at (5,0) {};
\node[vertex] (b1) at (6,0) {};
\node[vertex] (b2) at (7,0) {};
\draw[arc] (b0) -- (b1);
\draw[arc] (b2) -- (b1);

\begin{scope}[shift={(10,0)}]

\node[lbl] at (-1.5,0) {$\leq$};
\node[lbl] at (13.5,0) {\small by AM-GM and Lemma~\ref{lem:union}};

\node[lbl, anchor=east] at (1.5,0) {$\dfrac14\Bigl($};

\node[vertex] (c0) at (1.5,0) {};
\node[vertex] (c1) at (2.5,0) {};
\node[vertex] (c2) at (3.5,0) {};
\draw[arc] (c0) -- (c1);
\draw[arc] (c1) -- (c2);

\node[lbl] at (4.5,0) {$+$};

\node[vertex] (d0) at (5.5,0) {};
\node[vertex] (d1) at (6.5,0) {};
\node[vertex] (d2) at (7.5,0) {};
\draw[arc] (d0) -- (d1);
\draw[arc] (d2) -- (d1);

\node[lbl, anchor=west] at (7.5,0) {$\Bigr)^2$};

\end{scope}

\begin{scope}[shift={(9,-1.5)}]

\node[lbl] at (-0.5,0) {$\approx$};

\node[lbl] at (7.5,0) {\small by Lemma~\ref{lem:edgeFlip}};

\node[lbl, anchor=east] at (1.5,0) {$\dfrac14\Bigl($};

\node[vertex] (e0) at (1.5,0) {};
\node[vertex] (e1) at (2.5,0) {};
\draw[arc] (e0) -- (e1);

\node[vertex] (e2) at (3.5,0) {};

\node[lbl, anchor=west] at (3.5,0) {$\Bigr)^2$};

\end{scope}

\begin{scope}[shift={(9,-3)}]

\node[lbl] at (3.5,0) {\small by Lemma~\ref{lem:union}};
\node[lbl] at (-0.5,0) {$\approx$};

\node[lbl, anchor=east] at (1.5,0) {$\dfrac{1}{16}$};

\end{scope}

\end{tikzpicture}
\caption{
A pictorial summary of the proof of Proposition~\ref{prop:forest}.
}
\label{fig:forest-picture}
\end{figure}

\section{Reductions of Theorem~\ref{th:2blocks} Via Algebraic Expansion}
\label{sec:2blocks}

The focus of this section is on reducing Theorem~\ref{th:2blocks} to the case that both blocks have the same length. We start with the following standard application of the Cauchy--Schwarz inequality.

\begin{lem}
\label{lem:P1331reduction}
If $k\geq3$, then every tournament $\vec{T}$ satisfies
\[t(\vec{P}_{1,k},\vec{T})^2\leq t(\vec{P}_{1,3,3,1},\vec{T})t(\vec{P}_{k-3,k-3},\vec{T}).\]
\end{lem}

\begin{proof}
Let $\vec{T}$ be a tournament. Given an oriented graph $\vec{H}$, a vertex $v\in V(\vec{H})$ and $u\in V(\vec{T})$, let $\hom_{v\mapsto u}(\vec{H},\vec{T})$ be the number of homomorphisms $\varphi:\vec{H}\to\vec{T}$ such that $\varphi(v)=u$. Then, by the Cauchy--Schwarz Inequality,
\begin{align*}
\hom(\vec{P}_{1,k},\vec{T}) &= \sum_{u\in V(\vec{T})}\hom_{v_4\mapsto u}(\vec{P}_{1,3},\vec{T})\cdot \hom_{v_0\mapsto u}(\cev{P}_{k-3},\vec{T})\\
&\leq\left(\sum_{u\in V(\vec{T})}\hom_{v_4\mapsto u}(\vec{P}_{1,3},\vec{T})^2\right)^{1/2} \left(\sum_{u\in V(\vec{T})}\hom_{v_0\mapsto u}(\cev{P}_{k-3},\vec{T})^2\right)^{1/2}\\
& = \hom(\vec{P}_{1,3,3,1},\vec{T})^{1/2}\hom(\vec{P}_{k-3,k-3},\vec{T})^{1/2}.
\end{align*}
The result now follows by dividing by $v(\vec{T})^{k+2}$ and squaring both sides. For an illustration of this proof, see Figure~\ref{fig:P1331reduction-picture}.
\end{proof}

\begin{figure}[htbp]
\centering
\begin{tikzpicture}[x=0.58cm,y=0.95cm]

\begin{scope}[shift={(-1,0)}]

\node[vertex]     (a0) at (2,0) {};
\node[vertex]     (a1) at (3,0) {};
\node[vertex]     (a2) at (4,0) {};
\node[vertex]     (a3) at (5,0) {};
\node[vertex] (a4) at (6,0) {};
\node[vertex]     (a5) at (7,0) {};
\node[vertex]     (a6) at (8,0) {};
\node[vertex]     (a7) at (10,0) {};
\node[vertex]     (a8) at (11,0) {};

\node (dotsa) at (9,0) {$\dots$};

\draw[arc] (a0) -- (a1);
\draw[arc] (a2) -- (a1);
\draw[arc] (a3) -- (a2);
\draw[arc] (a4) -- (a3);

\draw[arc] (a5) -- (a4);
\draw[arc] (a6) -- (a5);
\draw[arc] (a8) -- (a7);

\draw[decorate,decoration={brace,amplitude=4pt}] (2,0.3) -- (3,0.3)
node[midway,above=5pt] {$1$};
\draw[decorate,decoration={brace,amplitude=4pt}] (3,0.3) -- (11,0.3)
node[midway,above=5pt] {$k$};

\node[lbl] at (12,0) {$=$};

\end{scope}

\begin{scope}[shift={(13,0)}]
\node[lbl, anchor=east] at (1.5,0) {$\sum_{u\in V(\vec{T})}$};

\node[vertex]     (a0) at (2,0) {};
\node[vertex]     (a1) at (3,0) {};
\node[vertex]     (a2) at (4,0) {};
\node[vertex]     (a3) at (5,0) {};
\node[openvertex] (a4) at (6,0) {};
\node[lbl, below=2pt] at (a4) {$u$};
\node[vertex]     (a5) at (7,0) {};
\node[vertex]     (a6) at (8,0) {};
\node[vertex]     (a7) at (10,0) {};
\node[vertex]     (a8) at (11,0) {};

\node (dotsa) at (9,0) {$\dots$};

\draw[arc] (a0) -- (a1);
\draw[arc] (a2) -- (a1);
\draw[arc] (a3) -- (a2);
\draw[arc] (a4) -- (a3);

\draw[arc] (a5) -- (a4);
\draw[arc] (a6) -- (a5);
\draw[arc] (a8) -- (a7);

\draw[decorate,decoration={brace,amplitude=4pt}] (2,0.3) -- (3,0.3)
node[midway,above=5pt] {$1$};
\draw[decorate,decoration={brace,amplitude=4pt}] (3,0.3) -- (6,0.3)
node[midway,above=5pt] {$3$};
\draw[decorate,decoration={brace,amplitude=4pt}] (6,0.3) -- (11,0.3)
node[midway,above=5pt] {$k-3$};

\end{scope}

\begin{scope}[shift={(0,-1.8)}]

\node[lbl] at (-0.5,0) {$\leq$};
\node[lbl] at (-0.5,0.6) {\small C-S};

\node[lbl, anchor=east] at (3.5,0) {$\Bigl(\sum_{u\in V(\vec{T})}$};

\node[vertex]     (b0) at (4,0) {};
\node[vertex]     (b1) at (5,0) {};
\node[vertex]     (b2) at (6,0) {};
\node[vertex]     (b3) at (7,0) {};
\node[openvertex] (b4) at (8,0) {};
\node[lbl, below=2pt] at (b4) {$u$};
\node[vertex]     (b5) at (9,0) {};
\node[vertex]     (b6) at (10,0) {};
\node[vertex]     (b7) at (11,0) {};
\node[vertex]     (b8) at (12,0) {};

\draw[arc] (b0) -- (b1);
\draw[arc] (b2) -- (b1);
\draw[arc] (b3) -- (b2);
\draw[arc] (b4) -- (b3);

\draw[arc] (b4) -- (b5);
\draw[arc] (b5) -- (b6);
\draw[arc] (b6) -- (b7);
\draw[arc] (b8) -- (b7);

\draw[decorate,decoration={brace,amplitude=4pt}] (4,0.3) -- (5,0.3)
node[midway,above=5pt] {$1$};
\draw[decorate,decoration={brace,amplitude=4pt}] (5,0.3) -- (8,0.3)
node[midway,above=5pt] {$3$};
\draw[decorate,decoration={brace,amplitude=4pt}] (8,0.3) -- (11,0.3)
node[midway,above=5pt] {$3$};
\draw[decorate,decoration={brace,amplitude=4pt}] (11,0.3) -- (12,0.3)
node[midway,above=5pt] {$1$};

\node[lbl, anchor=west] at (12,0) {$\Bigr)^{1/2}$};

\node[lbl] at (14,0) {$\cdot$};

\node[lbl, anchor=east] at (17.5,0) {$\Bigl(\sum_{u\in V(\vec{T})}$};

\node[vertex]     (c0) at (18,0) {};
\node[vertex]     (c1) at (19,0) {};
\node[vertex]     (c2) at (21,0) {};
\node[openvertex] (c3) at (22,0) {};
\node[lbl, below=2pt] at (c3) {$u$};
\node[vertex]     (c4) at (23,0) {};
\node[vertex]     (c5) at (25,0) {};
\node[vertex]     (c6) at (26,0) {};

\node (dotsc1) at (20,0) {$\dots$};
\node (dotsc2) at (24,0) {$\dots$};

\draw[arc] (c0) -- (c1);
\draw[arc] (c2) -- (c3);
\draw[arc] (c4) -- (c3);
\draw[arc] (c6) -- (c5);

\draw[decorate,decoration={brace,amplitude=4pt}] (18,0.3) -- (22,0.3)
node[midway,above=5pt] {$k-3$};
\draw[decorate,decoration={brace,amplitude=4pt}] (22,0.3) -- (26,0.3)
node[midway,above=5pt] {$k-3$};

\node[lbl, anchor=west] at (26,0) {$\Bigr)^{1/2}$};

\end{scope}

\begin{scope}[shift={(-2,-3.6)}]

\node[lbl] at (1.5,0) {$=$};

\node[lbl, anchor=east] at (3.5,0) {$\Bigl($};

\node[vertex] (d0) at (3.5,0) {};
\node[vertex] (d1) at (4.5,0) {};
\node[vertex] (d2) at (5.5,0) {};
\node[vertex] (d3) at (6.5,0) {};
\node[vertex] (d4) at (7.5,0) {};
\node[vertex] (d5) at (8.5,0) {};
\node[vertex] (d6) at (9.5,0) {};
\node[vertex] (d7) at (10.5,0) {};
\node[vertex] (d8) at (11.5,0) {};

\draw[arc] (d0) -- (d1);
\draw[arc] (d2) -- (d1);
\draw[arc] (d3) -- (d2);
\draw[arc] (d4) -- (d3);

\draw[arc] (d4) -- (d5);
\draw[arc] (d5) -- (d6);
\draw[arc] (d6) -- (d7);
\draw[arc] (d8) -- (d7);

\draw[decorate,decoration={brace,amplitude=4pt}] (3.5,0.3) -- (4.5,0.3)
node[midway,above=5pt] {$1$};
\draw[decorate,decoration={brace,amplitude=4pt}] (4.5,0.3) -- (7.5,0.3)
node[midway,above=5pt] {$3$};
\draw[decorate,decoration={brace,amplitude=4pt}] (7.5,0.3) -- (10.5,0.3)
node[midway,above=5pt] {$3$};
\draw[decorate,decoration={brace,amplitude=4pt}] (10.5,0.3) -- (11.5,0.3)
node[midway,above=5pt] {$1$};

\node[lbl, anchor=west] at (11.5,0) {$\Bigr)^{1/2}$};

\node[lbl] at (13.5,0) {$\cdot$};

\node[lbl, anchor=east] at (15,0) {$\Bigl($};

\node[vertex] (e0) at (15,0) {};
\node[vertex] (e1) at (16,0) {};
\node[vertex] (e2) at (18,0) {};
\node[vertex] (e3) at (19,0) {};
\node[vertex] (e4) at (20,0) {};
\node[vertex] (e5) at (22,0) {};
\node[vertex] (e6) at (23,0) {};

\node (dotse1) at (17,0) {$\dots$};
\node (dotse2) at (21,0) {$\dots$};

\draw[arc] (e0) -- (e1);
\draw[arc] (e2) -- (e3);
\draw[arc] (e4) -- (e3);
\draw[arc] (e6) -- (e5);

\draw[decorate,decoration={brace,amplitude=4pt}] (15,0.3) -- (19,0.3)
node[midway,above=5pt] {$k-3$};
\draw[decorate,decoration={brace,amplitude=4pt}] (19,0.3) -- (23,0.3)
node[midway,above=5pt] {$k-3$};

\node[lbl, anchor=west] at (23,0) {$\Bigr)^{1/2}$};

\end{scope}

\end{tikzpicture}
\caption{
A pictorial summary of the proof of Lemma~\ref{lem:P1331reduction}. 
}
\label{fig:P1331reduction-picture}
\end{figure}

In light of Lemma~\ref{lem:P1331reduction}, it is useful to show that $\vec{P}_{1,3,3,1}$ is tournament anti-Sidorenko. Thus, most of the rest of this section is devoted to proving the following proposition.

\begin{prop}
\label{prop:P1331}
The oriented path $\vec{P}_{1,3,3,1}$ is tournament anti-Sidorenko.
\end{prop}

Our approach to proving Proposition~\ref{prop:P1331} is based on the ``algebraic expansion'' trick that is commonly used in papers on graph limits. For this, it is convenient to work in terms of matrices as opposed to tournaments.\footnote{It would also be equivalent to work in the setting of \emph{tournament limits} as in~\cite{Grzesik+23,NoelRanganathanSimbaqueba26,ChanGrzesikKralNoel20,ZhaoZhou20}; however, for the purposes of this paper, this would introduce several technicalities and would not bring any advantages.} Given an $n\times n$ matrix $A$ and $1\leq i,j\leq n$, let $A(i,j)$ be the entry of $A$ on the $i$th row and $j$th column. Given an oriented graph $\vec{H}$, define
\begin{equation}
\label{eq:homHA}
\hom(\vec{H},A):=\sum_{f:V(\vec{H})\to [n]}\left( \prod_{(u,v)\in A(\vec{H})} A(f(u),f(v))\right)
\end{equation}
and $t(\vec{H},A):=\frac{1}{n^{v(\vec{H})}}\hom(\vec{H},A)$. The \emph{adjacency matrix} of a tournament $\vec{T}$ with vertices $u_1,\dots,u_n$ is the $n\times n$ matrix $A_{\vec{T}}$ where the entry on the $i$th row and $j$th column is $1$ if $(u_i,u_j)\in A(\vec{T})$ and $0$ otherwise. It is easily observed that $\hom(\vec{H},\vec{T})=\hom(\vec{H},A_{\vec{T}})$ for any tournament $\vec{T}$. It will actually be more convenient to deal with an \emph{augmented adjacency matrix} for $\vec{T}$ which we define by 
\[A_{\vec{T}}^*:=A_{\vec{T}}+\frac{1}{2}I_n\]
where $I_n$ is the $n\times n$ identity matrix. The next observation follows from the fact that all entries of $A_{\vec{T}}^*$ and $A_{\vec{T}}$ are non-negative and every entry of $A_{\vec{T}}^*$ is greater than or equal to the corresponding entry of $A_{\vec{T}}$. 

\begin{obs}
\label{obs:1/2diag}
For every oriented graph $\vec{H}$ and tournament $\vec{T}$, 
\[\hom(\vec{H},\vec{T})\leq \hom(\vec{H},A_{\vec{T}}^*).\]
\end{obs}

Given an oriented graph $\vec{H}$ and a set $S\subseteq A(\vec{H})$, let $\vec{H}[S]$ be the oriented graph with vertex set $V(\vec{H})$ and arc set $S$. Let $J_n$ be the $n\times n$ all-ones matrix. The next lemma provides a useful alternative expression for $\hom(\vec{H},A)$ for any matrix $A$.

\begin{lem}
\label{lem:expansion}
Let $\vec{H}$ be an oriented graph. For any $n\times n$ matrix $A$, if $U=A-\frac{1}{2}J_n$, then
\[\hom(\vec{H},A)=\sum_{S\subseteq A(\vec{H})}(1/2)^{a(\vec{H})-|S|}\hom(\vec{H}[S],U).\]
\end{lem}

\begin{proof}
By \eqref{eq:homHA},
\begin{align*}
\hom(\vec{H},A)&=\sum_{f:V(\vec{H})\to [n]} \left(\prod_{(u,v)\in A(\vec{H})} A(f(u),f(v))\right)=\sum_{f:V(\vec{H})\to [n]} \left(\prod_{(u,v)\in A(\vec{H})} (1/2+ U(f(u),f(v)))\right)\\
&=\sum_{f:V(\vec{H})\to [n]} \left(\sum_{S\subseteq A(\vec{H})}(1/2)^{a(\vec{H})-|S|}\prod_{(u,v)\in S} U(f(u),f(v))\right)\\
&=\sum_{S\subseteq A(\vec{H})}(1/2)^{a(\vec{H})-|S|} \left(\sum_{f:V(\vec{H})\to [n]}\prod_{(u,v)\in S} U(f(u),f(v))\right)\\
&=\sum_{S\subseteq A(\vec{H})}(1/2)^{a(\vec{H})-|S|}\hom(\vec{H}[S],U)
\end{align*}
as desired.
\end{proof}

Our proof of Proposition~\ref{prop:P1331} involves applying Lemma~\ref{lem:expansion} to $\vec{H}=\vec{P}_{1,3,3,1}$ and the matrix $A=A_{\vec{T}}^*$ for a tournament $\vec{T}$. Given an $n$-vertex tournament $\vec{T}$, the matrix $U_T:=A_{\vec{T}}^* - \frac{1}{2}J_n$ is skew-symmetric, meaning that $U^T=-U$, and all entries of $U$ are between $-1/2$ and $1/2$. Thus, we would benefit from knowing properties of $\hom(\vec{P},U)$ where $\vec{P}$ is an oriented path and $U$ is a skew-symmetric matrix with entries in $[-1/2,1/2]$. We borrow several such results from a paper of Grzesik, Il'kovi\v{c}, Kielak and Kr\'al'~\cite{Grzesik+23}.

\begin{lem}[Grzesik et al.~{\cite[Lemma~4]{Grzesik+23}}]
\label{lem:polynomial}
Let $p(x)=\sum_{k=1}^m\alpha_kx^k$ be a polynomial over $\mathbb{R}$. If $p(x)\geq 0$ for all $x\in [0,1/\pi^2]$, then 
\[\sum_{k=1}^m\alpha_k\cdot t(\vec{P}_{k,k},U)\geq 0\]
for every skew-symmetric matrix $U$ with entries in $[-1/2,1/2]$. 
\end{lem}

\begin{lem}[Grzesik et al.~{\cite[Equation~(6)]{Grzesik+23}}]
\label{lem:P11P22}
Every skew-symmetric matrix $U$ satisfies $t(\vec{P}_{1,1},U)^2\leq t(\vec{P}_{2,2},U)$.
\end{lem}

\begin{lem}[Grzesik et al.~{\cite[Lemma~8]{Grzesik+23}}]
\label{lem:P11upper}
Every skew-symmetric matrix $U$ with entries in $[-1/2,1/2]$ satisfies $0\leq t(\vec{P}_{1,1},U)\leq \frac{1}{12}$.
\end{lem}

We will also use the following two standard facts about homomorphism densities into skew-symmetric matrices.

\begin{lem}[Grzesik et al.~{\cite[Proposition~1]{Grzesik+23}}]
\label{lem:edgeFlipU}
Let $\vec{H}$ and $\vec{F}$ be two oriented graphs that differ by reversing the orientation of a single arc. Then $t(\vec{H},U)=-t(\vec{F},U)$ for every skew-symmetric matrix $U$. 
\end{lem}

\begin{proof}
Let $x,y$ be the pair of vertices such that $(x,y)\in A(\vec{H})$ and $(y,x)\in A(\vec{F})$. Then
\begin{align*}
\hom(\vec{H},U)&=\sum_{f:V(\vec{H})\to [n]}\left( \prod_{(u,v)\in A(\vec{H})} U(f(u),f(v))\right)\\
&=\sum_{f:V(\vec{H})\to [n]}\left(U(f(x),f(y)) \prod_{(u,v)\in A(\vec{H})\setminus\{(x,y)\}} U(f(u),f(v))\right)\\
&=-\sum_{f:V(\vec{H})\to [n]}\left(U(f(y),f(x)) \prod_{(u,v)\in A(\vec{H})\setminus\{(x,y)\}} U(f(u),f(v))\right)\\
&=-\hom(\vec{F},U)
\end{align*}
from which the result follows.
\end{proof}

\begin{lem}[Grzesik et al.~{\cite[Proposition~2]{Grzesik+23}}]
\label{lem:pathU}
Let $\vec{P}$ be an oriented path with an odd number of edges. Then $t(\vec{P},U)=0$ for every skew-symmetric matrix $U$. 
\end{lem}

\begin{proof}
Let $k=2\ell+1$ be the length of $\vec{P}$. By swapping the directions of arcs of $\vec{P}$, one by one, and applying Lemma~\ref{lem:edgeFlipU} we get that $t(\vec{P},U)$ is equal to either $t(\vec{P}_{\ell,\ell+1},U)$ or $-t(\vec{P}_{\ell,\ell+1},U)$. Also, by swapping the direction of the central arc, we see that $t(\vec{P}_{\ell,\ell+1},U) = -t(\vec{P}_{\ell+1,\ell},U)$. However, $\vec{P}_{\ell,\ell+1}$ and $\vec{P}_{\ell+1,\ell}$ are clearly isomorphic, and so $t(\vec{P}_{\ell,\ell+1},U) = t(\vec{P}_{\ell+1,\ell},U)$. Thus, we conclude that $t(\vec{P}_{\ell,\ell+1},U)=0$ and so $t(\vec{P},U)=0$ as well.
\end{proof}

We now present the proof of Proposition~\ref{prop:P1331}. 

\begin{proof}[Proof of Proposition~\ref{prop:P1331}]
Let $\vec{T}$ be a tournament with $n$ vertices and let $U_T:= A_{\vec{T}}^*-\frac{1}{2}J_n$. We start by applying Observation~\ref{obs:1/2diag} and Lemma~\ref{lem:expansion} and dividing both sides by $n^9$ to get
\[t(\vec{P}_{1,3,3,1},\vec{T}) \leq t(\vec{P}_{1,3,3,1},A_{\vec{T}}^*)=\sum_{S\subseteq A(\vec{P}_{1,3,3,1})}(1/2)^{8-|S|}t(\vec{P}_{1,3,3,1}[S],U_T).\]
In the expression on the right side, any term corresponding to a set $S$ such that $\vec{P}_{1,3,3,1}[S]$ has a component with an odd number of edges is zero by Lemmas~\ref{lem:pathU} and~\ref{lem:union}.\footnote{Technically, Lemma~\ref{lem:union} was only proved for homomorphisms to tournaments, not for general matrices. However, it is well known, and easy to see, that it holds for general matrices as well.} So, using Lemma~\ref{lem:edgeFlipU} and~\ref{lem:union}, we can express every remaining term on the right side as a polynomial in the expressions $t(\vec{P}_{k,k},U_T)$ for $k\in \{1,2,3,4\}$. It is equal to
\begin{equation}
\label{eq:expansion}
\begin{gathered}
t(\vec{P}_{4,4},U_T) + \frac{3}{4}t(\vec{P}_{3,3},U_T) - \frac{1}{2}t(\vec{P}_{1,1},U_T)t(\vec{P}_{2,2},U_T) + \frac{1}{4}t(\vec{P}_{1,1},U_T)^3-\frac{3}{16}t(\vec{P}_{2,2},U_T)\\
+\frac{1}{8}t(\vec{P}_{1,1},U_T)^2-\frac{1}{64}t(\vec{P}_{1,1},U_T) + \frac{1}{256}.
\end{gathered}
\end{equation}
We will be done if we can show that the above expression is always bounded above by $1/256$. To this end, define 
\[f(U_T):=t(\vec{P}_{4,4},U_T) + \frac{3}{4}t(\vec{P}_{3,3},U_T)-\frac{3}{16}t(\vec{P}_{2,2},U_T)-\frac{35}{6336}t(\vec{P}_{1,1},U_T)\]
and
\[g(U_T):=- \frac{1}{2}t(\vec{P}_{1,1},U_T)t(\vec{P}_{2,2},U_T)+ \frac{1}{4}t(\vec{P}_{1,1},U_T)^3 +\frac{1}{8}t(\vec{P}_{1,1},U_T)^2 - \frac{1}{99}t(\vec{P}_{1,1},U_T).\]
We observe that the expression in \eqref{eq:expansion} is precisely $f(U_T)+g(U_T)+\frac{1}{256}$. So, it suffices to show that $f(U_T)\leq 0$ and $g(U_T)\leq 0$. 

Consider first the polynomial
\[
p(x):=x^4+\frac{3}{4}x^3-\frac{3}{16}x^2-\frac{35}{6336}x.
\]
It is a simple calculus exercise to show that $p(x)\leq 0$ for all $x\in[0,1/\pi^2]$. Thus $-p(x)\geq0$ on this interval, so Lemma~\ref{lem:polynomial} gives $f(U_T)\leq 0$. Next, by Lemma~\ref{lem:P11P22}, we have that 
\[g(U_T)\leq - \frac{1}{4}t(\vec{P}_{1,1},U_T)^3 +\frac{1}{8}t(\vec{P}_{1,1},U_T)^2 - \frac{1}{99}t(\vec{P}_{1,1},U_T).\]
By Lemma~\ref{lem:P11upper}, we know that $0\leq t(\vec{P}_{1,1},U_T)\leq 1/12$. The polynomial $q(x)=-\frac{1}{4}x^3 + \frac{1}{8}x^2-\frac{1}{99}x$ satisfies $q(x)\leq 0$ for all $x\in [0,1/12]$ and so $g(U_T)\leq0$. This completes the proof. 
\end{proof}

Combining Lemma~\ref{lem:P1331reduction} and Proposition~\ref{prop:P1331} immediately yields the following. 

\begin{lem}
\label{lem:reduction}
For $k\geq5$, if $\vec{P}_{k-3,k-3}$ is tournament anti-Sidorenko, then so is $\vec{P}_{1,k}$.  
\end{lem}

Thus, as we already know that $\vec{P}_{1,2},\vec{P}_{1,3}$ and $\vec{P}_{1,4}$ are tournament anti-Sidorenko by Propositions~\ref{prop:P12},~\ref{prop:P13} and~\ref{prop:P14}, respectively, Lemma~\ref{lem:reduction} implies that, in order to prove that $\vec{P}_{1,k}$ is tournament anti-Sidorenko for every $k\geq2$, it suffices to know that $\vec{P}_{j,j}$ is tournament anti-Sidorenko for every $j\geq2$. All oriented paths with exactly two blocks, each of which has length two, including the paths $\vec{P}_{j,j}$ for $j\geq2$, are covered by  Theorem~\ref{th:0mod4}. Putting this all together, we see that, once Theorem~\ref{th:0mod4} has been proven, Theorem~\ref{th:2blocks} will easily  follow.

As it turns out, it is useful to establish the case of $\vec{P}_{2,2}$ separately. This is quite easy to do using the ideas from this section. 

\begin{prop}
\label{prop:P22}
The oriented path $\vec{P}_{2,2}$ is tournament anti-Sidorenko. 
\end{prop}

\begin{proof}
Let $\vec{T}$ be a tournament with $n$ vertices and let $U_T:= A_{\vec{T}}^*-\frac{1}{2}J_n$. Analogous to the proof of Proposition~\ref{prop:P1331}, we have
\[t(\vec{P}_{2,2},\vec{T})\leq t(\vec{P}_{2,2},A_{\vec{T}}^*)=\sum_{S\subseteq A(\vec{P}_{2,2})}(1/2)^{4-|S|}t(\vec{P}_{2,2}[S],U_T).\]
Using Lemmas~\ref{lem:pathU},~\ref{lem:union} and~\ref{lem:edgeFlipU}, the right-hand side is equal to
\[t(\vec{P}_{2,2},U_T)-\frac{1}{4}t(\vec{P}_{1,1},U_T) + \frac{1}{16}.\]
The polynomial $\frac{1}{4}x-x^2$ is non-negative for all $x\in [0,1/4]$; in particular, it is non-negative for $x\in[0,1/\pi^2]$. Hence Lemma~\ref{lem:polynomial} gives
\[
t(\vec{P}_{2,2},U_T)-\frac{1}{4}t(\vec{P}_{1,1},U_T)\leq 0.
\]
Therefore $t(\vec{P}_{2,2},A_{\vec{T}}^*)\leq 1/16$, and so $\vec{P}_{2,2}$ is tournament anti-Sidorenko.
\end{proof}

\section{Making Sandwiches}
\label{sec:sandwich}

The purpose of this section is to present an information-theoretic approach for certifying that a given oriented forest is tournament anti-Sidorenko. The approach boils down to constructing a certificate which we call a ``sandwich.'' The main lemma of the section states that the existence of an $\vec{H}$-sandwich implies that $\vec{H}$ is tournament anti-Sidorenko. The later sections are devoted to constructing sandwiches for the paths and spiders and using them to prove our main theorems.

Before stating our main definition and lemma, we require the following standard terminology. We say that an oriented graph $\vec{D}$ contains another oriented graph $\vec{H}$ as a \emph{subgraph} if $V(\vec{H})\subseteq V(\vec{D})$ and $A(\vec{H})\subseteq A(\vec{D})$. A set $S\subseteq V(\vec{D})$ is \emph{connected} if, for every partition $\{A,B\}$ of $S$, there exists an arc of $\vec{D}$ with one endpoint in each of $A$ and $B$. A \emph{component} of $\vec{D}$ is a maximal non-empty connected subset of $V(\vec{D})$. For $S\subseteq V(\vec{D})$, the subgraph of $\vec{D}$ \emph{induced} by $S$, denoted $\vec{D}[S]$, is the oriented graph with vertex set $S$ containing all arcs of $\vec{D}$ that have both endpoints in $S$. Given a homomorphism $\psi:\vec{D}\to \vec{H}$ and an arc $(u,v)\in A(\vec{H})$, let $\psi^{-1}(u,v)$ be the set of all arcs $(x,y)$ of $\vec{D}$ such that $\psi(x)=u$ and $\psi(y)=v$. An \emph{involution} on a set $X$ is a function $\pi:X\to X$ such that $\pi\circ\pi$ is the identity.  

\begin{defn}
\label{def:sandwich}
Let $\vec{H}$ be an oriented forest. An \emph{$\vec{H}$-sandwich}\footnote{We use the term ``sandwich'' because the diagrams in this paper typically consist of $\vec{H}$ with additional structures drawn above and below $\vec{H}$; so, it is like a sandwich in which $\vec{H}$ is the filling.} is a tuple $\mathcal{S}=(\vec{D},\psi,\omega,\pi)$ satisfying the following properties: 
\begin{enumerate}
    \stepcounter{equation}\item\label{eq:sandwichForest} $\vec{D}$ is an oriented forest that contains $\vec{H}$ as a subgraph.
    \stepcounter{equation}\item\label{eq:sandwichMap} $\psi$ is a homomorphism from $\vec{D}$ to $\vec{H}$ which fixes every vertex of $\vec{H}$.
    \stepcounter{equation}\item\label{eq:sandwichWeights} $\omega:\mathcal{C}\to [0,\infty)$ is a weight function, where $\mathcal{C}$ is the collection of all components of $\vec{D}\setminus V(\vec{H})$.  Given $v\in V(\vec{D})$, we let $\omega(v)$ be equal to $\omega(C)$ if $v\in C$ for some $C\in\mathcal{C}$ and $0$ otherwise. Given an arc $(u,v)\in A(\vec{D})$, if there exists $C\in\mathcal{C}$ such that $\{u,v\}\cap C\neq \emptyset$, then we let $\omega(u,v):=\omega(C)$ and we let $\omega(u,v):=0$ otherwise.
    \stepcounter{equation}\item\label{eq:coverede} For every arc $(u,v)$ of $\vec{H}$, $\sum_{(x,y)\in \psi^{-1}(u,v)}\omega(x,y) = 1$.
    \stepcounter{equation}\item\label{eq:coveredv} For every vertex $v$ of $\vec{H}$, $\sum_{x\in \psi^{-1}(v)}\omega(x) \leq 1$.
    \stepcounter{equation}\item\label{eq:partition} $\{\mathcal{C}_1,\dots,\mathcal{C}_m\}$ is a partition of $\mathcal{C}$ such that $\omega$ is constant on $\mathcal{C}_j$ for all $1\leq j\leq m$.
    \stepcounter{equation}\item\label{eq:anti-Sidcomponents} $\vec{D}\left[\bigcup_{C\in \mathcal{C}_j}C\right]$ is tournament anti-Sidorenko for all $1\leq j\leq m$.
    \stepcounter{equation}\item\label{eq:singleAttachment} Every component $C\in\mathcal{C}$ is incident with at most one arc of $\vec{D}$ that has one endpoint in $C$ and the other endpoint in $V(\vec{H})$.
    \stepcounter{equation}\item\label{eq:involution} $\pi$ is an involution on $V(\vec{D})\setminus V(\vec{H})$ that is a homomorphism from $\vec{D}\setminus V(\vec{H})$ to itself.
    \stepcounter{equation}\item\label{eq:sameWeight} $\omega(v)=\omega(\pi(v))$ for all $v\in V(\vec{D})\setminus V(\vec{H})$.
    \stepcounter{equation}\item\label{eq:reverseArcs} Given $u\in V(\vec{H})$ and $v\in V(\vec{D})\setminus V(\vec{H})$, there is an arc from $u$ to $v$ in $\vec{D}$ if and only if there is an arc from $\pi(v)$ to $u$ in $\vec{D}$. 
\end{enumerate}
We say that $\mathcal{S}$ is a \emph{partial $\vec{H}$-sandwich} if all of the above conditions hold, except that \eqref{eq:coverede} is relaxed to
\begin{enumerate}
    \stepcounter{equation}\item\label{eq:coveredepartial} For every arc $(u,v)$ of $\vec{H}$, $\sum_{(x,y)\in \psi^{-1}(u,v)}\omega(x,y) \leq 1$.
\end{enumerate}
\end{defn}

Let $\mathcal S=(\vec{D},\psi,\omega,\pi)$ be an $\vec{H}$-sandwich or a partial $\vec{H}$-sandwich. For an arc $(u,v)\in A(\vec{H})$, define
\[
\cov_{\mathcal S}(u,v):=
\sum_{(x,y)\in \psi^{-1}(u,v)}\omega(x,y)
\] 
and, for $v\in V(\vec{H})$, define
\[
\cov_{\mathcal S}(v):=
\sum_{x\in \psi^{-1}(v)}\omega(x).
\]
Then \eqref{eq:coverede} can be rewritten as $\cov_\mathcal{S}(u,v)=1$ for all $(u,v)\in A(\vec{H})$ and \eqref{eq:coveredv} and \eqref{eq:coveredepartial} can be expressed in terms of $\cov_\mathcal{S}$ in a similar fashion. The main goal of this section is to prove the following lemma which says that, if an $\vec{H}$-sandwich exists, then $\vec{H}$ is tournament anti-Sidorenko.

\begin{lem}
\label{lem:entropyMain}
Let $\vec{H}$ be an oriented forest. If there exists an $\vec{H}$-sandwich, then $\vec{H}$ is tournament anti-Sidorenko.
\end{lem}

While Definition~\ref{def:sandwich} and~\ref{lem:sandwich} are new, they are rooted in several existing ideas. In particular, they are heavily inspired by the approach in our recent solution~\cite{ChenClemenNoel25+} to an extremal problem of Basit, Granet, Horsley, K\"ungen and Staden~\cite{Basit+25+} on alternating paths in edge-coloured graphs. Moreover, the idea underlying the conditions \eqref{eq:anti-Sidcomponents} and \eqref{eq:reverseArcs} is closely related to Lemma~\ref{lem:inOut}, which is essentially the same as~\cite[Theorem~1.3]{HeManiNieTungWei25+}.

Before embarking on the proof of Lemma~\ref{lem:entropyMain}, let us demonstrate its utility by giving an alternative proof of the result~\cite[Theorem~1]{SahSawhneyZhao23} that directed paths are tournament anti-Sidorenko.

\begin{prop}[Sah, Sawhney and Zhao~{\cite[Theorem~1]{SahSawhneyZhao23}}]
\label{prop:SSZ}
$\vec{P}_k$ is tournament anti-Sidorenko for all $k\geq0$.
\end{prop}

\begin{proof}
It is easily observed that $\vec{P}_0$ and $\vec{P}_1$ are tournament anti-Sidorenko, and so we assume that $k\geq2$. Let $\vec{D}_k$ be an oriented forest obtained from $\vec{P}_k$ by adding
\begin{itemize}
\item vertices $w_0,w_1,w_{k-1},w_k$, $u_0^+,\dots,u_{k-2}^+$, and $u_2^-,\dots,u_k^-$.
\item arcs $(w_0,w_1),$ $(w_{k-1},w_k)$, $(u_0^+,v_1),\dots,(u_{k-2}^+,v_{k-1})$, and $(v_1,u_2^-),\dots,(v_{k-1},u_k^-)$.
\end{itemize}
We define $\psi_k:\vec{D}_k\to\vec{P}_k$ so that each vertex of $\vec{D}_k$ is mapped to the vertex with the same subscript as it. That is, $\psi_k(v_i)=v_i$, $\psi_k(u_i^+)=v_i$, $\psi_k(u_i^-)=v_i$ and $\psi_k(w_i)=v_i$. We let $\omega_k$ be a weight function on the components of $\vec{D}_k\setminus V(\vec{P}_k)$ which assigns weight $1/2$ to each component. Finally, we let $\pi_k$ be an involution on $V(\vec{D}_k)\setminus V(\vec{P}_k)$ which fixes $w_0,w_1,w_{k-1},w_k$ and maps $u_{i-1}^+$ to $u_{i+1}^-$ for all $1\leq i\leq k-1$. See Figure~\ref{fig:D5}. It is easily observed that all of the properties of Definition~\ref{def:sandwich} hold for $(\vec{D}_k,\psi_k,\omega_k,\pi_k)$ with respect to $\vec{H}=\vec{P}_k$, where the partition in \eqref{eq:partition} is taken to be the trivial partition on the components of $\vec{D}_k\setminus V(\vec{P}_k)$. Thus, $(\vec{D}_k,\psi_k,\omega_k,\pi_k)$ is a $\vec{P}_k$-sandwich and the result follows by Lemma~\ref{lem:entropyMain}. 
\end{proof}

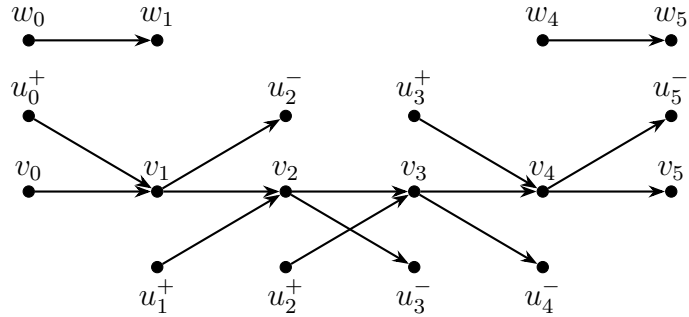
\begin{figure}[htbp]
\centering
\begin{tikzpicture}[x=1.7cm,y=1.0cm]

\node[vertex] (v0) at (0,0) {};
\node[lbl, above=0pt] at (v0) {$v_0$};
\node[vertex] (v1) at (1,0) {};
\node[lbl, above=0pt] at (v1) {$v_1$};
\node[vertex] (v2) at (2,0) {};
\node[lbl, above=0pt] at (v2) {$v_2$};
\node[vertex] (v3) at (3,0) {};
\node[lbl, above=0pt] at (v3) {$v_3$};
\node[vertex] (v4) at (4,0) {};
\node[lbl, above=0pt] at (v4) {$v_4$};
\node[vertex] (v5) at (5,0) {};
\node[lbl, above=0pt] at (v5) {$v_5$};

\draw[arc] (v0) -- (v1);
\draw[arc] (v1) -- (v2);
\draw[arc] (v2) -- (v3);
\draw[arc] (v3) -- (v4);
\draw[arc] (v4) -- (v5);

\node[vertex] (w0) at (0,2) {};
\node[lbl, above=0pt] at (w0) {$w_0$};
\node[vertex] (w1) at (1,2) {};
\node[lbl, above=0pt] at (w1) {$w_1$};
\node[vertex] (w4) at (4,2) {};
\node[lbl, above=0pt] at (w4) {$w_4$};
\node[vertex] (w5) at (5,2) {};
\node[lbl, above=0pt] at (w5) {$w_5$};

\node[vertex] (u0p) at (0,1) {};
\node[lbl, above=0pt] at (u0p) {$u_0^+$};
\node[vertex] (u2m) at (2,1) {};
\node[lbl, above=0pt] at (u2m) {$u_2^-$};

\node[vertex] (u1p) at (1,-1) {};
\node[lbl, below=0pt] at (u1p) {$u_1^+$};
\node[vertex] (u3m) at (3,-1) {};
\node[lbl, below=0pt] at (u3m) {$u_3^-$};

\node[vertex] (u2p) at (2,-1) {};
\node[lbl, below=0pt] at (u2p) {$u_2^+$};
\node[vertex] (u4m) at (4,-1) {};
\node[lbl, below=0pt] at (u4m) {$u_4^-$};

\node[vertex] (u3p) at (3,1) {};
\node[lbl, above=0pt] at (u3p) {$u_3^+$};
\node[vertex] (u5m) at (5,1) {};
\node[lbl, above=0pt] at (u5m) {$u_5^-$};

\draw[arc] (u0p) -- (v1);
\draw[arc] (v1) -- (u2m);

\draw[arc] (u1p) -- (v2);
\draw[arc] (v2) -- (u3m);

\draw[arc] (u2p) -- (v3);
\draw[arc] (v3) -- (u4m);

\draw[arc] (u3p) -- (v4);
\draw[arc] (v4) -- (u5m);

\draw[arc] (w0) -- (w1);
\draw[arc] (w4) -- (w5);

\end{tikzpicture}
\caption{
The oriented graph $\vec{D}_k$ used in the proof of Proposition~\ref{prop:SSZ} in the case $k=5$. The vertices are drawn in six columns and every vertex $u$ of $\vec{D}_5$ is mapped by $\psi_5$ to the vertex $v_i$ in the same column as $u$. Every component of $\vec{D}_5\setminus V(\vec{P}_5)$ has weight $1/2$. We let $\pi_5$ be the unique map which satisfies condition \eqref{eq:reverseArcs} and fixes $w_0,w_1,w_4,w_5$.
}
\label{fig:D5}
\end{figure}

\begin{rem}
\label{rem:convention}
All of our diagrams of $\vec{P}$-sandwiches and partial $\vec{P}$-sandwiches $\mathcal{S}=(\vec{D},\psi,\omega,\pi)$ for an oriented path $\vec{P}$ follow the same convention as Figure~\ref{fig:D5}. Near the middle of the diagram is a copy of $\vec{P}$ with its vertices $v_0,\dots,v_k$ drawn from left to right in $k+1$ columns. All vertices of $\vec{D}$ which are drawn in the $i$th column are mapped by $\psi$ to $v_i$. The weight function is indicated in the caption of the figure or in the diagram itself. Consequently, the values of $\cov_\mathcal{S}$ can be read from the figure by summing the weights in each column and between consecutive columns. In most cases, the involution $\pi$ is almost uniquely determined by the fact that it needs to satisfy \eqref{eq:involution}, \eqref{eq:sameWeight} and \eqref{eq:reverseArcs} and the partition of $\mathcal{C}$ is the trivial partition. Any potential ambiguities are explained in the caption of the figure. To keep the main text readable, most of these figures are collected in Appendix~\ref{app:gadget-atlas}.
\end{rem}

\begin{rem}
A useful feature of our approach is that the search for sandwiches can be automated using linear programming software, which can allow one to quickly discover constructions which are far more complex than any human could hope to find by hand. That is, given an oriented forest $\vec{H}$ and a collection $\mathcal{A}$ of oriented forests which are already known to be tournament anti-Sidorenko, one can create a variable for each tuple $(\vec{A},\psi_{\vec{A}},u,v,s)$ where $\vec{A}\in\mathcal{A},\psi_{\vec{A}}:\vec{A}\to\vec{H},u\in V(\vec{H}),v\in V(\vec{A})$ and $s\in\{-1,0,1\}$. This variable corresponds to the weight of a component of $V(\vec{D})\setminus V(\vec{H})$ which is isomorphic to $\vec{A}$ such that the restriction of $\psi$ to this component is $\psi_{\vec{A}}$ and $\vec{D}$ has an arc from $v$ to $u$ if $s=-1$, from $u$ to $v$ if $s=1$, and no arc between $u$ and $v$ if $s=0$. To ensure that the final construction satisfies all properties of Definition~\ref{def:sandwich} simply boils down to searching for a weight function which satisfies a particular system of linear constraints. This explains how all of the sandwiches in this paper were discovered in practice.
\end{rem}

We now turn our attention toward building up the ideas that we need to prove Lemma~\ref{lem:entropyMain}. Given a $\vec{H}$-sandwich $(\vec{D},\psi,\omega,\pi)$, our goal will be to construct an oriented graph $\vec{F}$ that is larger than $\vec{H}$ such that $t(\vec{H},\vec{T})$ is bounded above and below by two different expressions of $t(\vec{F},\vec{T})$ for every tournament $\vec{T}$, from which it will follow that $\vec{H}$ is tournament anti-Sidorenko. This is made more precise in the next lemma, which is inspired by~\cite[Lemma~2.1]{ChenClemenNoel25+}.

\begin{lem}
\label{lem:sandwich}
Let $\vec{H}$ be an oriented graph with $a(\vec{H})\neq0$. If there exists an oriented graph $\vec{F}$ such that $a(\vec{F})>a(\vec{H})$ and every tournament $\vec{T}$ satisfies
\begin{equation}
\label{eq:entropyBound}
t(\vec{H},\vec{T})^{a(\vec{F})/a(\vec{H})}\leq t(\vec{F},\vec{T})
\end{equation}
and 
\begin{equation}
\label{eq:strippingOff}
t(\vec{F},\vec{T})\leq (1/2)^{a(\vec{F})-a(\vec{H})}t(\vec{H},\vec{T}),
\end{equation}
then $\vec{H}$ is tournament anti-Sidorenko. 
\end{lem}

\begin{proof}
Let $\vec{T}$ be a tournament. We have 
\[t(\vec{H},\vec{T})\leq t(\vec{F},\vec{T})^{a(\vec{H})/a(\vec{F})}\leq \left((1/2)^{a(\vec{F})-a(\vec{H})}t(\vec{H},\vec{T})\right)^{a(\vec{H})/a(\vec{F})}.\]
Since $a(\vec{F})>a(\vec{H})$, this simplifies to $t(\vec{H},\vec{T})\leq (1/2)^{a(\vec{H})}$, as desired. 
\end{proof}

When applying Lemma~\ref{lem:sandwich}, we will always certify the inequality \eqref{eq:strippingOff} using Lemma~\ref{lem:inOut}. When certifying \eqref{eq:entropyBound}, we apply an entropy-based approach that originated in a paper of Kopparty and Rossman~\cite{KoppartyRossman11} and has been used in, e.g.,~\cite{ChenClemenNoel25+,BehagueMorrisonNoel24,BlekhermanRaymond22,BlekhermanRaymond23,BehagueCrudeleNoelSimbaqueba25,BitontiHoganNoelTsarev26+}. We will now review the basic properties of entropy that we will need. 

Let $X$ be a discrete random variable. Its \emph{range} is
\[
\rng(X):=\{x:\mathbb{P}(X=x)>0\}
\]
and the \emph{entropy} of $X$ is
\[
\mathbb{H}(X):=-\sum_{x\in\rng(X)}\mathbb{P}(X=x)\log_2(\mathbb{P}(X=x)).
\]
A useful way of thinking of the entropy of a random variable is as the average amount of information carried by the random variable $X$, measured in bits. As a simple example, if the range of $X$ has cardinality $2^n$ and all outcomes are equally likely, then $X$ has the same distribution as a uniformly random binary string of length $n$. So, revealing the outcome of $X$ is like revealing $n$ bits of information and so, naturally, the entropy of $X$ is equal to $n$.

If $X$ and $Y$ are discrete random variables and $y\in \rng(Y)$, then the \emph{conditional entropy} of $X$ given that $Y=y$ is defined to be
\[
\mathbb{H}(X\mid Y=y):=-\sum_{x\in \rng(X\mid Y=y)}\mathbb{P}(X=x\mid Y=y)\log_2(\mathbb{P}(X=x\mid Y=y)),
\]
and the \emph{conditional entropy} of $X$ given $Y$ is
\[
\mathbb{H}(X\mid Y):=\sum_{y\in \rng(Y)}\mathbb{P}(Y=y)\mathbb{H}(X\mid Y=y).
\]
We can think of $\mathbb{H}(X\mid Y)$ as the average amount of information that is gained by learning the outcome of $X$ after the outcome of $Y$ is already known. 

We will use the following three standard facts about entropy.

\begin{lem}[Maximality of the uniform distribution]
\label{lem:entropyUniform}
If $X$ is a discrete random variable with finite range, then
\[
\mathbb{H}(X)\leq \log_2(|\rng(X)|),
\]
with equality if and only if $X$ is uniformly distributed on $\rng(X)$.
\end{lem}

\begin{lem}[Chain rule]
\label{lem:entropyChain}
For any discrete random variables $X_1,\dots,X_m$,
\[
\mathbb{H}(X_1,\dots,X_m)=\mathbb{H}(X_1)+\sum_{i=2}^m \mathbb{H}(X_i\mid X_1,\dots,X_{i-1}).
\]
\end{lem}

\begin{lem}[Deconditioning]
\label{lem:entropyDecon}
For any discrete random variables $X,Y$ and $Z$,
\[
\mathbb{H}(X\mid Y,Z)\leq \mathbb{H}(X\mid Z)
\]
with equality if and only if $X$ and $Y$ are conditionally independent given $Z$.
\end{lem}

Our next goal is to obtain a formula for the entropy of a uniformly random homomorphism from an oriented forest. If $\vec{F}$ is an oriented forest and $v\in V(\vec{F})$, then $d_{\vec{F}}(v)$ denotes the number of arcs that are incident with $v$; i.e., it is the degree of $v$ in the unoriented forest underlying $\vec{F}$.

\begin{lem}
\label{lem:forestEntropy}
Let $\vec{F}$ be an oriented forest and let $\vec{G}$ be an oriented graph with $\hom(\vec{F},\vec{G})\geq 1$. Let $\varphi$ be a uniformly random homomorphism from $\vec{F}$ to $\vec{G}$. Then
\[
\mathbb{H}(\varphi)=\sum_{(u,v)\in A(\vec{F})}\mathbb{H}(\varphi(u),\varphi(v))-\sum_{v\in V(\vec{F})}(d_{\vec{F}}(v)-1)\mathbb{H}(\varphi(v)).
\]
\end{lem}

\begin{proof}
Let $\vec{F}_1,\dots,\vec{F}_s$ be the components of $\vec{F}$. The set of homomorphisms from $\vec{F}$ to $\vec{G}$ is the Cartesian product of the sets of homomorphisms from the components $\vec{F}_1,\dots,\vec{F}_s$ to $\vec{G}$. Therefore the restrictions of $\varphi$ to the components of $\vec{F}$ are independent uniformly random homomorphisms. Entropy is additive on independent tuples (by Lemmas~\ref{lem:entropyChain} and~\ref{lem:entropyDecon}), so it suffices to prove the lemma when $\vec{F}$ is connected; i.e. $\vec{F}$ is an oriented tree.

Choose a root $r\in V(\vec{F})$ and order the vertices as $v_1,\dots,v_m$ so that $v_1=r$ and, for each $i\geq 2$, there is a unique index $p(i)\in\{1,\dots,i-1\}$ such that one of the arcs $(v_i,v_{p(i)})$ or $(v_{p(i)},v_i)$ is present in $\vec{F}$; we call $v_{p(i)}$ the \emph{parent} of $v_i$. For each $i\geq 2$, let $\vec{S}_i$ be the subtree of $\vec{F}$ induced by $v_i$ and all of its descendants, and, for each $x\in V(\vec{G})$, let $\hom_{v_i\mapsto x}(\vec{S}_i,\vec{G})$ be the number of homomorphisms from $\vec{S}_i$ to $\vec{G}$ that send $v_i$ to $x$.

Fix $i\geq 2$. Suppose that $x_1,\dots,x_{i-1}\in V(\vec{G})$ such that the event
\[
(\varphi(v_1),\dots,\varphi(v_{i-1}))=(x_1,\dots,x_{i-1})
\]
has positive probability. Since $\vec{F}$ is an oriented tree, for any $x\in V(\vec{G})$,
\[
\mathbb{P}(\varphi(v_i)=x\mid \varphi(v_1)=x_1,\dots,\varphi(v_{i-1})=x_{i-1})
\]
is proportional to the product of the indicator that $x$ and $x_{p(i)}$ form an arc in $\vec{G}$ of the required orientation and the quantity $\hom_{v_i\mapsto x}(\vec{S}_i,\vec{G})$. In particular, the only dependence of $\varphi(v_i)=x$ on the event $(\varphi(v_1),\dots,\varphi(v_{i-1}))=(x_1,\dots,x_{i-1})$ is through $\varphi(v_{p(i)})=x_{p(i)}$. Therefore,
\[
\mathbb{H}(\varphi(v_i)\mid \varphi(v_1),\dots,\varphi(v_{i-1}))
=
\mathbb{H}(\varphi(v_i)\mid \varphi(v_{p(i)})).
\]
By combining the above equality with two applications of Lemma~\ref{lem:entropyChain}, we obtain
\begin{align*}
\mathbb{H}(\varphi)
&=\mathbb{H}(\varphi(v_1),\dots,\varphi(v_m))\\
&=\mathbb{H}(\varphi(v_1))+\sum_{i=2}^m\mathbb{H}(\varphi(v_i)\mid \varphi(v_1),\dots,\varphi(v_{i-1}))\\
&=\mathbb{H}(\varphi(v_1))+\sum_{i=2}^m\mathbb{H}(\varphi(v_i)\mid \varphi(v_{p(i)}))\\
&=\mathbb{H}(\varphi(v_1))+\sum_{i=2}^m\left(\mathbb{H}(\varphi(v_i),\varphi(v_{p(i)}))-\mathbb{H}(\varphi(v_{p(i)}))\right).
\end{align*}
Each edge of the underlying undirected tree of $\vec{F}$ appears exactly once as an arc between a vertex and its parent, and entropy is invariant under permuting the coordinates of a tuple. Hence
\[
\sum_{i=2}^m\mathbb{H}(\varphi(v_i),\varphi(v_{p(i)}))
=
\sum_{(u,v)\in A(\vec{F})}\mathbb{H}(\varphi(u),\varphi(v)).
\]
Moreover, a non-root vertex $v$ is the parent of exactly $d_{\vec{F}}(v)-1$ vertices, while the root is the parent of exactly $d_{\vec{F}}(r)$ vertices. Therefore the coefficient of $\mathbb{H}(\varphi(v))$ in the above expression is $-(d_{\vec{F}}(v)-1)$ for every $v\in V(\vec{F})$, which proves the lemma.
\end{proof}

The next lemma is the main mechanism for establishing the bound \eqref{eq:entropyBound} required for applications of Lemma~\ref{lem:sandwich}. Similar ideas are used in, e.g.,~\cite{BehagueCrudeleNoelSimbaqueba25,BitontiHoganNoelTsarev26+,ChenClemenNoel25+}. 

\begin{lem}
\label{lem:coveringEntropy}
Let $\vec{H}$ and $\vec{F}$ be oriented forests. Suppose that there exists a homomorphism $\phi:\vec{F}\to \vec{H}$ and an integer $c\geq 1$ such that the following hold:
\begin{enumerate}
    \stepcounter{equation}\item \label{eq:Acover} $|\phi^{-1}(u,v)|=c$ for every arc $(u,v)\in A(\vec{H})$,
    \stepcounter{equation}\item \label{eq:Vcover} $|\phi^{-1}(v)|=c$ for every vertex $v\in V(\vec{H})$.
\end{enumerate}
Then every oriented graph $\vec{G}$ satisfies
\[
t(\vec{H},\vec{G})^c\leq t(\vec{F},\vec{G}).
\]
\end{lem}

\begin{proof}
If $\hom(\vec{H},\vec{G})=0$, then the inequality is trivial. So assume that $\hom(\vec{H},\vec{G})\geq 1$, and let $\varphi$ be a uniformly random homomorphism from $\vec{H}$ to $\vec{G}$.

Our goal is to construct a random homomorphism $\theta:\vec{F}\to \vec{G}$ of entropy equal to $c\cdot\mathbb{H}(\varphi)$. Choose a root in each component of $\vec{F}$. For each root $r$, choose $\theta(r)$ according to the distribution of $\varphi(\phi(r))$, independently of all choices made so far. If $x$ is not a root and has parent $y$, then, having already defined $\theta(y)$, choose $\theta(x)$ according to the conditional distribution of $\varphi(\phi(x))$ given the event that $\varphi(\phi(y))=\theta(y)$, independently of all other previously made choices. Since $\phi$ and $\varphi$ are homomorphisms, every pair in the support of $(\varphi(\phi(y)),\varphi(\phi(x)))$ is an arc of $\vec{G}$ with the same direction as the arc between $y$ and $x$. Thus, $\theta$ is a homomorphism from $\vec{F}$ to $\vec{G}$ with probability one.

By induction on the distance from the roots, it can be easily verified that $\theta(u)$ has the same distribution as $\varphi(\phi(u))$ and $(\theta(u),\theta(v))$ has the same distribution as $(\varphi(\phi(u)),\varphi(\phi(v)))$ for every $u\in V(\vec{F})$ and $(u,v)\in A(\vec{F})$. Moreover, for each non-root vertex $u$, the random variable $\theta(u)$ is conditionally independent of all previously exposed variables other than its parent, given the image of its parent. Consequently, the same chain-rule computation used in the proof of Lemma~\ref{lem:forestEntropy} gives
\[
\mathbb{H}(\theta)=\sum_{(u,v)\in A(\vec{F})}\mathbb{H}(\theta(u),\theta(v))-\sum_{u\in V(\vec{F})}(d_{\vec{F}}(u)-1)\mathbb{H}(\theta(u)).
\]
Substituting the distributions of the random variables $\theta(u)$ and then grouping the resulting terms according to $\phi$, we obtain
\begin{align*}
\mathbb{H}(\theta)
&=\sum_{(u,v)\in A(\vec{F})}\mathbb{H}(\varphi(\phi(u)),\varphi(\phi(v)))
-\sum_{u\in V(\vec{F})}(d_{\vec{F}}(u)-1)\mathbb{H}(\varphi(\phi(u)))\\
&=\sum_{(a,b)\in A(\vec{H})}|\phi^{-1}(a,b)|\,\mathbb{H}(\varphi(a),\varphi(b))
-\sum_{v\in V(\vec{H})}\left(\sum_{u\in\phi^{-1}(v)}(d_{\vec{F}}(u)-1)\right)\mathbb{H}(\varphi(v)).
\end{align*}
By \eqref{eq:Acover}, the first term is equal to
\[
c\sum_{(a,b)\in A(\vec{H})}\mathbb{H}(\varphi(a),\varphi(b)).
\]
For the second term, fix $v\in V(\vec{H})$. Since $\phi$ is a homomorphism, every arc of $\vec{F}$ incident with a vertex in $\phi^{-1}(v)$ is mapped to an arc of $\vec{H}$ incident with $v$, and every arc of $\vec{H}$ incident with $v$ has exactly $c$ preimages. Therefore,
\[
\sum_{u\in\phi^{-1}(v)}d_{\vec{F}}(u)=c\,d_{\vec{H}}(v).
\]
Using \eqref{eq:Vcover}, we deduce that
\[
\sum_{u\in\phi^{-1}(v)}(d_{\vec{F}}(u)-1)=c\,d_{\vec{H}}(v)-c=c(d_{\vec{H}}(v)-1).
\]
Hence
\[
\mathbb{H}(\theta)
=
c\sum_{(a,b)\in A(\vec{H})}\mathbb{H}(\varphi(a),\varphi(b))
-
c\sum_{v\in V(\vec{H})}(d_{\vec{H}}(v)-1)\mathbb{H}(\varphi(v)).
\]
Applying Lemma~\ref{lem:forestEntropy} to $\varphi$, we conclude that
\[
\mathbb{H}(\theta)=c\,\mathbb{H}(\varphi)=c\log_2(\hom(\vec{H},\vec{G})).
\]
Since $\theta$ is supported on homomorphisms from $\vec{F}$ to $\vec{G}$, Lemma~\ref{lem:entropyUniform} implies that
\[
\log_2(\hom(\vec{F},\vec{G}))\geq \mathbb{H}(\theta)=c\log_2(\hom(\vec{H},\vec{G})).
\]
Exponentiating gives
\[
\hom(\vec{F},\vec{G})\geq \hom(\vec{H},\vec{G})^c.
\]
Finally, dividing by $v(\vec{G})^{v(\vec{F})}=v(\vec{G})^{c\,v(\vec{H})}$, where the equality follows from \eqref{eq:Vcover}, yields
\[
t(\vec{F},\vec{G})\geq t(\vec{H},\vec{G})^c.
\]
This completes the proof.
\end{proof}

We are now ready to prove the main lemma of the section.

\begin{proof}[Proof of Lemma~\ref{lem:entropyMain}]
If $a(\vec{H})=0$, then the conclusion is immediate. So assume from now on that $a(\vec{H})\geq 1$.

Let $\mathcal{S}=(\vec{D},\psi,\omega,\pi)$ be an $\vec{H}$-sandwich, and let $\mathcal{C}$ be the set of components of $\vec{D}\setminus V(\vec{H})$. By Lemma~\ref{lem:union}, the disjoint union of two tournament anti-Sidorenko oriented graphs is again tournament anti-Sidorenko.

Next, observe that, given the choice of $\vec{D},\psi$ and $\pi$, the conditions imposed on $\omega$ in Definition~\ref{def:sandwich} form a finite system of linear equalities and inequalities with rational coefficients. Since this system is feasible (because $\omega$ exists), it has a rational feasible point. We may therefore assume that $\omega(C)\in\mathbb{Q}$ for every $C\in\mathcal{C}$. Choose a positive integer $q$ such that $q\omega(C)\in\mathbb{Z}$ for every $C\in\mathcal{C}$.

We first construct an oriented forest $\vec{F}$ as follows. Start with one distinguished copy of $\vec{H}$. Then, for each $C\in\mathcal{C}$, add $q\omega(C)$ disjoint copies of $C$ and attach each of them to the distinguished copy of $\vec{H}$ exactly as $C$ is attached to $\vec{H}$ inside $\vec{D}$. Then, for each $v\in V(\vec{H})$, add a set $L_v$ of
\[
q-\sum_{x\in\psi^{-1}(v)}q\omega(x)
\]
isolated vertices associated with $v$. This quantity is a nonnegative integer by \eqref{eq:coveredv}. Let $\vec{F}$ denote the resulting oriented forest, and let
\[
\phi:\vec{F}\to\vec{H}
\]
be the natural homomorphism that fixes the distinguished copy of $\vec{H}$ pointwise, acts as $\psi$ on every non-isolated added copy, and, for each $v\in V(\vec{H})$, maps every vertex of $L_v$ to $v$.

For every arc $(u,v)\in A(\vec{H})$, we have
\[
|\phi^{-1}(u,v)|=1+\sum_{(x,y)\in\psi^{-1}(u,v)}q\omega(x,y)=q+1,
\]
where the initial $1$ comes from the distinguished copy of $\vec{H}$ and the final equality follows from \eqref{eq:coverede}. By construction, every vertex of $\vec{H}$ also has exactly $q+1$ preimages under $\phi$; the cardinality of the sets $L_v$ was chosen precisely to ensure this. Therefore Lemma~\ref{lem:coveringEntropy} gives
\begin{equation}
\label{eq:entropyMainLower}
t(\vec{H},\vec{T})^{q+1}\leq t(\vec{F},\vec{T})
\end{equation}
for every tournament $\vec{T}$. 

So, by Lemma~\ref{lem:coveringEntropy}, to establish \eqref{eq:entropyBound}, it suffices to prove that
\begin{equation}
    \label{eq:q+1}
q+1=\frac{a(\vec{F})}{a(\vec{H})}.
\end{equation}
For each $C\in\mathcal{C}$, let $a^*(C)$ denote the number of arcs of $\vec{D}$ having at least one endpoint in $C$. Since the isolated vertices added at the final step contribute no arcs, we have
\[
a(\vec{F})=a(\vec{H})+\sum_{C\in\mathcal{C}}q\omega(C)\,a^*(C).
\]
On the other hand, summing \eqref{eq:coverede} over all arcs of $\vec{H}$ shows that
\[
\sum_{C\in\mathcal{C}}\omega(C)\,a^*(C)=a(\vec{H}),
\]
because every arc of $\vec{D}$ with at least one endpoint outside $\vec{H}$ contributes exactly the weight of its component to the left-hand side. Hence
\[
a(\vec{F})=a(\vec{H})+q\,a(\vec{H})=(q+1)a(\vec{H}).
\]
So, \eqref{eq:q+1} holds. Therefore, \eqref{eq:entropyBound} holds. 

It remains to prove that \eqref{eq:strippingOff} holds. That is, we prove that every tournament $\vec{T}$ satisfies
\begin{equation}
\label{eq:entropyMainUpperGoal}
t(\vec{F},\vec{T})\leq (1/2)^{a(\vec{F})-a(\vec{H})}t(\vec{H},\vec{T}).
\end{equation}
Fix a tournament $\vec{T}$, and write $n:=v(\vec{T})$. Let $\vec{F}_0$ be the subgraph of $\vec{F}$ obtained by deleting the isolated vertices in the sets $L_v$ for $v\in V(\vec{H})$. Then
\[
\hom(\vec{F},\vec{T})
=
n^{v(\vec{F})-v(\vec{F}_0)}\hom(\vec{F}_0,\vec{T}).
\]

For a component $C\in\mathcal C$, let $\pi(C)$ denote the component of $\vec{D}\setminus V(\vec{H})$ containing $\pi(x)$ for some (equivalently every) vertex $x\in C$. This is well-defined because $\pi$ is a homomorphism from $\vec{D}\setminus V(\vec{H})$ to itself and is an involution. Moreover, $\pi$ restricts to an isomorphism from $\vec{D}[C]$ to $\vec{D}[\pi(C)]$.

Let $\mathcal{C}^*$ be the set of components of $\vec{D}\setminus V(\vec{H})$ that are incident with an arc having one endpoint in $V(\vec{H})$. For $C\in\mathcal C$, write
\[
r_C:=q\omega(C).
\]
By \eqref{eq:sameWeight}, we have $r_C=r_{\pi(C)}$ for every $C\in\mathcal C$.

We first handle the components attached to $\vec{H}$. The components in $\mathcal{C}^*$ are paired by $\pi$. Indeed, if $C$ is attached to $\vec{H}$, then \eqref{eq:reverseArcs} implies that $\pi(C)$ is also attached to $\vec{H}$. Moreover, no component in $\mathcal{C}^*$ is fixed by $\pi$. To see this, suppose that $C=\pi(C)$ and that $C$ is attached through an arc involving a vertex $x\in C$ and a vertex $h\in V(\vec{H})$. Then \eqref{eq:reverseArcs} gives the reversed attachment arc involving $\pi(x)\in C$ and the same vertex $h$. This contradicts \eqref{eq:singleAttachment}. Thus the components in $\mathcal{C}^*$ split into two-element pairs $\{C,\pi(C)\}$.

For each pair, choose one representative $C$ so that the attachment arc from $\vec{H}$ to the pair has the form
\[
(h_C,x_C),
\]
where $h_C\in V(\vec{H})$ and $x_C\in C$. Then \eqref{eq:reverseArcs} says that $\pi(C)$ is attached through $\pi(x_C)$ by the arc
\[
(\pi(x_C),h_C).
\]
Let $\mathcal R$ be the set of these chosen representatives.

We now define an explicit sequence of graphs. Enumerate the multiset
\[
\{(C,s): C\in\mathcal R,\ 1\leq s\leq r_C\}
\]
as
\[
(C_1,s_1),\dots,(C_N,s_N).
\]
Let $\vec{B}_0$ be the distinguished copy of $\vec{H}$. Having defined $\vec{B}_{i-1}$, define $\vec{B}_i$ by adding to $\vec{B}_{i-1}$ one fresh copy of $\vec{D}[C_i]$ and one fresh copy of $\vec{D}[\pi(C_i)]$, attached to the distinguished copy of $\vec{H}$ by the arcs corresponding to
\[
(h_{C_i},x_{C_i})
\qquad\text{and}\qquad
(\pi(x_{C_i}),h_{C_i}).
\]
Since $\pi$ restricts to an isomorphism from $\vec{D}[C_i]$ to $\vec{D}[\pi(C_i)]$, the graph $\vec{B}_i$ is isomorphic to the graph obtained from $\vec{B}_{i-1}$ by applying the construction in Lemma~\ref{lem:inOut} with $u=h_{C_i}$ and with the rooted graph $(\vec{D}[C_i],x_{C_i})$. Therefore Lemma~\ref{lem:inOut} gives
\[
\hom(\vec{B}_i,\vec{T})
\leq
\frac14\hom(\vec{B}_{i-1},\vec{T})\hom(\vec{D}[C_i],\vec{T})^2
=
\frac14\hom(\vec{B}_{i-1},\vec{T})\hom(\vec{D}[C_i],\vec{T})\hom(\vec{D}[\pi(C_i)],\vec{T}).
\]
Iterating over $i=1,\dots,N$ gives
\[
\hom(\vec{B}_N,\vec{T})
\leq
\hom(\vec{H},\vec{T})
\prod_{C\in\mathcal R}
\left(
\frac14\hom(\vec{D}[C],\vec{T})\hom(\vec{D}[\pi(C)],\vec{T})
\right)^{r_C}.
\]

The graph $\vec{B}_N$ is the part of $\vec{F}_0$ consisting of the distinguished copy of $\vec{H}$ together with all copies of components in $\mathcal{C}^*$. The remaining components of $\vec{F}_0$ are disjoint copies of the unattached components. Hence
\[
\hom(\vec{F}_0,\vec{T})
=
\hom(\vec{B}_N,\vec{T})
\prod_{C\in\mathcal C\setminus\mathcal C^*}\hom(\vec{D}[C],\vec{T})^{r_C}.
\]
Combining this with the previous inequality, and using that the two-element pairs $\{C,\pi(C)\}$ partition $\mathcal{C}^*$, gives
\[
\hom(\vec{F}_0,\vec{T})
\leq
\hom(\vec{H},\vec{T})\,
(1/2)^M
\prod_{C\in\mathcal C}\hom(\vec{D}[C],\vec{T})^{r_C},
\]
where
\[
M:=\sum_{C\in\mathcal{C}^*} r_C.
\]
This number \(M\) is exactly the number of arcs of \(\vec{F}_0\) with one endpoint in the distinguished copy of \(V(\vec{H})\) and the other endpoint outside it.

We now use the partition from \eqref{eq:partition}. For each $1\leq j\leq m$, let
\[
\vec{U}_j:=\vec{D}\left[\bigcup_{C\in\mathcal C_j}C\right],
\]
and let $r_j$ be the common value of $r_C=q\omega(C)$ for $C\in\mathcal C_j$. Since $\vec{U}_j$ is the disjoint union of the components $\vec{D}[C]$ with $C\in\mathcal C_j$, Lemma~\ref{lem:union} gives
\[
\prod_{C\in\mathcal C_j}\hom(\vec{D}[C],\vec{T})
=
\hom(\vec{U}_j,\vec{T}).
\]
Therefore
\[
\prod_{C\in\mathcal C}\hom(\vec{D}[C],\vec{T})^{r_C}
=
\prod_{j=1}^m \hom(\vec{U}_j,\vec{T})^{r_j}.
\]
By \eqref{eq:anti-Sidcomponents}, each $\vec{U}_j$ is tournament anti-Sidorenko. Hence
\[
\hom(\vec{U}_j,\vec{T})
\leq
(1/2)^{a(\vec{U}_j)}n^{v(\vec{U}_j)}.
\]
Substituting this into the previous bound gives
\[
\hom(\vec{F}_0,\vec{T})
\leq
\hom(\vec{H},\vec{T})\,
(1/2)^M
\prod_{j=1}^m
\left((1/2)^{a(\vec{U}_j)}n^{v(\vec{U}_j)}\right)^{r_j}.
\]
Since the vertices in $\vec{F}\setminus V(\vec{F}_0)$ are isolated, we get
\[
\hom(\vec{F},\vec{T})
\leq
n^{v(\vec{F})-v(\vec{F}_0)}
\hom(\vec{H},\vec{T})\,
(1/2)^M
\prod_{j=1}^m
\left((1/2)^{a(\vec{U}_j)}n^{v(\vec{U}_j)}\right)^{r_j}.
\]

Finally, observe that
\[
v(\vec{F})-v(\vec{H})
=
v(\vec{F})-v(\vec{F}_0)+\sum_{j=1}^m r_jv(\vec{U}_j),
\]
and
\[
a(\vec{F})-a(\vec{H})
=
M+\sum_{j=1}^m r_ja(\vec{U}_j).
\]
The first identity follows because $\vec{F}_0$ consists of the distinguished copy of $\vec{H}$ together with the non-isolated added components, while $\vec{F}\setminus V(\vec{F}_0)$ consists exactly of the isolated vertices added in the sets $L_v$. The second identity follows because the arcs of $\vec{F}$ outside the distinguished copy of $\vec{H}$ are exactly the internal arcs of the added components together with the \(M\) attachment arcs.

Thus
\[
\hom(\vec{F},\vec{T})
\leq
(1/2)^{a(\vec{F})-a(\vec{H})}
n^{v(\vec{F})-v(\vec{H})}
\hom(\vec{H},\vec{T}),
\]
which is equivalent to \eqref{eq:entropyMainUpperGoal}.

Since $a(\vec{F})=(q+1)a(\vec{H})>a(\vec{H})$, the lower bound \eqref{eq:entropyMainLower} and the upper bound \eqref{eq:entropyMainUpperGoal} together show that $\vec{F}$ satisfies the hypotheses of Lemma~\ref{lem:sandwich}. Therefore Lemma~\ref{lem:sandwich} implies that $\vec{H}$ is tournament anti-Sidorenko.
\end{proof}

\section{Paths With Restricted Block Lengths}
\label{sec:0mod4}

Our goal in this section is to prove Theorem~\ref{th:0mod4} by exhibiting a sandwich for every oriented path $\vec{P}$ in which every block has length at least two and every internal block has length divisible by four. Our strategy is to find $\vec{P}_k$-sandwiches for each $k\in\{2,3,4,5\}$ with additional properties which allow us to ``glue'' them together to yield a sandwich for any such $\vec{P}$. We start by presenting a special $\vec{P}_4$-sandwich that handles the internal blocks of $\vec{P}$.

\begin{constr}
\label{constr:D4*}
There exists a $\vec{P}_4$-sandwich $\mathcal S_4^*=(\vec{D}_4^*,\psi_4^*,\omega_4^*,\pi_4^*)$ such that $\cov_{\mathcal S_4^*}(v_0)=\cov_{\mathcal S_4^*}(v_4)=\frac{1}{2}$.
\end{constr}

\begin{proof}
The oriented tree $\vec{D}_4^*$ is shown in Figure~\ref{fig:D4*} in Appendix~\ref{app:gadget-atlas}. It is obtained from $\vec{P}_4$ by adding
\begin{itemize}
    \item eight vertices $u_0^+,u_1^0,u_2^-,u_3^2,u_1^2,u_2^+,u_3^4$ and $u_4^-$, and
    \item eight arcs $(u_0^+,v_1),(u_0^+,u_1^0),(v_1,u_2^-),(u_2^-,u_3^2),(u_1^2,u_2^+),(u_2^+,v_3),(v_3,u_4^-)$ and $(u_3^4,u_4^-)$. 
\end{itemize}
The homomorphism $\psi_4^*$ is obtained by mapping each vertex of $\vec{D}_4^*$ to the vertex of $\vec{P}_4$ with the same subscript. The weight function $\omega_4^*$ assigns each component of $\vec{D}_4^*\setminus V(\vec{P}_4)$ to weight $1/2$. Finally, $\pi_4^*$ is the involution on $V(\vec{D}_4^*)\setminus V(\vec{P}_4)$ defined by $\pi_4^*(u_0^+)=u_2^-, \pi_4^*(u_1^0)=u_3^2,\pi_4^*(u_1^2)=u_3^4$ and $\pi_4^*(u_2^+)=u_4^-$. It is easily checked that $(\vec{D}_4^*,\psi_4^*,\omega_4^*,\pi_4^*)$ is a $\vec{P}_4$-sandwich and that $\cov_{\mathcal S_4^*}(v_0)=\cov_{\mathcal S_4^*}(v_4)=\frac{1}{2}$. 
\end{proof}

Next, we present $\vec{P}_k$-sandwiches for $k\in\{2,3,5\}$. These are used for handling the extreme ends of the path in the proof of Theorem~\ref{th:0mod4}. Unlike the case $k=4$, for these cases, we are only able to get the weight mapped to one (not two) of the endpoints to be at most $1/2$. 

\begin{constr}
\label{constr:Dk*}
For each $k\in\{2,3,5\}$, there exists a $\vec{P}_k$-sandwich $\mathcal S_k^*=(\vec{D}_k^*,\psi_k^*,\omega_k^*,\pi_k^*)$ 
such that $\cov_{\mathcal S_k^*}(v_0)=\frac{1}{2}$.
\end{constr}

\begin{proof}
First, consider $k=2$. Let $\vec{D}_2^*$ be the oriented forest depicted in Figure~\ref{fig:D2*} in Appendix~\ref{app:gadget-atlas}. It is obtained from $\vec{P}_2$ by 
\begin{itemize}
    \item adding five vertices $u_0,u_1,u_1',u_2,u_2'$
    \item adding four arcs $(u_0,u_1),(u_0,u_1'),(u_1,u_2),(u_1',u_2')$. 
\end{itemize}
The homomorphism $\psi_2^*$ is obtained by mapping each vertex of $\vec{D}_2^*$ to the vertex of $\vec{P}_2$ with the same subscript. The weight function $\omega_2^*$ assigns the unique component of $\vec{D}_2^*\setminus V(\vec{P}_2)$
to weight $1/2$. Finally, $\pi_2^*$ is the identity function on $V(\vec{D}_2^*)\setminus V(\vec{P}_2)$. All of the properties of Definition~\ref{def:sandwich} are easily shown to hold for $(\vec{D}_2^*,\psi_2^*,\omega_2^*,\pi_2^*)$, where \eqref{eq:anti-Sidcomponents} holds by Proposition~\ref{prop:P22}. Clearly, this construction satisfies $\cov_{\mathcal S_2^*}(v_0)=\frac{1}{2}$.

Next, consider $k=3$. Let $\vec{D}_3^*$ be the oriented forest depicted in Figure~\ref{fig:D3*} in Appendix~\ref{app:gadget-atlas}. It is obtained from $\vec{P}_3$ by 
\begin{itemize}
    \item adding ten vertices $u_0^+,u_1^0,u_2^-,u_3^2,u_1^+,u_3^-,w_1,w_2,w_2'$ and $w_3'$, and
    \item adding eight arcs $(u_0^+,v_1),(u_0^+,u_1^0),(v_1,u_2^-),(u_2^-,u_3^2),(u_1^+,v_2),(v_2,u_3^-),(w_1,w_2)$ and $(w_2',w_3')$. 
\end{itemize}
The homomorphism $\psi_3^*$ is obtained by mapping each vertex of $\vec{D}_3^*$ to the vertex of $\vec{P}_3$ with the same subscript. The weight function $\omega_3^*$ assigns weight $1/2$ to the components $\{u_0^+,u_1^0\}$ and $\{u_2^-,u_3^2\}$ of $\vec{D}_3^*\setminus V(\vec{P}_3)$ and weight $1/4$ to all other components. Finally, $\pi_3^*$ is the involution on $V(\vec{D}_3^*)\setminus V(\vec{P}_3)$ which fixes $w_1,w_2,w_2'$ and $w_3'$ and satisfies $\pi_3^*(u_0^+)=u_2^-, \pi_3^*(u_1^0)=u_3^2,\pi_3^*(u_1^+)=u_3^-$. All of the properties of Definition~\ref{def:sandwich} are easily shown to hold for $(\vec{D}_3^*,\psi_3^*,\omega_3^*,\pi_3^*)$. Clearly, this construction satisfies $\cov_{\mathcal S_3^*}(v_0)=\frac{1}{2}$.

Finally, suppose that $k=5$.  Let $\vec{D}_5^*$ be the oriented forest depicted in Figure~\ref{fig:D5*} in Appendix~\ref{app:gadget-atlas}. This construction is rather involved and so we will not describe all of the vertices and arcs added and simply refer to the picture.

The homomorphism $\psi_5^*$ is obtained by mapping each vertex of $\vec{D}_5^*$ to the vertex of $\vec{P}_5$ with the same subscript. The weight function $\omega_5^*$ assigns weight $1/4$ to the components $\{u_0^+,u_1^0\},\{u_2^-,u_3^2\},\{u_1^2,u_2^+\}$ and $\{u_3^4,u_4^-\}$ of $\vec{D}_5^*\setminus V(\vec{P}_5)$ and weight $1/8$ to all other components. Let $\pi_5^*$ be the involution on $V(\vec{D}_5^*)\setminus V(\vec{P}_5)$ which fixes all twelve of the $w$ vertices and satisfies
\[(\pi_5^*(u_0^+),\pi_5^*(u_1^0),\pi_5^*(u_2^+),\pi_5^*(u_1^2)) = (u_2^-,u_3^2,u_4^-,u_3^4)\]
and
\[(\pi_5^*(z_1^+),\pi_5^*(z_2^1),\pi_5^*(z_3^1),\pi_5^*(z_0^+),\pi_5^*(z_1^0),\pi_5^*(z_2^0)) = (z_3^-,z_4^3,z_5^3,z_2^-,z_3^2,z_4^2).\]
Clearly, this construction satisfies $\cov_{\mathcal S_5^*}(v_0)=\frac{1}{2}$. All of the properties of Definition~\ref{def:sandwich}, apart from \eqref{eq:anti-Sidcomponents}, are easily shown to hold for $(\vec{D}_5^*,\psi_5^*,\omega_5^*,\pi_5^*)$. To show that \eqref{eq:anti-Sidcomponents} holds, we observe that $\vec{D}_5^*\setminus V(\vec{P}_5)$ has four 1-vertex components, four 2-vertex components and eight 3-vertex components, four of which are isomorphic to $\vec{P}_2$ and four of which are isomorphic to $\cev{P}_{1,1}$. Partition the components so that each component with one or two vertices is in its own partition class and there are four partition classes consisting of two 3-vertex components each, one of each type. The fact that \eqref{eq:anti-Sidcomponents} holds now follows from Proposition~\ref{prop:forest} and Lemma~\ref{lem:reverseAllArcs}. 
\end{proof}

We now prove Theorem~\ref{th:0mod4}.

\begin{proof}[Proof of Theorem~\ref{th:0mod4}]
Let $\vec{P}$ be an oriented path in which every block has length at least two and every internal block has length divisible by four. We can therefore divide $\vec{P}$ into directed subpaths in which the first and last subpath have length $2,3,4$ or $5$ and all other subpaths have length $4$. 

We use Constructions~\ref{constr:D4*} and~\ref{constr:Dk*} to build a $\vec{P}$-sandwich. 
Let the first and last subpaths have lengths $k$ and $\ell$, respectively. Thus $k,\ell\in\{2,3,4,5\}$. For the first subpath, we use a copy of \(\vec D_k^*\) with the order
of the path labels reversed, so that the low-cover endpoint lies at the
gluing vertex.  For the last subpath, we use a copy of \(\vec D_\ell^*\)
with its usual labeling. For each subpath, if its orientation disagrees with the orientation of the corresponding subpath of $\vec{P}$, then reverse all of its arcs. We then glue together these oriented forests by identifying the vertex $v_k$ from the first subpath with the vertex $v_0$ in the second, and so on.

The functions $\psi,\omega$ and $\pi$ are inherited from the constituent sandwiches in the natural way. Since each constituent is a sandwich, every arc of $\vec{P}$ is covered with total weight $1$. At every glued vertex, the two contributions to $\cov$ are at most $1/2$ each, by Constructions~\ref{constr:D4*} and~\ref{constr:Dk*}. All other vertices of the path are only covered by vertices within the unique sandwich that contains it. Thus the vertex-covering inequality \eqref{eq:coveredv} in Definition~\ref{def:sandwich} holds, and all remaining properties are inherited from the constituent sandwiches. Hence $(\vec{D},\psi,\omega,\pi)$ is a $\vec{P}$-sandwich, and the theorem follows from Lemma~\ref{lem:entropyMain}.
\end{proof}

By combining Theorem~\ref{th:0mod4} with the results of the previous section, we can now complete the proof of Theorem~\ref{th:2blocks}.

\begin{proof}[Proof of Theorem~\ref{th:2blocks}]
Consider the path $\vec{P}_{a,b}$ where $a+b\geq3$. If $a\geq2$ and $b\geq2$, then $\vec{P}_{a,b}$ is tournament anti-Sidorenko by Theorem~\ref{th:0mod4}. 

So, we assume that $a=1$ or $b=1$. By reflecting the path if necessary and then applying
Lemma~\ref{lem:reverseAllArcs}, we may assume that \(a=1\) and \(b\geq2\). If $b\in\{2,3,4\}$, then we are done by Proposition~\ref{prop:P12},~\ref{prop:P13} or~\ref{prop:P14}, respectively. If $b\geq5$, then we are done by Lemma~\ref{lem:reduction} and the fact that we have already settled the cases in which $a,b\geq2$ in the previous paragraph. Thus, the proof is complete. 
\end{proof}

\section{3-Spiders}
\label{sec:spiders}

In this section, we prove that every $3$-spider admits a tournament anti-Sidorenko orientation. The key input is Lemma~\ref{lem:cover1/2} below, which supplies sandwiches for paths in which one prescribed internal vertex is covered by vertices of weight at most $1/2$. We first use this lemma to prove Theorem~\ref{th:3spider}. After that, the remainder of the section is devoted to the certificate constructions needed for Lemma~\ref{lem:cover1/2}. The proof of the lemma is a finite congruence-class analysis; each case is handled by one of the displayed certificates together with the four-block extension lemma below.

\begin{lem}
\label{lem:cover1/2}
Let $k\geq4$ and $2\leq a\leq k-2$. Suppose that one of the following holds: 
\begin{enumerate}
\stepcounter{equation}\item \label{eq:00}$k\equiv 0\bmod 4$ and $a\equiv 0\bmod 4$.
\stepcounter{equation}\item\label{eq:02} $k\equiv 0\bmod 4$ and $a\equiv 2\bmod 4$.
\stepcounter{equation}\item\label{eq:21} $k\equiv 2\bmod 4$ and $a\equiv 1\bmod 4$.
\stepcounter{equation}\item\label{eq:23} $k\equiv 2\bmod 4$ and $a\equiv 3\bmod 4$.
\stepcounter{equation}\item\label{eq:30} $k\equiv 3\bmod 4$ and $a\equiv 0\bmod 4$.
\stepcounter{equation}\item\label{eq:31} $k\equiv 3\bmod 4$ and $a\equiv 1\bmod 4$.
\end{enumerate}
Then there exists $0\leq \ell\leq k-1$ and a $\vec{P}_{k-\ell,\ell}$-sandwich $\mathcal S_k^a=(\vec{D}_{k}^a,\psi_{k}^a,\omega_{k}^a,\pi_{k}^a)$ such that
\begin{equation}\label{eq:covera}\cov_{\mathcal S_k^a}(v_a)\leq\frac{1}{2}.\end{equation}
\end{lem}

\begin{proof}[Proof of Theorem~\ref{th:3spider}]
Consider the $(a,b,c)$-spider where $a,b,c\geq1$. Note that, if any of $a,b$ or $c$ is equal to $1$, then the spider is a caterpillar, and the result follows from~\cite[Theorem~1.4]{HeManiNieTungWei25+}. So, we may assume that $a,b,c\geq2$. 

By the pigeonhole principle, at least two of $a,b,c$ must be congruent to $0$ or $3$ modulo $4$, or at least two of $a,b,c$ must be congruent to $1$ or $2$ modulo $4$. By symmetry, we assume that either $a$ and $b$ are both congruent to $0$ or $3$ or that both of $a$ and $b$ are congruent to $1$ or $2$. Define $k=a+b$. After possibly swapping $a$ and $b$, for each of the six possible pairs of congruence classes of $a$ and $b$, the congruence classes for $k$ and $a$ fall into one of the cases listed in Lemma~\ref{lem:cover1/2}: namely $(a,b)=(0,0)$ gives \eqref{eq:00}, $(0,3)$ gives \eqref{eq:30}, $(3,3)$ gives \eqref{eq:23}, $(1,1)$ gives \eqref{eq:21}, $(1,2)$ gives \eqref{eq:31}, and $(2,2)$ gives \eqref{eq:02}. Thus Lemma~\ref{lem:cover1/2} gives a $\vec{P}_{k-\ell,\ell}$-sandwich $(\vec{D}_k^a,\psi_k^a,\omega_k^a,\pi_k^a)$ satisfying \eqref{eq:covera} for some $0\leq \ell\leq k-1$. Since \(c\geq2\), we obtain a $\vec{P}_c$-sandwich $\mathcal S_c=(\vec{D}_c,\psi_c,\omega_c,\pi_c)$ with the property that $\cov_{\mathcal S_c}(v_0)\leq\frac{1}{2}$ by using one of
Construction~\ref{constr:Dk*}, for \(k\in\{2,3,5\}\), or
Construction~\ref{constr:D4*}, according to the residue of \(c\bmod4\),
and then gluing on copies of Construction~\ref{constr:D4*}. Thus, by taking the disjoint union of $\vec{D}_k^a$ and $\vec{D}_c$ and then identifying vertex $v_a$ from $\vec{D}_k^a$ with $v_0$ from $\vec{D}_c$, we get a sandwich for an orientation of the $(a,b,c)$-spider. Thus, the proof is complete by Lemma~\ref{lem:entropyMain}. 
\end{proof}

From here forward, we focus on the proof of Lemma~\ref{lem:cover1/2}. 

\subsection{Two ``Repeating Block'' Gadgets}

We now turn our attention to proving Lemma~\ref{lem:cover1/2}. We start by presenting a  gadget that is useful for extending existing constructions. 

\begin{constr}
\label{constr:extension}
There exists a partial $\vec{P}_4$-sandwich $\mathcal S_4^\dag=(\vec{D}_4^\dag,\psi_4^\dag,\omega_4^\dag,\pi_4^\dag)$ such that
\[
\cov_{\mathcal S_4^\dag}(v_{i})=
\begin{cases}
1/4&\text{if }i=0,\\
3/4&\text{if }i\in\{1,3,4\},\\
1&\text{if }i=2,
\end{cases}
\]
and
\[
\cov_{\mathcal S_4^\dag}(e_i)=
\begin{cases}
3/4&\text{if }i\in \{0,3\},\\
1&\text{otherwise}.
\end{cases}
\]
\end{constr}

\begin{proof}
The oriented graph $\vec{D}_4^\dag$ is shown in Figure~\ref{fig:extension} in Appendix~\ref{app:gadget-atlas}. Each component of $\vec{D}_4^\dag\setminus V(\vec{P}_4)$ is assigned a weight of $1/4$ by $\omega_4^\dag$. Summing the component weights gives the stated values of $\cov_{\mathcal S_4^\dag}$, and the verification of \eqref{eq:anti-Sidcomponents} uses Proposition~\ref{prop:P22}. 
\end{proof}

The key property of the construction in the previous proposition is that multiple copies of it can be ``chained together'' to get a partial sandwich of a longer path. The next proposition makes this more precise. 

\begin{constr}
\label{constr:extended}
For each $t\geq1$, there exists a partial $\vec{P}_{4t}$-sandwich
\[
\mathcal S_{4t}^\ddag=(\vec{D}_{4t}^\ddag,\psi_{4t}^\ddag,\omega_{4t}^\ddag,\pi_{4t}^\ddag)
\]
such that
\[
\cov_{\mathcal S_{4t}^\ddag}(v_{i})=
\begin{cases}
1/4&\text{if }i=0,\\
3/4&\text{if }i\in\{1,4t-1,4t\},\\
1&\text{otherwise},
\end{cases}
\]
and
\[
\cov_{\mathcal S_{4t}^\ddag}(e_i)=
\begin{cases}
3/4&\text{if }i\in\{0,4t-1\},\\
1&\text{otherwise}.
\end{cases}
\]
\end{constr}

\begin{proof}
The construction of $\vec{D}_{4t}^\ddag$ is as follows. Take $t$ disjoint copies of the oriented forest $\vec{D}_4^\dag$ from Construction~\ref{constr:extension} and identify the vertex corresponding to $v_4$ in each of the first $t-1$ copies with the vertex corresponding to $v_0$ in the next copy. Note that the copies of $\vec{P}_4$ in the copies of $\vec{D}_{4t}^\ddag$ now form a copy of $\vec{P}_{4t}$; relabel these vertices $v_0,\dots,v_{4t}$ in the order that they come on the path. Next, for each $1\leq j\leq t-1$, add two vertices $y_{4j-1}^+$ and $y_{4j+1}^-$ and arcs $(y_{4j-1}^+,v_{4j})$ and $(v_{4j},y_{4j+1}^-)$. See Figure~\ref{fig:extended} in Appendix~\ref{app:gadget-atlas} for a depiction of this in the case $t=3$. All components of $\vec{D}_{4t}^\ddag\setminus V(\vec{P}_{4t})$ are assigned weight $1/4$. The functions $\psi_{4t}^\ddag,\omega_{4t}^\ddag$ and $\pi_{4t}^\ddag$ are inherited from $\psi_4^\dag,\omega_4^\dag$ and $\pi_4^\dag$ in the natural way, where \(\pi_{4t}^\ddag(y_{4j-1}^+)=y_{4j+1}^-\) for all $1\leq j\leq t-1$. 
\end{proof}

The preceding construction will be used to extend existing constructions via the following lemmas.  Throughout the rest of this section, if an oriented path has vertices $v_0,\dots,v_k$, then $e_i$ denotes the arc of the path with endpoints $v_i$ and $v_{i+1}$, with its actual orientation.

\begin{lem}
\label{lem:fourBlockExtensionRight}
Let $\vec{Q}$ be an oriented path with vertices $v_0,\dots,v_k$, and let $\mathcal S=(\vec{D},\psi,\omega,\pi)$ be a partial $\vec{Q}$-sandwich. If
\[
\cov_{\mathcal S}(v_{k-1})\leq 3/4, \quad \cov_{\mathcal S}(v_{k})\leq 3/4, \quad\cov_{\mathcal S}(e_{k-1})=3/4,
\]
then, for all $t\geq0$, there exists a partial sandwich $\mathcal S_{+4t}=(\vec{D}_{+4t},\psi_{+4t},\omega_{+4t},\pi_{+4t})$ for the path obtained from $\vec{Q}$ by extending the right endpoint by a directed block of length $4t$ such that
\[\cov_{\mathcal S_{+4t}}(v_{i})=\cov_{\mathcal S}(v_{i})\text{ for all }0\leq i\leq k-2,\]
\[\cov_{\mathcal S_{+4t}}(e_i)=\cov_{\mathcal S}(e_i)\text{ for all }0\leq i\leq k-2,\]
\[\cov_{\mathcal S_{+4t}}(v_{k-1})=\cov_{\mathcal S}(v_{k-1})+1/4,\quad \cov_{\mathcal S_{+4t}}(v_{k})=\cov_{\mathcal S}(v_{k})+1/4,\quad \cov_{\mathcal S_{+4t}}(e_{k-1})=1,\]
\[\cov_{\mathcal S_{+4t}}(v_{i})=1\text{ for all }k+1\leq i\leq k+4t,\]
and
\[\cov_{\mathcal S_{+4t}}(e_i)=1\text{ for all }k\leq i\leq k+4t-1.\]
\end{lem}

\begin{proof}
Without loss of generality, we can assume $e_{k-1}=(v_{k-1},v_k)$, as the case $e_{k-1}=(v_k,v_{k-1})$ is obtained by reversing all arcs in the construction below. 

First suppose that $t=0$. Let $r_{k-1}$ and $r_k$ be two new vertices, and add the arc $(r_{k-1},r_k)$.
Define
\[
\psi_{+0}(r_{k-1})=v_{k-1}
\qquad\text{and}\qquad
\psi_{+0}(r_k)=v_k.
\]
We assign the component $\{r_{k-1},r_k\}$ to weight $1/4$, and let $\pi_{+0}$ fix both $r_{k-1}$ and $r_k$. On all vertices and components already present in $\vec{D}$, let $\psi_{+0},\omega_{+0}$ and $\pi_{+0}$ agree with $\psi,\omega$ and $\pi$. The new disjoint arc contributes $1/4$ to the cover of $v_{k-1}$, $v_k$ and $e_{k-1}$, and changes no other covers. This proves the result when $t=0$.

Now assume that $t\geq1$. Let $\mathcal S_{4t}^{\ddag}
=
(\vec{D}_{4t}^{\ddag},\psi_{4t}^{\ddag},\omega_{4t}^{\ddag},\pi_{4t}^{\ddag})$ 
be the partial $\vec{P}_{4t}$-sandwich from Construction~\ref{constr:extended}. Form a new oriented graph by taking the disjoint union of $\vec{D}$ and this copy of $\vec{D}_{4t}^{\ddag}$. On the copy of $\vec{D}_{4t}^{\ddag}$, relabel the vertices on its central path by $v_k,v_{k+1},\dots,v_{k+4t}$ and then identify the vertex $v_k$ in $\vec{D}$ with the vertex that is now labeled $v_k$ in $\vec{D}_{4t}^{\ddag}$. 

Next add two new one-vertex components $p_{k-1}$ and $p_{k+1}$, each of weight $1/4$, with
\[
\psi_{+4t}(p_{k-1})=v_{k-1}
\qquad\text{and}\qquad
\psi_{+4t}(p_{k+1})=v_{k+1}.
\]
Add the two attachment arcs
\[
(p_{k-1},v_k)
\qquad\text{and}\qquad
(v_k,p_{k+1}).
\]
Finally, add two new vertices $r_{k+4t-1}$ and $r_{k+4t}$, with
\[
\psi_{+4t}(r_{k+4t-1})=v_{k+4t-1}
\qquad\text{and}\qquad
\psi_{+4t}(r_{k+4t})=v_{k+4t},
\]
and add the disjoint arc
\[
(r_{k+4t-1},r_{k+4t}).
\]
Give the component $\{r_{k+4t-1},r_{k+4t}\}$ weight $1/4$.

We now define the maps and weights. On the original copy of $\vec{D}$, set
\[
\psi_{+4t}=\psi,\qquad \omega_{+4t}=\omega,\qquad \pi_{+4t}=\pi.
\]
On the copy of $\vec{D}_{4t}^{\ddag}$, use $\psi_{4t}^{\ddag},\omega_{4t}^{\ddag}$ and $\pi_{4t}^{\ddag}$, where the domains of these functions are adjusted to account for the fact that the vertices of $\vec{D}_{4t}^{\ddag}$ have been relabeled $v_k,\dots,v_{k+4t}$. On the newly added one-vertex components, set
\[
\pi_{+4t}(p_{k-1})=p_{k+1}
\qquad\text{and}\qquad
\pi_{+4t}(p_{k+1})=p_{k-1}.
\]
On the final disjoint arc component, set
\[
\pi_{+4t}(r_{k+4t-1})=r_{k+4t-1}
\qquad\text{and}\qquad
\pi_{+4t}(r_{k+4t})=r_{k+4t}.
\]

It remains to check the covers. On the old part of the path away from the right endpoint, nothing has changed, so
\[
\cov_{\mathcal S_{+4t}}(v_i)=\cov_{\mathcal S}(v_i)
\quad\text{for all }0\leq i\leq k-2,
\]
and
\[
\cov_{\mathcal S_{+4t}}(e_i)=\cov_{\mathcal S}(e_i)
\quad\text{for all }0\leq i\leq k-2.
\]
The vertex $p_{k-1}$ contributes $1/4$ to the cover of $v_{k-1}$, while the copy of $\vec{D}_{4t}^{\ddag}$ contributes $1/4$ to the cover of $v_k$, since
\[
\cov_{\mathcal S_{4t}^{\ddag}}(v_0)=1/4
\]
by Construction~\ref{constr:extended}. Therefore
\[
\cov_{\mathcal S_{+4t}}(v_{k-1})
=
\cov_{\mathcal S}(v_{k-1})+1/4
\]
and
\[
\cov_{\mathcal S_{+4t}}(v_k)
=
\cov_{\mathcal S}(v_k)+1/4.
\]
The arc $(p_{k-1},v_k)$ maps under $\psi_{+4t}$ to $(v_{k-1},v_k)=e_{k-1}$, so it supplies the missing $1/4$ on $e_{k-1}$. Hence
\[
\cov_{\mathcal S_{+4t}}(e_{k-1})=1.
\]

For the added block, Construction~\ref{constr:extended} gives cover $3/4$ at the vertices $v_{k+1},v_{k+4t-1}$ and $v_{k+4t}$, and cover $1$ at all other new vertices. The component $p_{k+1}$ contributes $1/4$ to $v_{k+1}$, while the final disjoint arc component contributes $1/4$ to both $v_{k+4t-1}$ and $v_{k+4t}$. Thus
\[
\cov_{\mathcal S_{+4t}}(v_i)=1
\quad\text{for all }k+1\leq i\leq k+4t.
\]
Similarly, Construction~\ref{constr:extended} gives cover $3/4$ on the first and last arcs of the added block and cover $1$ on all other arcs of the added block. The arc $(v_k,p_{k+1})$ maps under $\psi_{+4t}$ to $(v_k,v_{k+1})=e_k$, so it supplies the missing $1/4$ on $e_k$. The disjoint arc $(r_{k+4t-1},r_{k+4t})$ maps to $e_{k+4t-1}$, so it supplies the missing $1/4$ on the last arc of the added block. Hence
\[
\cov_{\mathcal S_{+4t}}(e_i)=1
\quad\text{for all }k\leq i\leq k+4t-1.
\]

Finally, all remaining properties of a partial sandwich are inherited from $\mathcal S$ and Construction~\ref{constr:extended}, together with the definitions above. The two one-vertex components $p_{k-1}$ and $p_{k+1}$ are paired by $\pi_{+4t}$, have the same weight, and their attachment arcs are reversed around the path vertex $v_k$. The final two-vertex component is disjoint from the path and is fixed by $\pi_{+4t}$. The new components are either isolated vertices or a directed arc, and hence satisfy the required anti-Sidorenko component condition. Therefore
\[
\mathcal S_{+4t}
=
(\vec{D}_{+4t},\psi_{+4t},\omega_{+4t},\pi_{+4t})
\]
is the desired partial sandwich.
\end{proof}

The next lemma is proved in the same way as Lemma~\ref{lem:fourBlockExtensionRight}, applied at the left endpoint instead of the right endpoint. Equivalently, one may reverse the order of the vertices of the path, apply Lemma~\ref{lem:fourBlockExtensionRight}, and then reverse the order back. 

\begin{lem}
\label{lem:fourBlockExtensionLeft}
Let $\vec{Q}$ be an oriented path with vertices $v_0,\dots,v_k$, and let $\mathcal S=(\vec{D},\psi,\omega,\pi)$ be a partial $\vec{Q}$-sandwich. If
\[
\cov_{\mathcal S}(v_0)\leq 3/4, \quad \cov_{\mathcal S}(v_1)\leq 3/4, \quad\cov_{\mathcal S}(e_0)=3/4,
\]
then, for all $t\geq0$, there exists a partial sandwich $\mathcal S_{-4t}=(\vec{D}_{-4t},\psi_{-4t},\omega_{-4t},\pi_{-4t})$ for the path obtained from $\vec{Q}$ by extending the left endpoint by a directed block of length $4t$ such that, after relabeling the old vertex $v_i$ as $v_{i+4t}$,
\[
\cov_{\mathcal S_{-4t}}(v_i)=1\text{ for all }0\leq i\leq 4t-1,
\]
\[
\cov_{\mathcal S_{-4t}}(e_i)=1\text{ for all }0\leq i\leq 4t-1,
\]
\[
\cov_{\mathcal S_{-4t}}(v_{4t})=\cov_{\mathcal S}(v_0)+1/4,\quad
\cov_{\mathcal S_{-4t}}(v_{4t+1})=\cov_{\mathcal S}(v_1)+1/4,\quad
\cov_{\mathcal S_{-4t}}(e_{4t})=1,
\]
\[
\cov_{\mathcal S_{-4t}}(v_{i+4t})=\cov_{\mathcal S}(v_i)\text{ for all }2\leq i\leq k,
\]
and
\[
\cov_{\mathcal S_{-4t}}(e_{i+4t})=\cov_{\mathcal S}(e_i)\text{ for all }1\leq i\leq k-1.
\]
\end{lem}

The next lemma is obtained by applying Lemmas~\ref{lem:fourBlockExtensionLeft} and~\ref{lem:fourBlockExtensionRight} on the left and right side, respectively.

\begin{lem}
\label{lem:fourBlockExtensionBoth}
Let $\vec{Q}$ be an oriented path with vertices $v_0,\dots,v_k$, and let $\mathcal S=(\vec{D},\psi,\omega,\pi)$ be a partial $\vec{Q}$-sandwich. If
\[
\cov_{\mathcal S}(v_0)\leq 3/4, \quad \cov_{\mathcal S}(v_1)\leq 3/4, \quad\cov_{\mathcal S}(e_0)=3/4,
\]
and
\[
\cov_{\mathcal S}(v_{k-1})\leq 3/4, \quad \cov_{\mathcal S}(v_k)\leq 3/4, \quad\cov_{\mathcal S}(e_{k-1})=3/4,
\]
then, for all $s,t\geq0$, there exists a partial sandwich
\[
\mathcal S_{-4s,+4t}=(\vec{D}_{-4s,+4t},\psi_{-4s,+4t},\omega_{-4s,+4t},\pi_{-4s,+4t})
\]
for the path obtained from $\vec{Q}$ by extending the left endpoint by a directed block of length $4s$ and the right endpoint by a directed block of length $4t$ such that, after relabeling the old vertex $v_i$ as $v_{i+4s}$,
\[
\cov_{\mathcal S_{-4s,+4t}}(v_i)=1\text{ for all }0\leq i\leq 4s-1,
\]
\[
\cov_{\mathcal S_{-4s,+4t}}(e_i)=1\text{ for all }0\leq i\leq 4s-1,
\]
\[
\cov_{\mathcal S_{-4s,+4t}}(v_{4s})=\cov_{\mathcal S}(v_0)+1/4,\quad
\cov_{\mathcal S_{-4s,+4t}}(v_{4s+1})=\cov_{\mathcal S}(v_1)+1/4,\quad
\cov_{\mathcal S_{-4s,+4t}}(e_{4s})=1,
\]
\[
\cov_{\mathcal S_{-4s,+4t}}(v_{i+4s})=\cov_{\mathcal S}(v_i)\text{ for all }2\leq i\leq k-2,
\]
\[
\cov_{\mathcal S_{-4s,+4t}}(e_{i+4s})=\cov_{\mathcal S}(e_i)\text{ for all }1\leq i\leq k-2,
\]
\[
\cov_{\mathcal S_{-4s,+4t}}(v_{k+4s-1})=\cov_{\mathcal S}(v_{k-1})+1/4,\quad
\cov_{\mathcal S_{-4s,+4t}}(v_{k+4s})=\cov_{\mathcal S}(v_k)+1/4,\quad
\cov_{\mathcal S_{-4s,+4t}}(e_{k+4s-1})=1,
\]
\[
\cov_{\mathcal S_{-4s,+4t}}(v_i)=1\text{ for all }k+4s+1\leq i\leq k+4s+4t,
\]
and
\[
\cov_{\mathcal S_{-4s,+4t}}(e_i)=1
\text{ for all }k+4s\leq i\leq k+4s+4t-1.
\]
\end{lem}

Next, we present a second repeating block gadget and then we show how they can be chained together.

\begin{constr}
\label{constr:extension2}
There exists a partial $\cev{P}_4$-sandwich
\[
\mathcal S_4^\mid=(\vec{D}_4^\mid,\psi_4^\mid,\omega_4^\mid,\pi_4^\mid)
\]
such that
\[
\cov_{\mathcal S_4^\mid}(v_{i})=
\begin{cases}
1/5&\text{if }i=0,\\
4/5&\text{if }i=2,\\
3/5&\text{otherwise},
\end{cases}
\]
and
\[
\cov_{\mathcal S_4^\mid}(v_{i+1},v_i)=
\begin{cases}
1&\text{if }i=1,\\
3/5&\text{otherwise}.
\end{cases}
\]
\end{constr}

\begin{proof}
The oriented graph $\vec{D}_4^\mid$ is shown in Figure~\ref{fig:extension2} in Appendix~\ref{app:gadget-atlas}. Each component of $\vec{D}_4^\mid\setminus V(\cev{P}_4)$ is assigned a weight of $1/5$ by $\omega_4^\mid$. Summing the indicated component weights gives the stated values of $\cov$, and the verification of \eqref{eq:anti-Sidcomponents} uses Proposition~\ref{prop:P22}. 
\end{proof}

\begin{constr}
\label{constr:extended2}
For each $t\geq1$, there exists a partial $\cev{P}_{4t}$-sandwich
\[
\mathcal S_{4t}^\parallel=(\vec{D}_{4t}^\parallel,\psi_{4t}^\parallel,\omega_{4t}^\parallel,\pi_{4t}^\parallel)
\]
such that
\[
\cov_{\mathcal S_{4t}^\parallel}(v_{i})=
\begin{cases}
1/5&\text{if }i=0,\\
3/5&\text{if }i\in\{1,4t-1,4t\},\\
4/5&\text{if }i=4t-2,\\
1&\text{otherwise},
\end{cases}
\]
and
\[
\cov_{\mathcal S_{4t}^\parallel}(v_{i+1},v_i)=
\begin{cases}
3/5&\text{if }i=0,\\
1&\text{otherwise}.
\end{cases}
\]
\end{constr}

\begin{proof}
The construction of $\vec{D}_{4t}^\parallel$ is as follows. Take $t$ disjoint copies of the oriented forest $\vec{D}_4^\mid$ from Construction~\ref{constr:extension2} and identify the vertex corresponding to $v_4$ in each of the first $t-1$ copies with the vertex corresponding to $v_0$ in the next copy. Note that the copies of $\cev{P}_4$ in the copies of $\vec{D}_{4t}^\parallel$ now form a copy of $\cev{P}_{4t}$; relabel these vertices $v_0,\dots,v_{4t}$ in the order that they come on the path. Next, for each $1\leq j\leq t-1$, add two vertices $y_{4j-1}^+$ and $y_{4j+1}^-$ and arcs $(y_{4j-1}^+,v_{4j})$ and $(v_{4j},y_{4j+1}^-)$. Also, for each $1\leq j\leq t-1$, add four vertices $z_{4j-2}^+,z_{4j-1}^{4j-2},z_{4j}^-$ and $z_{4j+1}^{4j}$ and four arcs $(v_{4j-1},z_{4j-2}^+),(z_{4j-1}^{4j-2},z_{4j-2}^+),(z_{4j}^-,v_{4j-1})$ and $(z_{4j+1}^{4j},z_{4j}^-)$. See Figure~\ref{fig:extended2} in Appendix~\ref{app:gadget-atlas} for a depiction of this in the case $t=3$. All components of $\vec{D}_{4t}^\parallel\setminus V(\cev{P}_{4t})$ are assigned weight $1/5$. The functions $\psi_{4t}^\parallel,\omega_{4t}^\parallel$ and $\pi_{4t}^\parallel$ are inherited from $\psi_4^\mid,\omega_4^\mid$ and $\pi_4^\mid$ in the natural way, where 
\[
\pi_{4t}^\parallel(z_{4j-2}^+)=z_{4j}^-,
\qquad
\pi_{4t}^\parallel(z_{4j-1}^{4j-2})=z_{4j+1}^{4j}
\]
for all $1\leq j\leq t-1$. 
\end{proof}

In the next six subsections, we prove the six cases of Lemma~\ref{lem:cover1/2}, one by one.

\subsection[Case 00 of Lemma~\ref{lem:cover1/2}]{Case \eqref{eq:00} of Lemma~\ref{lem:cover1/2}}

We consider the first case of Lemma~\ref{lem:cover1/2}: $k\equiv 0\bmod 4$ and $a\equiv 0\bmod 4$. We exhibit a partial $\vec{P}_8$-sandwich which is used with Lemma~\ref{lem:fourBlockExtensionBoth} to deal with this case. Recall that, throughout this section, $e_i$ is the arc between $v_i$ and $v_{i+1}$, in one direction or the other.

\begin{constr}
\label{constr:0mod4 0mod4}
There exists a partial $\vec{P}_8$-sandwich
\[
\mathcal S_8^\S=(\vec{D}_8^\S,\psi_8^\S,\omega_8^\S,\pi_8^\S)
\]
such that
\[
\cov_{\mathcal S_8^\S}(v_{i})=
\begin{cases}
1/2&\text{if }i=4,\\
3/4&\text{if }i\in\{0,1,7,8\},\\
1&\text{otherwise},
\end{cases}
\]
and
\[
\cov_{\mathcal S_8^\S}(e_i)=
\begin{cases}
3/4&\text{if }i\in \{0,7\},\\
1&\text{otherwise}.
\end{cases}
\]
\end{constr}

\begin{proof}
The oriented graph $\vec{D}_8^\S$ is shown in Figure~\ref{fig:0mod4 0mod4} in Appendix~\ref{app:gadget-atlas}. Each component of $\vec{D}_8^\S\setminus V(\vec{P}_8)$ is assigned a weight of $1/4$ by $\omega_8^\S$. Summing the indicated component weights gives the stated values of $\cov$, and the verification of \eqref{eq:anti-Sidcomponents} uses Proposition~\ref{prop:P12}. 
\end{proof}

\begin{proof}[Proof of Lemma~\ref{lem:cover1/2} \eqref{eq:00}]
Let $k$ and $a$ be such that $2\leq a\leq k-2$, $k\equiv0\bmod4$ and $a\equiv0\bmod4$. Thus $4\leq a\leq k-4$ and $k\geq8$. Start with the partial sandwich $\mathcal S_8^\S$ from Construction~\ref{constr:0mod4 0mod4}. Its distinguished vertex is $v_4$, and it has the endpoint cover pattern required by Lemma~\ref{lem:fourBlockExtensionBoth} at both endpoints.

Extend this certificate to the left by $a-4$ arcs and to the right by $k-a-4$ arcs using Lemma~\ref{lem:fourBlockExtensionBoth}. After relabeling the central path as $v_0,\dots,v_k$, the vertex $v_4$ of the original copy of $\vec{D}_8^\S$ becomes $v_a$. The extension lemma shows that all arc covers are equal to $1$ and all vertex covers are at most $1$, while the cover of this distinguished vertex remains $1/2$. Thus the resulting tuple is the desired sandwich satisfying \eqref{eq:covera}.
\end{proof}

\subsection[Case 02 of Lemma~\ref{lem:cover1/2}]{Case \eqref{eq:02} of Lemma~\ref{lem:cover1/2}}

Next, we tackle the second case of Lemma~\ref{lem:cover1/2}: $k\equiv 0\bmod 4$ and $a\equiv 2\bmod 4$. For this, we require separate constructions for the case that $(k,a)=(4,2)$, that $a\in\{2,k-2\}$ and $k\geq8$, or that $a\notin\{2,k-2\}$. 

\begin{constr}
\label{constr:0mod4 2mod4 4}
There exists a partial $\vec{P}_{1,3}$-sandwich
\[
\mathcal S_{1,3}^\S=(\vec{D}_{1,3}^\S,\psi_{1,3}^\S,\omega_{1,3}^\S,\pi_{1,3}^\S)
\]
such that
\[
\cov_{\mathcal S_{1,3}^\S}(v_{i})=
\begin{cases}
0&\text{if }i=2,\\
1&\text{otherwise},
\end{cases}
\]
and $\cov_{\mathcal S_{1,3}^\S}(e_i)=1$ for all $0\leq i\leq3$.
\end{constr}

\begin{proof}
The oriented graph $\vec{D}_{1,3}^\S$ is shown in Figure~\ref{fig:0mod4 2mod4 4} in Appendix~\ref{app:gadget-atlas}. Each component of $\vec{D}_{1,3}^\S\setminus V(\vec{P}_{1,3})$ is assigned a weight of $1$ by $\omega_{1,3}^\S$. Summing the indicated component weights gives the stated values of $\cov$.
\end{proof}

\begin{constr}
\label{constr:0mod4 2mod4 8}
There exists a partial $\vec{P}_{3,5}$-sandwich
\[
\mathcal S_{3,5}^\S=(\vec{D}_{3,5}^\S,\psi_{3,5}^\S,\omega_{3,5}^\S,\pi_{3,5}^\S)
\]
such that
\[
\cov_{\mathcal S_{3,5}^\S}(v_{i})=
\begin{cases}
3/8&\text{if }i=2,\\
3/4&\text{if }i\in\{7,8\},\\
1&\text{otherwise},
\end{cases}
\]
and
\[
\cov_{\mathcal S_{3,5}^\S}(e_i)=
\begin{cases}
3/4&\text{if }i=7,\\
1&\text{otherwise}.
\end{cases}
\]
\end{constr}

\begin{proof}
The oriented graph $\vec{D}_{3,5}^\S$ is shown in Figure~\ref{fig:0mod4 2mod4 8} in Appendix~\ref{app:gadget-atlas}. Each component of $\vec{D}_{3,5}^\S\setminus V(\vec{P}_{3,5})$ is assigned a weight of $1/8,2/8,3/8$ or $5/8$ by $\omega_{3,5}^\S$, as indicated in the figure. Summing the indicated component weights gives the stated values of $\cov$ and the verification of \eqref{eq:anti-Sidcomponents} uses Proposition~\ref{prop:P12} and Proposition~\ref{prop:P22}.
\end{proof}

\begin{constr}
\label{constr:0mod4 2mod4 general}
Let $a\geq6$ such that $a\equiv 2\bmod 4$. There exists a partial $\vec{P}_{a-1,a+1}$-sandwich
\[
\mathcal S_{a-1,a+1}^\S=(\vec{D}_{a-1,a+1}^\S,\psi_{a-1,a+1}^\S,\omega_{a-1,a+1}^\S,\pi_{a-1,a+1}^\S)
\]
such that
\[
\cov_{\mathcal S_{a-1,a+1}^\S}(v_{i})=
\begin{cases}
3/4&\text{if }i\in\{0,1\},\\
1/2&\text{if }i=a,\\
1&\text{otherwise},
\end{cases}
\]
and
\[
\cov_{\mathcal S_{a-1,a+1}^\S}(e_i)=
\begin{cases}
3/4&\text{if }i=0,\\
1&\text{otherwise}.
\end{cases}
\]
\end{constr}

\begin{proof}
Write $a=4t+2$, where $t\geq 1$. The construction is shown in Figure~\ref{fig:0mod4 2mod4 general} in Appendix~\ref{app:gadget-atlas} in the case $a=10$ (i.e. $t=2$). The general construction is obtained from the same diagram by extending the repeated $4$-edge gadgets on the left and on the right. More precisely, the left side consists of the initial two one-vertex components of weight $5/54$ and a disjoint arc of weight $1/36$, followed by $t$ copies of the $\vec{D}_4^\dag$-type gadget used in Construction~\ref{constr:extension}, each component of which has weight $5/54$, linked by the one-vertex connector components of weight $5/54$, as in the proof of Construction~\ref{constr:extension}. The right side is obtained analogously using $\vec{D}_4^\mid$-type gadgets similar to the proof of Construction~\ref{constr:extension2}, where every component has weight $4/54$, and there is a disjoint copy of $\vec{P}_{2,2}$ of weight $4/54$ at the end. The central part of the construction is the part of the diagram around the columns $a-2,a-1,a,a+1,a+2,a+3$ depicted in blue in the diagram; it is unchanged, up to translating the indices. Finally, there are two copies of $\vec{P}_{a-1,a-1}$ where the first copy has its leaves in column $0$ and centre vertex in column $a-1$ and the second has its leaves in column $2a$ and centre vertex in column $a+1$. There is an arc from $v_a$ to the centre vertex of the first copy and an arc from the centre vertex of the second copy to $v_a$. These paths receive weight $17/54$ and arcs are depicted in red in the figure. 

Let $\vec{D}_{a-1,a+1}^\S$ be the resulting oriented forest. We define $\psi_{a-1,a+1}^\S:\vec{D}_{a-1,a+1}^\S\to \vec{P}_{a-1,a+1}$ by sending every vertex drawn in column $i$ to $v_i$; in particular, $\psi_{a-1,a+1}^\S$ fixes the copy of $\vec{P}_{a-1,a+1}$. The weight function $\omega_{a-1,a+1}^\S$ assigns to each component of $\vec{D}_{a-1,a+1}^\S\setminus V(\vec{P}_{a-1,a+1})$ the weight indicated on that component in the figure. The involution $\pi_{a-1,a+1}^\S$ pairs the components in the evident way: each component attached to the path by an outgoing arc is paired with an isomorphic component of the same weight attached by the reversed incoming arc.

It is immediate from the construction that $\psi_{a-1,a+1}^\S$ is a homomorphism and that $\pi_{a-1,a+1}^\S$ satisfies the required reversal condition for arcs between the path and the components. Also, every component of $\vec{D}_{a-1,a+1}^\S\setminus V(\vec{P}_{a-1,a+1})$ has at most one arc to the path. The verification of \eqref{eq:anti-Sidcomponents} uses Proposition~\ref{prop:P12}, Proposition~\ref{prop:P22}, and Lemma~\ref{lem:union}, exactly as in the preceding constructions. It also uses Theorem~\ref{th:2blocks}, which ensures that $\vec{P}_{a-1,a-1}$ is anti-Sidorenko. 

Summing the indicated weights column by column gives the stated values of $\cov$.  The repeated $4$-edge gadgets contribute total weight $1$ to every internal vertex and arc in their range, the central components contribute total vertex-cover $1/2$ in column $a$, and the red rail components supply the missing cover near the two ends.  This proves that $\mathcal S_{a-1,a+1}^\S$ is the desired partial $\vec{P}_{a-1,a+1}$-sandwich.
\end{proof}

Using the constructions described in this section, we prove Lemma~\ref{lem:cover1/2} \eqref{eq:02}. 

\begin{proof}[Proof of Lemma~\ref{lem:cover1/2} \eqref{eq:02}]
Let $k\equiv0\bmod4$ and $a\equiv2\bmod4$, where $2\leq a\leq k-2$. If $k=4$ and $a=2$, then Construction~\ref{constr:0mod4 2mod4 4} gives the desired sandwich.

Suppose next that one of $a$ and $k-a$ is equal to $2$. By reversing the labels on the central path if necessary, we may assume that $a=2$. Start with the partial sandwich $\mathcal S_{3,5}^\S$ from Construction~\ref{constr:0mod4 2mod4 8}. The distinguished vertex is $v_2$, and the right endpoint has the cover pattern required by Lemma~\ref{lem:fourBlockExtensionRight}. Extend to the right by $k-8$ arcs using Lemma~\ref{lem:fourBlockExtensionRight}. The distinguished vertex remains $v_2$, so \eqref{eq:covera} holds.

It remains to consider the case $6\leq a\leq k-6$. By reversing the labels on the central path if necessary, we may assume that $a\geq k/2$. Put $b:=k-a$. Then $b\equiv2\bmod4$ and $b\geq6$. Start with the partial sandwich $\mathcal S_{b-1,b+1}^\S$ from Construction~\ref{constr:0mod4 2mod4 general}; its distinguished vertex is the vertex in column $b$. Extend this certificate to the left by $a-b$ arcs using Lemma~\ref{lem:fourBlockExtensionLeft}. After relabeling, the distinguished vertex is $v_a$. The extension lemma preserves the stated cover at the distinguished vertex and fills all newly created arcs, so the resulting tuple is the desired sandwich satisfying \eqref{eq:covera}.
\end{proof}

\subsection[Case 21 of Lemma~\ref{lem:cover1/2}]{Case \eqref{eq:21} of Lemma~\ref{lem:cover1/2}}

Next, we obtain a construction that is useful for the third case of Lemma~\ref{lem:cover1/2}.

\begin{constr}
\label{constr:2mod4 1mod4}
There exists a partial $\vec{P}_{10}$-sandwich
\[
\mathcal S_{10}^\S=(\vec{D}_{10}^\S,\psi_{10}^\S,\omega_{10}^\S,\pi_{10}^\S)
\]
such that
\[
\cov_{\mathcal S_{10}^\S}(v_{i})=
\begin{cases}
1/2&\text{if }i=5,\\
3/4&\text{if }i\in\{0,1,9,10\},\\
1&\text{otherwise},
\end{cases}
\]
and
\[
\cov_{\mathcal S_{10}^\S}(e_i)=
\begin{cases}
3/4&\text{if }i\in \{0,9\},\\
1&\text{otherwise}.
\end{cases}
\]
\end{constr}

\begin{proof}
The oriented graph $\vec{D}_{10}^\S$ is shown in Figure~\ref{fig:2mod4 1mod4} in Appendix~\ref{app:gadget-atlas}. Each component of $\vec{D}_{10}^\S\setminus V(\vec{P}_{10})$ is assigned a weight of $1/4$ by $\omega_{10}^\S$ except for $\{y_4^+\}$ and $\{y_6^-\}$ which have weight $1/2$. Summing the indicated component weights gives the stated values of $\cov$ and the verification of \eqref{eq:anti-Sidcomponents} uses Proposition~\ref{prop:P12}.
\end{proof}

\begin{proof}[Proof of Lemma~\ref{lem:cover1/2} \eqref{eq:21}]
Let $k\equiv2\bmod4$ and $a\equiv1\bmod4$, where $2\leq a\leq k-2$. Then $a\geq5$ and $k-a\geq5$. Start with the partial sandwich $\mathcal S_{10}^\S$ from Construction~\ref{constr:2mod4 1mod4}. Its distinguished vertex is $v_5$, and it has the endpoint cover pattern required by Lemma~\ref{lem:fourBlockExtensionBoth} at both endpoints.

Extend this certificate to the left by $a-5$ arcs and to the right by $k-a-5$ arcs using Lemma~\ref{lem:fourBlockExtensionBoth}. After relabeling the central path as $v_0,\dots,v_k$, the vertex $v_5$ of the original copy of $\vec{D}_{10}^\S$ becomes $v_a$. Thus the distinguished vertex has cover $1/2$, all arc covers are equal to $1$, and all vertex covers are at most $1$. This gives the desired sandwich.
\end{proof}

\subsection[Case 23 of Lemma~\ref{lem:cover1/2}]{Case \eqref{eq:23} of Lemma~\ref{lem:cover1/2}}

Next, we present two constructions that are used to establish the fourth case of Lemma~\ref{lem:cover1/2}.

\begin{constr}
\label{constr:2mod4 3mod4 6}
There exists a partial $\vec{P}_{2,4}$-sandwich
\[
\mathcal S_{2,4}^\S=(\vec{D}_{2,4}^\S,\psi_{2,4}^\S,\omega_{2,4}^\S,\pi_{2,4}^\S)
\]
such that
\[
\cov_{\mathcal S_{2,4}^\S}(v_{i})=
\begin{cases}
1/2&\text{if }i=3,\\
3/4&\text{if }i\in\{0,1\},\\
1&\text{otherwise},
\end{cases}
\]
and
\[
\cov_{\mathcal S_{2,4}^\S}(e_i)=
\begin{cases}
3/4&\text{if }i=0,\\
1&\text{otherwise}.
\end{cases}
\]
\end{constr}

\begin{proof}
The oriented graph $\vec{D}_{2,4}^\S$ is shown in Figure~\ref{fig:2mod4 3mod4 6} in Appendix~\ref{app:gadget-atlas}. Each component of $\vec{D}_{2,4}^\S\setminus V(\vec{P}_{2,4})$ is assigned a weight of $1/4$ by $\omega_{2,4}^\S$. Summing the indicated component weights gives the stated values of $\cov$ and the verification of \eqref{eq:anti-Sidcomponents} uses  Proposition~\ref{prop:P22}. 
\end{proof}

\begin{constr}
\label{constr:2mod4 3mod4 14}
There exists a partial $\vec{P}_{4,4,6}$-sandwich
\[
\mathcal S_{4,4,6}^\S=(\vec{D}_{4,4,6}^\S,\psi_{4,4,6}^\S,\omega_{4,4,6}^\S,\pi_{4,4,6}^\S)
\]
such that
\[
\cov_{\mathcal S_{4,4,6}^\S}(v_{i})=
\begin{cases}
1/2&\text{if }i=7,\\
3/4&\text{if }i\in\{0,1,13,14\},\\
1&\text{otherwise},
\end{cases}
\]
and
\[
\cov_{\mathcal S_{4,4,6}^\S}(e_i)=
\begin{cases}
3/4&\text{if }i\in\{0,13\},\\
1&\text{otherwise}.
\end{cases}
\]
\end{constr}

\begin{proof}
The oriented graph $\vec{D}_{4,4,6}^\S$ is shown in Figure~\ref{fig:2mod4 3mod4 14} in Appendix~\ref{app:gadget-atlas}. Each component of $\vec{D}_{4,4,6}^\S\setminus V(\vec{P}_{4,4,6})$ is assigned a weight of $1/16,2/16$ or $4/16$ by $\omega_{4,4,6}^\S$, as indicated in the figure. Summing the indicated component weights gives the stated values of $\cov$ and the verification of \eqref{eq:anti-Sidcomponents} uses Propositions~\ref{prop:P12} and~\ref{prop:P22}. 
\end{proof}

\begin{proof}[Proof of Lemma~\ref{lem:cover1/2} \eqref{eq:23}]
Let $k\equiv2\bmod4$ and $a\equiv3\bmod4$, where $2\leq a\leq k-2$. Then $k-a\equiv3\bmod4$ as well.

First suppose that one of $a$ and $k-a$ is equal to $3$. By reversing the labels on the central path if necessary, we may assume that $k-a=3$. Start with the partial sandwich $\mathcal S_{2,4}^\S$ from Construction~\ref{constr:2mod4 3mod4 6}, relabeled so that its distinguished vertex $v_3$ becomes $v_a$. Extend to the left by \(a-3\) arcs using Lemma~\ref{lem:fourBlockExtensionLeft}. This gives the desired sandwich.

Now suppose that $a\geq7$ and $k-a\geq7$. Start with the partial sandwich $\mathcal S_{4,4,6}^\S$ from Construction~\ref{constr:2mod4 3mod4 14}, relabeled so that its distinguished vertex $v_7$ becomes $v_a$. Extend to the left by $a-7$ arcs and to the right by $k-a-7$ arcs using Lemma~\ref{lem:fourBlockExtensionBoth}. The extension lemma preserves the cover $1/2$ at the distinguished vertex and fills all endpoint and gluing deficits. Hence the resulting tuple satisfies \eqref{eq:covera}.
\end{proof}

\subsection[Case 30 of Lemma~\ref{lem:cover1/2}]{Case \eqref{eq:30} of Lemma~\ref{lem:cover1/2}}

We present two constructions which are used to establish the fifth case of Lemma~\ref{lem:cover1/2}.

\begin{constr}
\label{constr:3mod4 0mod4 7}
There exists a partial $\vec{P}_{3,4}$-sandwich
\[
\mathcal S_{3,4}^\S=(\vec{D}_{3,4}^\S,\psi_{3,4}^\S,\omega_{3,4}^\S,\pi_{3,4}^\S)
\]
such that
\[
\cov_{\mathcal S_{3,4}^\S}(v_{i})=
\begin{cases}
1/2&\text{if }i=4,\\
3/4&\text{if }i\in\{0,1\},\\
15/16&\text{if }i=7,\\
1&\text{otherwise},
\end{cases}
\]
and
\[
\cov_{\mathcal S_{3,4}^\S}(e_i)=
\begin{cases}
3/4&\text{if }i=0,\\
1&\text{otherwise}.
\end{cases}
\]
\end{constr}

\begin{proof}
The oriented graph $\vec{D}_{3,4}^\S$ is shown in Figure~\ref{fig:3mod4 0mod4 7} in Appendix~\ref{app:gadget-atlas}. Each component of $\vec{D}_{3,4}^\S\setminus V(\vec{P}_{3,4})$ is assigned a weight of $1/32,2/32,4/32,5/32,8/32$ or $10/32$ by $\omega_{3,4}^\S$, as indicated in the figure. Summing the indicated component weights gives the stated values of $\cov$ and, in the verification of \eqref{eq:anti-Sidcomponents}, we partition the components so that each $2$-arc path with two blocks with another $2$-arc path with one block which has the same weight and apply Proposition~\ref{prop:forest}. 
\end{proof}

\begin{constr}
\label{constr:3mod4 0mod4 11}
There exists a partial $\vec{P}_{3,8}$-sandwich
\[
\mathcal S_{3,8}^\S=(\vec{D}_{3,8}^\S,\psi_{3,8}^\S,\omega_{3,8}^\S,\pi_{3,8}^\S)
\]
such that
\[
\cov_{\mathcal S_{3,8}^\S}(v_{i})=
\begin{cases}
1/2&\text{if }i=4,\\
3/4&\text{if }i\in\{0,1,10,11\},\\
1&\text{otherwise},
\end{cases}
\]
and
\[
\cov_{\mathcal S_{3,8}^\S}(e_i)=
\begin{cases}
3/4&\text{if }i\in\{0,10\},\\
1&\text{otherwise}.
\end{cases}
\]
\end{constr}

\begin{proof}
The oriented graph $\vec{D}_{3,8}^\S$ is shown in Figure~\ref{fig:3mod4 0mod4 11} in Appendix~\ref{app:gadget-atlas}. Each component of $\vec{D}_{3,8}^\S\setminus V(\vec{P}_{3,8})$ is assigned a weight of $1/4$ or $1/8$ by $\omega_{3,8}^\S$, as indicated in the figure. Summing the indicated component weights gives the stated values of $\cov$ and, as in the proof of Construction~\ref{constr:extension}, the verification of \eqref{eq:anti-Sidcomponents} uses Proposition~\ref{prop:P22} and the fact that the path consisting of a block of length one and a block of length two is tournament anti-Sidorenko. 
\end{proof}

\begin{proof}[Proof of Lemma~\ref{lem:cover1/2} \eqref{eq:30}]
Let $k\equiv3\bmod4$ and $a\equiv0\bmod4$, where $2\leq a\leq k-2$. Thus $a\geq4$.

First suppose that $a=k-3$. Start with the partial sandwich $\mathcal S_{3,4}^\S$ from Construction~\ref{constr:3mod4 0mod4 7}, relabeled so that its distinguished vertex $v_4$ becomes $v_a$. Extend to the left by $a-4$ arcs using Lemma~\ref{lem:fourBlockExtensionLeft}. The distinguished vertex has cover $1/2$, so \eqref{eq:covera} holds.

Now suppose that $a\leq k-7$. Start with the partial sandwich $\mathcal S_{3,8}^\S$ from Construction~\ref{constr:3mod4 0mod4 11}, relabeled so that its distinguished vertex $v_4$ becomes $v_a$. Extend to the left by $a-4$ arcs and to the right by $k-a-7$ arcs using Lemma~\ref{lem:fourBlockExtensionBoth}. Again the distinguished vertex has cover $1/2$, and the extension lemma gives the required arc and vertex covers everywhere else.
\end{proof}

\subsection[Case 31 of Lemma~\ref{lem:cover1/2}]{Case \eqref{eq:31} of Lemma~\ref{lem:cover1/2}}

Next, we obtain a construction that is useful for the sixth case of Lemma~\ref{lem:cover1/2}.

\begin{constr}
\label{constr:3mod4 1mod4}
There exists a partial $\vec{P}_{4,3}$-sandwich
\[
\mathcal S_{4,3}^\S=(\vec{D}_{4,3}^\S,\psi_{4,3}^\S,\omega_{4,3}^\S,\pi_{4,3}^\S)
\]
such that
\[
\cov_{\mathcal S_{4,3}^\S}(v_{i})=
\begin{cases}
1/2&\text{if }i=5,\\
3/4&\text{if }i\in\{0,1,6,7\},\\
1&\text{otherwise},
\end{cases}
\]
and
\[
\cov_{\mathcal S_{4,3}^\S}(e_i)=
\begin{cases}
3/4&\text{if }i\in\{0,6\},\\
1&\text{otherwise}.
\end{cases}
\]
\end{constr}

\begin{proof}
The oriented graph $\vec{D}_{4,3}^\S$ is shown in Figure~\ref{fig:3mod4 1mod4} in Appendix~\ref{app:gadget-atlas}. Each component of $\vec{D}_{4,3}^\S\setminus V(\vec{P}_{4,3})$ is assigned a weight of $1/4$ by $\omega_{4,3}^\S$. Summing the indicated component weights gives the stated values of $\cov$ and, as in the proof of Construction~\ref{constr:extension}, the verification of \eqref{eq:anti-Sidcomponents} uses Proposition~\ref{prop:P22}. 
\end{proof}

\begin{proof}[Proof of Lemma~\ref{lem:cover1/2} \eqref{eq:31}]
Let $k\equiv3\bmod4$ and $a\equiv1\bmod4$, where $2\leq a\leq k-2$. Then $a\geq5$ and $k-a\geq2$. Start with the partial sandwich $\mathcal S_{4,3}^\S$ from Construction~\ref{constr:3mod4 1mod4}. Its distinguished vertex is $v_5$, and it has the endpoint cover pattern required by Lemma~\ref{lem:fourBlockExtensionBoth} at both endpoints.

Extend this certificate to the left by $a-5$ arcs and to the right by $k-a-2$ arcs using Lemma~\ref{lem:fourBlockExtensionBoth}. After relabeling, the distinguished vertex is $v_a$, and its cover is still $1/2$. Thus we obtain the desired sandwich satisfying \eqref{eq:covera}.
\end{proof}

\section{Conclusion}
\label{sec:concl}

We mention some related independent work. During discussions on Sidorenko and anti-Sidorenko phenomena while Kr\'a\v{l} was visiting the University of Victoria, we learned that Kr\'a\v{l}, Ku\v{c}er\'ak and Lidick\'y~\cite{KralKucerakLidickyPrivate} have obtained further examples of tournament anti-Sidorenko orientations of trees by different methods, based on Taylor-type expansions and the analysis of subgraph counts. In particular, their arguments give a tournament anti-Sidorenko orientation of the $(2,3,4)$-spider, which is a special case of Theorem~\ref{th:3spider}, and also show that the canonical orientation of the $(2,2,2,2)$-spider obtained as the union of two directed paths is tournament anti-Sidorenko. The latter clarifies the scope of Remark~5.3 in~\cite{FoxHimwichManiZhou24+}; it seems that this statement should be interpreted as applying only for sufficiently large $k$.

\begin{ai}
Between March 18 and May 12, 2026, a model of ChatGPT in Pro mode assisted with rewriting and revising proofs, proofreading, generating figures, and checking proofs and calculations. The exact model version was not recorded. All proofs and calculations that were written or checked by AI were independently verified by the authors. The main mathematical ideas and proof strategies originated with the authors, who take full responsibility for this work.
\end{ai}


\begin{thebibliography}{10}

\bibitem{Basit+25+}
A.~Basit, B.~Granet, D.~Horsley, A.~K\"undgen, and K.~Staden.
\newblock The semi-inducibility problem.
\newblock E-print arXiv:2501.09842v2, 2025.

\bibitem{BehagueCrudeleNoelSimbaqueba25}
N.~Behague, G.~Crudele, J.~A. Noel, and L.~M. Simbaqueba.
\newblock Sidorenko-{T}ype {I}nequalities for {P}airs of {T}rees.
\newblock {\em Random Structures Algorithms}, 67(1):Paper No. e70026, 2025.

\bibitem{BehagueMorrisonNoel24}
N.~Behague, N.~Morrison, and J.~A. Noel.
\newblock Off-diagonal commonality of graphs via entropy.
\newblock {\em SIAM J. Discrete Math.}, 38(3):2335--2360, 2024.

\bibitem{BitontiHoganNoelTsarev26+}
V.~Bitonti, E.~Hogan, J.~A. Noel, and D.~Tsarev.
\newblock Relative {S}idorenko inequalities in oriented graphs.
\newblock In preparation, 2026.

\bibitem{BlekhermanRaymond22}
G.~Blekherman and A.~Raymond.
\newblock A path forward: tropicalization in extremal combinatorics.
\newblock {\em Adv. Math.}, 407:Paper No. 108561, 68, 2022.

\bibitem{BlekhermanRaymond23}
G.~Blekherman and A.~Raymond.
\newblock A new proof of the {E}rd{\H o}s-{S}imonovits conjecture on walks.
\newblock {\em Graphs Combin.}, 39(3):Paper No. 53, 8, 2023.

\bibitem{BucicLongShapiraSudakov21}
M.~Buci\'{c}, E.~Long, A.~Shapira, and B.~Sudakov.
\newblock Tournament quasirandomness from local counting.
\newblock {\em Combinatorica}, 41(2):175--208, 2021.

\bibitem{ChanGrzesikKralNoel20}
T.~F.~N. Chan, A.~Grzesik, D.~Kr\'{a}l', and J.~A. Noel.
\newblock Cycles of length three and four in tournaments.
\newblock {\em J. Combin. Theory Ser. A}, 175:105276, 23, 2020.

\bibitem{ChenClemenNoel25+}
H.~Chen, F.~C. Clemen, and J.~A. Noel.
\newblock Maximizing alternating paths via entropy.
\newblock E-print arXiv:2505.03903v1, 2025.

\bibitem{ChenClemenNoelSharfenberg26+}
H.~Chen, F.~C. Clemen, J.~A. Noel, and A.~Sharfenberg.
\newblock Most oriented paths are not tournament anti-{S}idorenko.
\newblock In preparation, 2026.

\bibitem{ConlonFoxSudakov10}
D.~Conlon, J.~Fox, and B.~Sudakov.
\newblock An approximate version of {S}idorenko's conjecture.
\newblock {\em Geom. Funct. Anal.}, 20(6):1354--1366, 2010.

\bibitem{ConlonLee21}
D.~Conlon and J.~Lee.
\newblock {S}idorenko's conjecture for blow-ups.
\newblock {\em Discrete Anal.}, pages Paper No. 2, 13, 2021.

\bibitem{CoreglianoParenteSato19}
L~N. Coregliano, R~F. Parente, and C~M. Sato.
\newblock On the maximum density of fixed strongly connected subtournaments.
\newblock {\em Electron. J. Combin.}, 26(1):Paper No. 1.44, 48, 2019.

\bibitem{CoreglianoRazborov17}
L.~N. Coregliano and A.~A. Razborov.
\newblock On the density of transitive tournaments.
\newblock {\em J. Graph Theory}, 85(1):12--21, 2017.

\bibitem{FoxHimwichManiZhou24+}
J.~Fox, Z.~Himwich, N.~Mani, and Y.~Zhou.
\newblock Variations on {S}idorenko's conjecture in tournaments.
\newblock To appear in J. Graph Theory. E-print arXiv:2402.08418v1, 2024.

\bibitem{FoxHimwichManiZhou25}
J.~Fox, Z.~Himwich, N.~Mani, and Y.~Zhou.
\newblock A note on directed analogues of the {S}idorenko and forcing
  conjectures.
\newblock {\em Electron. J. Combin.}, 32(3):Paper No. 3.38, 2025.

\bibitem{Griffiths13}
S.~Griffiths.
\newblock Quasi-random oriented graphs.
\newblock {\em J. Graph Theory}, 74(2):198--209, 2013.

\bibitem{Grzesik+23}
A.~Grzesik, D.~Il'kovi\v{c}, B.~Kielak, and D.~Kr\'a\v{l}.
\newblock Quasirandom-forcing orientations of cycles.
\newblock {\em SIAM J. Discrete Math.}, 37(4):2689--2716, 2023.

\bibitem{GrzesikKralLovaszVolec23}
A.~Grzesik, D.~Kr\'{a}l', L.~M. Lov\'{a}sz, and J.~Volec.
\newblock Cycles of a given length in tournaments.
\newblock {\em J. Combin. Theory Ser. B}, 158:117--145, 2023.

\bibitem{Hancock+23}
R.~Hancock, A.~Kabela, D.~Kr\'{a}l', T.~Martins, R.~Parente, F.~Skerman, and
  J.~Volec.
\newblock No additional tournaments are quasirandom-forcing.
\newblock {\em European J. Combin.}, 108:Paper No. 103632, 10, 2023.

\bibitem{Hatami10}
H.~Hatami.
\newblock Graph norms and {S}idorenko's conjecture.
\newblock {\em Israel J. Math.}, 175:125--150, 2010.

\bibitem{HeManiNieTungWei25+}
X.~He, N.~Mani, J.~Nie, N.~Tung, and F.~Wei.
\newblock New {S}idorenko-type inequalities in tournaments.
\newblock E-print arXiv:2512.11222v1, 2025.

\bibitem{KalyanasundaramShapira13}
S.~Kalyanasundaram and A.~Shapira.
\newblock A note on even cycles and quasirandom tournaments.
\newblock {\em J. Graph Theory}, 73(3):260--266, 2013.

\bibitem{KoppartyRossman11}
S.~Kopparty and B.~Rossman.
\newblock The homomorphism domination exponent.
\newblock {\em European J. Combin.}, 32(7):1097--1114, 2011.

\bibitem{KralKucerakLidickyPrivate}
D.~Kr{\'a}{\v{l}}, F.~Ku{\v{c}}er{\'a}k, and B.~Lidick{\'y}.
\newblock Private communication, 2026.

\bibitem{Kral+26+}
D.~Kr\'a\v{l}, M.~Krnc, F.~Ku\v{c}er\'ak, B.~Lidick\'y, and J.~Volec.
\newblock {S}idorenko property and forcing in regular tournaments.
\newblock E-print arXiv:2602.12551v1, 2026.

\bibitem{NoelRanganathanSimbaqueba26}
J.~A. Noel, A.~Ranganathan, and L.~M. Simbaqueba.
\newblock Forcing quasirandomness in a regular tournament.
\newblock {\em Innov. Graph Theory}, 3:127--169, 2026.

\bibitem{SahSawhneyZhao23}
A.~Sah, M.~Sawhney, and Y.~Zhao.
\newblock Paths of given length in tournaments.
\newblock {\em Comb. Theory}, 3(2):Paper No. 5, 6, 2023.

\bibitem{Sidorenko93}
A.~F. Sidorenko.
\newblock A correlation inequality for bipartite graphs.
\newblock {\em Graphs Combin.}, 9(2):201--204, 1993.

\bibitem{ZhaoZhou20}
Y.~Zhao and Y.~Zhou.
\newblock Impartial digraphs.
\newblock {\em Combinatorica}, 40(6):875--896, 2020.

\end{thebibliography}

\appendix

\section{Atlas of Sandwiches}
\label{app:gadget-atlas}

This appendix contains the explicit sandwich certificates used in Sections~\ref{sec:0mod4} and~\ref{sec:spiders}.  All figures follow the convention described in Remark~\ref{rem:convention}.  The values of $\cov$ in the corresponding propositions are obtained by summing these component weights column by column and between consecutive columns.

\begin{figure}[htbp]
\centering

\caption{
The oriented graph $\vec{D}_4^*$ from Construction~\ref{constr:D4*}. All components of $\vec{D}_4^*\setminus V(\vec{P}_4)$ have weight $1/2$. 
}
\label{fig:D4*}
\end{figure}

\begin{figure}[htbp]
\centering
%
\caption{
The oriented graph $\vec{D}_2^*$ from Construction~\ref{constr:Dk*}. The unique component of $\vec{D}_2^*\setminus V(\vec{P}_2)$ has weight $1/2$. 
}
\label{fig:D2*}
\end{figure}

\begin{figure}[htbp]
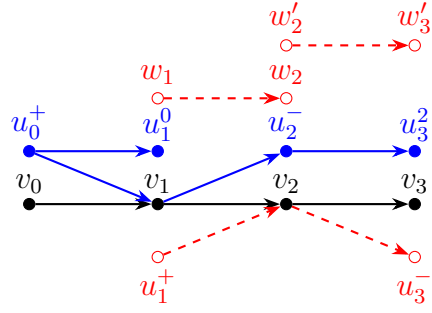

\centering
%
\caption{
The oriented graph $\vec{D}_3^*$ from Construction~\ref{constr:Dk*}. Blue solid components of $\vec{D}_3^*\setminus V(\vec{P}_3)$ have have weight $1/2$, and red dashed components have weight $1/4$.
}
\label{fig:D3*}
\end{figure}

\begin{figure}[htbp]
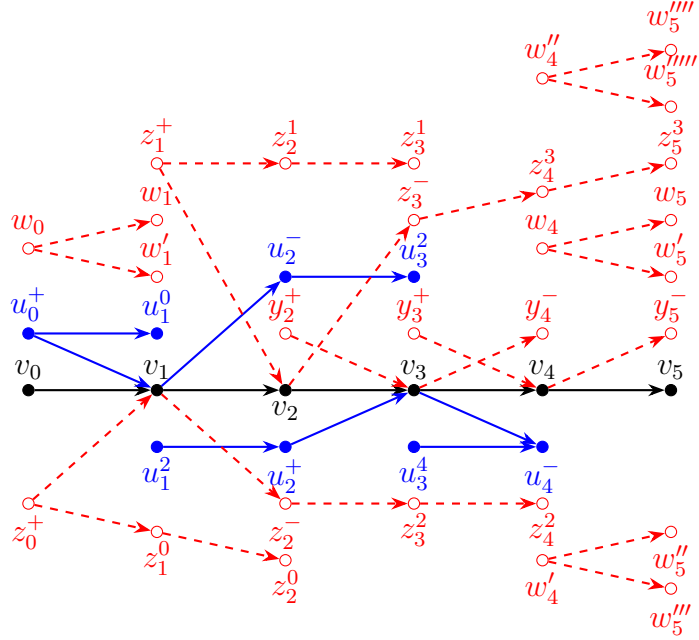

\centering
%
\caption{
The oriented graph $\vec{D}_5^*$ from Construction~\ref{constr:Dk*}. Blue solid components of $\vec{D}_5^*\setminus V(\vec{P}_5)$ have weight $1/4$, and red dashed components have weight $1/8$. There are eight $3$-vertex components of $\vec{D}_5^*\setminus V(\vec{P}_5)$, four are isomorphic to $\cev{P}_{1,1}$ and four are isomorphic to $\vec{P}_2$. The partition of $\mathcal{C}$ groups together each copy of $\cev{P}_{1,1}$ with a copy of $\vec{P}_2$.
}
\label{fig:D5*}
\end{figure}

\begin{figure}[htbp]
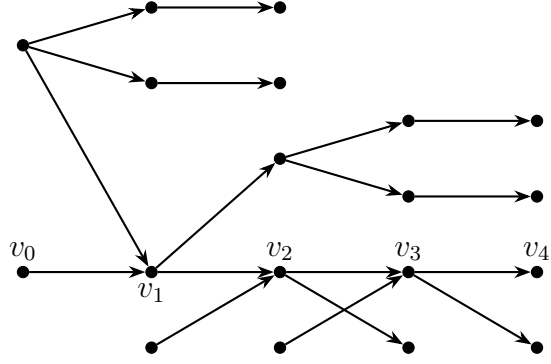

\centering
%
\caption{
The repeating block oriented graph $\vec{D}_4^\dag$ from Construction~\ref{constr:extension}. Each component of $\vec{D}_4^\dag\setminus V(\vec{P}_4)$ has weight $1/4$.
}
\label{fig:extension}
\end{figure}

\begin{figure}[htbp]
\centering
%
\caption{
The oriented graph $\vec{D}_{12}^\ddag$ from Construction~\ref{constr:extended}, shown for $t=3$. Each component of $\vec{D}_{12}^\dag\setminus V(\vec{P}_{12})$ has weight $1/4$.
}
\label{fig:extended}
\end{figure}

\begin{figure}[htbp]
\centering
%
\caption{
The oriented graph $\vec{D}_4^\mid$ from Construction~\ref{constr:extension2}. Each component of $\vec{D}_4^\mid\setminus V(\vec{P}_4)$ has weight $1/5$.}
\label{fig:extension2}
\end{figure}

\begin{figure}[htbp]
\centering
%
\caption{
The oriented graph $\vec{D}_{12}^\parallel$ from Construction~\ref{constr:extended2}, shown for $t=3$. Each component of $\vec{D}_{12}^\mid\setminus V(\vec{P}_{12})$ has weight $1/5$.}
\label{fig:extended2}
\end{figure}

\begin{figure}[htbp]
\centering
%
\caption{
The oriented graph $\vec{D}_8^\S$ from Construction~\ref{constr:0mod4 0mod4}. Each component of $\vec{D}_8^\S\setminus V(\vec{P}_8)$ has weight $1/4$.}
\label{fig:0mod4 0mod4}
\end{figure}

\begin{figure}[htbp]
\centering
%
\caption{
The oriented graph $\vec{D}_{1,3}^\S$ from Construction~\ref{constr:0mod4 2mod4 4}. Each component of $\vec{D}_{1,3}^\S\setminus V(\vec{P}_{1,3})$ has weight $1$.}
\label{fig:0mod4 2mod4 4}
\end{figure}

\begin{figure}[htbp]
\centering
%
\caption{
The oriented graph $\vec{D}_{3,5}^\S$ from Construction~\ref{constr:0mod4 2mod4 8}. Each component of $\vec{D}_{3,5}^\S\setminus V(\vec{P}_{3,5})$ contains a vertex that is labeled with the weight of that component.
}
\label{fig:0mod4 2mod4 8}
\end{figure}

\begin{figure}[htbp]
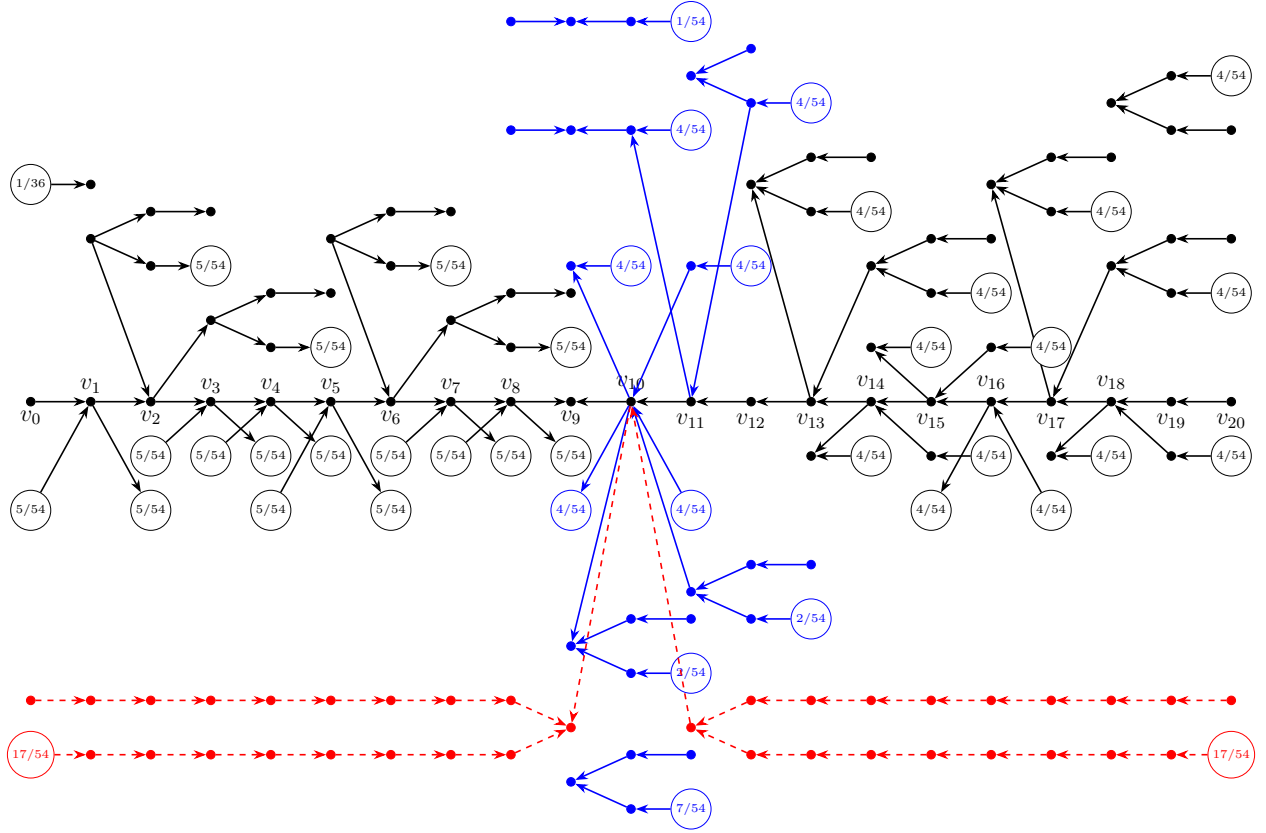

\centering
\rotatebox{360}{%
\begin{minipage}{\textwidth}
\centering
\resizebox{\linewidth}{!}{%
%
%
}
\end{minipage}%
}
\caption{
The oriented graph $\vec{D}_{a-1,a+1}^\S$ from Construction~\ref{constr:0mod4 2mod4 general}, shown for $a=10$. Each component of $\vec{D}_{9,11}^\S\setminus V(\vec{P}_{9,11})$ contains a vertex that is labeled with the weight of that component. The blue vertices and arcs in the central part are contained in $\vec{D}_{a-1,a+1}$ for all $a\geq6$. The red vertices and arcs are the copies of $\vec{P}_{a-1,a-1}$ and the arcs that attach them to $\vec{D}_{a-1,a+1}$.}
\label{fig:0mod4 2mod4 general}
\end{figure}

\label{app:spider-certificates}
\begin{figure}[htbp]
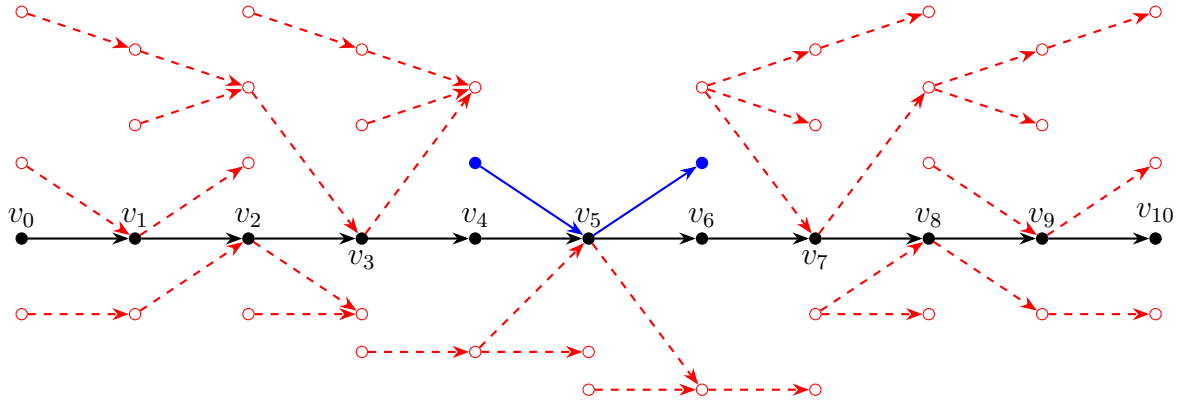

\centering

\caption{
The oriented graph $\vec{D}_{10}^\S$ from Construction~\ref{constr:2mod4 1mod4}. Each blue solid component of $\vec{D}_{10}^\S\setminus V(\vec{P}_{10})$  has weight $1/2$ and each red dashed component has weight $1/4$.}
\label{fig:2mod4 1mod4}
\end{figure}

\begin{figure}[htbp]
\centering
%
\caption{
The oriented graph $\vec{D}_{2,4}^\S$ from Construction~\ref{constr:2mod4 3mod4 6}. Each component of $\vec{D}_{2,4}^\S\setminus V(\vec{P}_{2,4})$ has weight $1/4$.
}
\label{fig:2mod4 3mod4 6}
\end{figure}

\begin{figure}[htbp]
\centering
%
\caption{
The oriented graph $\vec{D}_{4,4,6}^\S$ from Construction~\ref{constr:2mod4 3mod4 14}. Each component of $\vec{D}_{4,4,6}^\S\setminus V(\vec{P}_{4,4,6})$ contains a vertex that is labeled with the weight of that component.
}
\label{fig:2mod4 3mod4 14}
\end{figure}

\begin{figure}[htbp]
\centering
%
\caption{
The oriented graph $\vec{D}_{3,4}^\S$ from Construction~\ref{constr:3mod4 0mod4 7}. Each component of $\vec{D}_{3,4}^\S\setminus V(\vec{P}_{3,4})$ contains a vertex that is labeled with the weight of that component.
}
\label{fig:3mod4 0mod4 7}
\end{figure}

\begin{figure}[htbp]
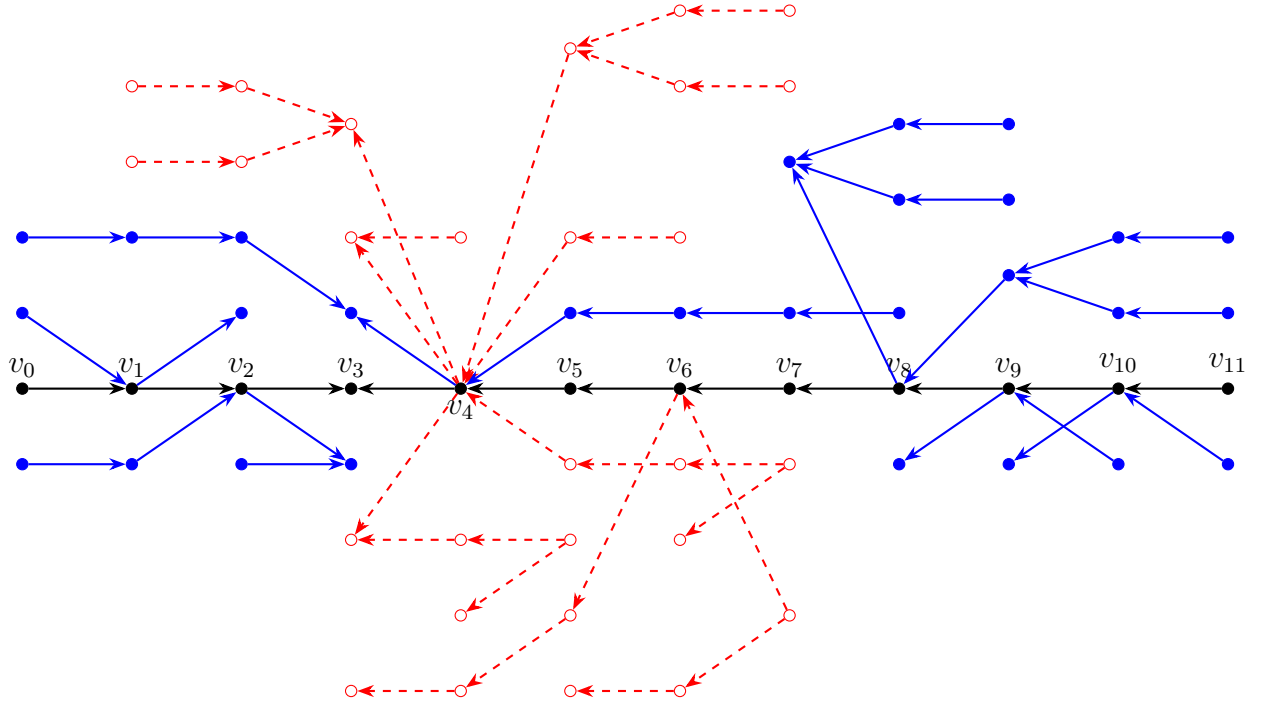

\centering
%
\caption{
The oriented graph $\vec{D}_{3,8}^\S$ from Construction~\ref{constr:3mod4 0mod4 11}. Red dashed components have weight $1/8$, and blue solid components have weight $1/4$.
}
\label{fig:3mod4 0mod4 11}
\end{figure}

\begin{figure}[htbp]
\centering
%
\caption{
The oriented graph $\vec{D}_{4,3}^\S$ from Construction~\ref{constr:3mod4 1mod4}. Each component of $\vec{D}_{4,3}^\S\setminus V(\vec{P}_{4,3})$ has weight $1/4$.
}
\label{fig:3mod4 1mod4}
\end{figure}

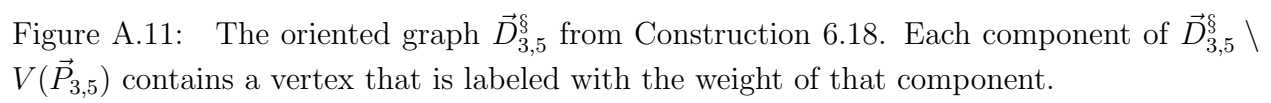
\end{document}